\documentclass[10pt,twoside, a4paper, english, reqno]{amsart}
\usepackage[dvips]{epsfig}
\usepackage{amscd}
\usepackage{amssymb}
\usepackage{amsthm}
\usepackage{amsmath}
\usepackage{latexsym}

\usepackage{upref}
\usepackage{hyperref}
\usepackage{color}
\usepackage[usenames,dvipsnames]{xcolor}

\theoremstyle{plain}
\newtheorem{thm}{Theorem}[section]
\theoremstyle{plain}
\newtheorem{lem}[thm]{Lemma}
\newtheorem{prop}[thm]{Proposition}

\theoremstyle{definition}
\newtheorem{defi}{Definition}[section]
\newtheorem{con}{Condition}[section]
\newtheorem{rem}{Remark}[section]

\newenvironment{Assumptions}
{
\setcounter{enumi}{0}

\begin{enumerate}}
{\end{enumerate} }
 
\newenvironment{Assumptions2}
{
\setcounter{enumi}{0}

\begin{enumerate}}
{\end{enumerate} }

\newenvironment{Assumptions3}
{
\setcounter{enumi}{0}

\begin{enumerate}}
{\end{enumerate} }

\newenvironment{Assumptions4}
{
\setcounter{enumi}{0}

\begin{enumerate}}
{\end{enumerate} }

\newcommand{\R}{\ensuremath{\mathbb{R}}}

\newcommand{\goto}{\ensuremath{\rightarrow}}

\newcommand{\eps}{\ensuremath{\varepsilon}}

\numberwithin{equation}{section} \allowdisplaybreaks

\title[Well-posedness and LDP for SHE with logarithmic nonlinearity]
{L\'{e}vy driven stochastic heat equation with logarithmic nonlinearity: Well-posedness and Large deviation principle}

\date{}

\subjclass[2000]{45K05, 46S50, 49L20, 49L25, 91A23, 93E20}

\keywords{ Nonlinear stochastic PDE, Logarithmic nonlinearity, Strong and martingale solution, Jakubowski’s version of Skorokhod theorem.}

\author[ R. Kavin]{R. Kavin}
\address[R. Kavin] {\newline 
Department of Mathematics,
Indian Institute of Technology Delhi,
Hauz Khas, New Delhi, 110016, India.}
\email[] {maz198757@iitd.ac.in}

\author[A. K. Majee]{Ananta K. Majee}
\address[A. K. Majee]{\newline
Department of Mathematics, Indian Institute of Technology Delhi,
Hauz Khas, New Delhi-110016, India.}
\email[]{majee@maths.iitd.ac.in}

\thanks{}

\begin{document}
\begin{abstract}
In this article, we study the well-posedness theory for solutions of the stochastic heat equations with logarithmic nonlinearity perturbed by multiplicative L\'{e}vy noise. By using Aldous tightness criteria and Jakubowski’s version of the Skorokhod theorem on non-metric spaces along with the standard $L^2$-method, we establish the existence of a path-wise unique strong solution. Moreover, by using a weak convergence method, we establish a large deviation 
principle for the strong solution of the underlying problem. Due to the lack of linear growth and locally Lipschitzness of the term $ u \log(|u|)$ present in the underlying problem, the logarithmic Sobolev inequality and the nonlinear versions of Gronwall’s inequalities play a crucial role. 
\end{abstract}

\maketitle

\section{Introduction}
Let $D\subset \R^d$ be a bounded domain with Lipschitz boundary $\partial D$,  
$N({\rm d}z,{\rm d}t)$ be a time-homogeneous Poisson random measure \cite[pp. 631]{Peszat-2007,Erika-2009} on $({\pmb E}, \mathcal{B}({\pmb E}))$ with intensity measure $m({\rm d}z)$, defined on the given filtered probability space
$(\Omega,  \mathcal{F},\mathbb{P}, \mathbb{F}=\{ \mathcal{F}_t\}_{t \geq 0})$ satisfying the usual hypothesis, where $({\pmb E},\mathcal{B}({\pmb E}),m)$ is a $\sigma$-finite measure space. 
We are interested in the well-posedness theory and theory of large deviation principle of a strong solution for the nonlinear stochastic problem perturbed by L\'{e}vy noise:
 \begin{equation}\label{eq:log-nonlinear}
 \begin{aligned}
  {\rm d} u(t,x) - \Delta u(t,x) \,{\rm d}t &=  u(t,x) \log |u(t,x)|\,{\rm d}t + \int_{{\pmb E}}\eta(u(t,x);z)\widetilde{N}({\rm d}z,{\rm d}t),\quad t>0,~x\in D\,, \\
  u(t,x)&=0,  \quad t>0,~x\in \partial D\,, \\
  u(0,x)&=u_0(x),\quad x\in D\,.
  \end{aligned}
 \end{equation}
  In \eqref{eq:log-nonlinear}, $\eta: \R \times \pmb{E} \mapsto \R$
 is a given noise coefficient signifying the multiplicative nature of the noise, and $$ \widetilde{N}({\rm d}z,{\rm d}t):= N({\rm d}z,{\rm d}t)-\, m({\rm d}z)\,{\rm d}t,$$ 
 the time-homogeneous compensated Poisson random measure.
 \vspace{0.1cm}

 In the last decades, among researchers, there has been growing interest in studying the stochastic partial differential equations (in short, SPDEs) with monotone or locally monotone coefficients/non-linearity. We recommend readers see the monographs \cite{Zabczyk-2014} and \cite{Michael-2007} for a thorough understanding of SPDEs in general. However, coefficients with logarithmic non-linearity do not fit into this category. The exploration of logarithmic non-linearity has been addressed in the study of relativistic fields in physics and nonlinear wave mechanics, as discussed in references \cite{Rosen-1969,Bialynicki-1976}. Logarithmic nonlinear partial differential equations have applications across various fields of science and engineering, such as physics, mathematical biology, material science, geophysics, nonlinear photonics, and fluid dynamics \cite{Edelstein-2005,Zloshchastiev-2019,LeiWang-2012}. Equation \eqref{eq:log-nonlinear} could be viewed as a stochastic perturbation of heat equation with logarithmic non-linearity perturbed by  L\'{e}vy noise. Due to technical novelties and the wide range of applications in physical contexts, the study of the well-posedness results for equation \eqref{eq:log-nonlinear} is more subtle. There is a requirement to analyze and quantify the probability of rare events across various fields, enabling an understanding of the limiting behaviour of specific probability models. Furthermore, the following natural question arises: what is the asymptotic relationship between the solution of the underlying problem \eqref{eq:log-nonlinear} and the corresponding deterministic equation when the stochastic perturbation is significantly small?  In other words, one needs to study the small noise large deviation principle (LDP in short) for the solution of \eqref{eq:log-nonlinear}.
\vspace{0.1cm}

In the absence of noise term, \eqref{eq:log-nonlinear} becomes a deterministic heat equation involving logarithmic non-linearity.  A number of authors have contributed since then, and we mention a few of the works e.g, \cite{Chen-2015,Chen-2015-JD,Alfaro-2017,Yang-2016} and references therein. The deterministic heat equation with logarithmic non-linearity on a bounded domain was investigated by Chen et al. in \cite{Chen-2015}; see also \cite{Chen-2015-JD}. They established the existence of global solutions and blow-up at infinity under some appropriate conditions. However, their study did not establish the uniqueness of the solutions. In \cite{Alfaro-2017}, the authors studied the problem on a whole real line and established the well-posedness of solutions in the case of a bounded and smooth initial condition.
\vspace{0.1cm}

We mention the paper \cite{Dalang-2019}, where the authors investigated stochastic reaction diffusion equations with logarithmic non-linearity driven by space-time white noise on the bounded domain $[0,1]$. They have shown non-blow up $L^2$-valued solution but did not provide the well-posedness theory for it. Very recently, Shang and Zhang in \cite{Shang-2022} have revealed the global existence and uniqueness of solutions of the stochastic heat equation \eqref{eq:log-nonlinear} in case of multiplicative Brownian noise. The authors have constructed a sequence of approximate solutions (via Galerkin methods), and shown the tightness of the approximate solutions on some appropriate space. Using the theorems of Skorokhod and Prokhorov along with the pathwise uniqueness of solutions and the Yamada Watanabe theorem, they have shown the well-posedness of a global strong solution of the underlying problem in the case of sub-linear and super-linear growth of the noise coefficient. In the case of sub-linear growth, they are able to derive the global moment estimate of the solution. 
\vspace{0.1cm}

\vspace{0.1cm}

To the best of our knowledge, there are currently no results available regarding the global solution aspect in this context. We mention the works of \cite{Yuchao-2018, Fubao-2019} where the SDEs with logarithic non-linearity was studied. Taking primary motivation from \cite{Yuchao-2018}, in the first part of this article, we aim to generalize the well-posedness result of \cite{Shang-2022} in the case of jump-diffusion noise.

\vspace{0.1cm}

The concept of large deviation theory, which concerns the asymptotic behaviour of remote tails of some sequence of probability distribution, is a popular and important topic in probability and statistics; see \cite{Stroock-1984,Varadhan-1984,Varadhan-1996,Stroock-1989}. In the last three decades, there have been numerous results available over LDP for SPDEs, and most of them are limited to SPDEs driven by Gaussian noise; see for example, \cite{Budhiraja-2008, Rockner-2004, Gess-2023, Goldys-2017,Matoussi-2021, Liu-2010, Wang-2012, Rockner-2012, Hong-2021} and references therein to name a few. Recently, in \cite{Kavin-2024,Kavin-2023LDP} the authors achieved LDP for stochastic evolutionary $p$-Laplace equation and pseudo-monotone evolutionary equation, respectively. Moreover, because of the nonlinear perturbation flux function, the authors could not use the results of monotone or locally monotone SPDEs to establish LDP. However, coefficients with logarithmic nonlinearity do not fit into this category. Very recently, using nonlinear versions of Gronwall's inequalities and log-Sobolev inequalities along with the weak convergence approach, the authors in \cite{Pan-2022} established a large deviation principle for the solutions of equation \eqref{eq:log-nonlinear}
with Brownian noise--- which is neither locally Lipschitz nor locally monotone.

There has not been very much works about the LDP for SPDEs driven by L{\'e}vy noise. In \cite{Budhiraja-2013}, the authors utilized the results of Budhiraja and Dupuis \cite{Budhiraja-2000,Budhiraja-2011} to establish the LDP for SPDEs with Poisson noises. Their method involved establishing the Laplace principle in Polish spaces using the weak convergence approach. The study of LDP was carried out for multiplicative L{\'e}vy noise for fully nonlinear stochastic equation in \cite{Yang-2015}, for $2$-D stochastic Navier-Stokes equations driven by multiplicative L{\'e}vy noise in \cite{Zhai-2015, Brzezniak-2023},  for $2$-D primitive equations in \cite{Jin-2021}, for stochastic generalized Ginzburg-Landau equation in \cite{Ran-2023}, for stochastic generalized porous media equations driven by {L}\'evy noise in \cite{Wu-2024}. Although SPDEs with monotone or locally monotone coefficients have been extensively studied (cf.\cite{Xiong-2018, Li-2020,Hong-2021} etc., and references therein), the systems characterized by logarithmic nonlinearity remain relatively less explored in the literature. This is primarily due to the lack of well-posedness results previously established for such systems. In the second part of this article, we propose to study Freidlin-Wentzell's large deviation principle for \eqref{eq:log-nonlinear} along with the weak convergence approach and the improved sufficient criterion proposed by \cite{Zhang-2023}.
 
\subsection{Aim and outline of the paper}
The aim of this article is bi-fold. In the first part, we analyze the well-posedness theory for the SPDEs \eqref{eq:log-nonlinear} driven by L\'{e}vy noise. The second part establishes Freidlin-Wentzell's large deviation principle for the strong solution of \eqref{eq:log-LDP}. More precisely, we proceed as follows.
\begin{itemize}
    \item [{\rm (i)}] Firstly, we prove the well-posedness of strong solutions of \eqref{eq:log-nonlinear}. To do so, we first construct an approximate solution $\{u_n(t) \}$ (cf.~\eqref{eq:finite-dim-def}) of the finite dimensional stochastic differential equations \eqref{eq:finite-dim} via Galerkin method, and then using logarithmic Sobolev inequality and the nonlinear versions of Gronwall's inequality, we derive its necessary {\it a-priori} estimates---which is essential to prove the tightness of $\big(\mathcal{L}(u_n)\big)_{n \in \mathbb{N}}$ in some appropriate space via Aldous condition~(see Definition \ref{defi:aldous-condition}). Then, we use Jakubowski's version of the Skorokhod theorem on a non-metric space to generate a new probability space $(\bar{\Omega}, \bar{\mathcal{F}}, \bar{\mathbb{P}})$, and a sequence of random variables $u_{n}^{*}, u_*$ in $(\bar{\Omega}, \bar{\mathcal{F}}, \bar{\mathbb{P}})$ such that with probability $1$, $u_{n}^{*} \rightarrow u_*$
    which leads to existence of weak solution $\bar{\pi}:=\big( \bar{\Omega}, \bar{\mathcal{F}}, \bar{\mathbb{P}}, \bar{\mathbb{F}}, N_*, u_* \big)$ of \eqref{eq:log-nonlinear} on the time interval $[0,T_2]$ ~(see Lemma \ref{lem:est-galerkin} for the definition of $T_2$). Moreover, using the moment estimates, we construct a weak solution on the time interval $[0,T]$ for any given $t>0$. We utilize the standard $L^2$-contraction principle to determine the path-wise uniqueness of weak solutions. Later, we use the Yamada-Watanabe theorem to assure the existence of a unique strong solution of \eqref{eq:log-nonlinear}. 
 \item [{\rm (ii)}] Secondly, we establish the small perturbation LDP for the strong solution of \eqref{eq:log-nonlinear} on the solution space $\mathcal{E}_T = \mathbb{D}([0,T];L^{2}(D))\cap L^{2}([0,T];W_{0}^{1,2}(D))$ via the weak convergence method and the variational representation of the Poisson random measure.  
 \begin{itemize}
 \item [a)] We first prove the well-posedness result of the corresponding skeleton equation~(cf.~\eqref{eq:log-skeleton}) based on the technique of Galerkin method, nonlinear versions of Gronwall's lemma, the log-Sobolev inequality and the adaptation of a series of technical lemmas. We construct a sequence of approximate solutions and avail of {\it a-priori} estimates together with the necessary compact embedding to prove the existence of a solution for the skeleton equation. Uniqueness is achieved via the standard $L^2$--contraction principle. 
 \item[b)] Based on the improved sufficient conditions \cite{Zhang-2023} and a certain generalization of Girsanov-type theorem for Poisson random measures together with the weak convergence approach \cite{Budhiraja-2011} and the {\it a-priori} estimates for the skeleton equation and the $\tilde{u}_\epsilon$, the solution of equation \eqref{eq:epsilon-main}, we first establish the LDP for \eqref{eq:log-LDP} on the space $\mathcal{E}_{T^*}$ for some $T^*>0$ and then extend it to the whole time interval $[0,T]$ by induction. 
 \end{itemize}
\end{itemize}

\vspace{0.2cm}

The remaining part of the paper is organized as follows. We state the technical assumptions, define the notion of solutions
for the problems \eqref{eq:log-nonlinear}, and state the main results of the paper in Section \ref{sec:technical-framework}. Section \ref{sec:existence-weak-solu} is devoted to establish the well-posedness theory for the strong solution of \eqref{eq:log-nonlinear}. In Section \ref{sec:log-skeleton}, we show the well-posedness result for the skeleton equation. The final Section \ref{sec:log-LDP} is devoted to proving the LDP for the strong solution of \eqref{eq:log-LDP}.

 \section{Technical framework and statement of the main results}\label{sec:technical-framework}
 We refer to numerous generic constants in every section of this article by the letters C, K, etc.
 Here and throughout, we represent the standard spaces
 of $p^{th}$ order integrable functions on $D$ and  the standard Sobolev spaces on $D$  by $(L^p(D), \|\cdot\|_{L^p(D)})$ and $(W_0^{1,p}(D), \|\cdot\|_{W_0^{1,p}(D)})$ respectively for $p\in \mathbb{N}$. By $W^{-1,p^\prime}(D)$, we denote the dual space of $W_0^{1,p}(D)$ with duality pairing $\big\langle \cdot,\cdot\big\rangle$, where $p^\prime$ is the convex conjugate of $p$. For any $x,y \in \R$, we denote $x\wedge y:=\min\{x,y\}$ and 
 $x\vee y:=\max\{x,y\}$. Moreover, for any $z\in \R$, we set $\log_{+}z:=\log(1 \vee z)$.  
For any $0\le a<b < +\infty$, define the space 
$$ \mathcal{E}_{a,b}:= \mathbb{D}([a,b];L^2(D))\cap L^2([a,b];W_0^{1,2}(D)).$$
It is standard that the space $\mathcal{E}_{a,b}$ endowed with the norm $\|\cdot\|_{\mathcal{E}_{a,b}}$ defined by 
\begin{align*}
\|y\|_{\mathcal{E}_{a,b}}:= \sup_{s\in [a,b]}\|y(s)\|_{L^2(D)} + \Big( \int_a^b \|y(s)\|_{W_0^{1,2}(D)}^2\,{\rm d}s\Big)^\frac{1}{2}
\end{align*}
is a Banach space. Consider the metric $\rho_{a,b}$ on $\mathcal{E}_{a,b}$ as
\begin{align}
\rho_{a,b}^2(u,v):=  \sup_{s\in [a,b]}\|u(s)-v(s)\|_{L^2(D)}^2 + \int_a^b \|u(s)-v(s)\|_{W_0^{1,2}(D)}^2\,{\rm d}s\,,~~~u,v\in  \mathcal{E}_{a,b}.\label{defi:metric}
\end{align}
For any $b>0$, we write $\rho_b=\rho_{0,b}$ and $\mathcal{E}_b=\mathcal{E}_{0,b}$.
 
\subsection{Notion of solutions and main theorem}
In the theory of stochastic evolution equations, two types of solution concepts are considered
namely strong solution and weak solution. We first provide the notion of a strong solution for \eqref{eq:log-nonlinear}.
 \begin{defi}[Strong solution] \label{defi:strong-solun}
 Let  $(\Omega, \mathcal{F}, \mathbb{P}, \{\mathcal{F}_t\}_{t\ge 0} )$ be a given complete stochastic basis, $N$ be a time-homogeneous Poisson random measure on $\pmb{E}$ with intensity measure $m({\rm d}z)$ defined on $(\Omega, \mathcal{F}, \mathbb{P}, \{\mathcal{F}_t\}_{t\ge 0} )$.  We say that a predictable process $u:\Omega \times[0,\infty)\goto L^2(D)$ is a global strong solution of  \eqref{eq:log-nonlinear}, if and only if for any $T>0$, $ u\in L^2(\Omega; \mathbb{D}([0,T]; L^2(D)))\cap L^2(\Omega;L^2([0,T]; W_0^{1,2}(D))) \cap L^2(\Omega; L^{\infty}([0,T];L^{2}(D)))$ such that 
 \begin{itemize}
 \item[i)] $u(0,\cdot)=u_0$ in $L^2(D)$, 
\item[ii)] $\mathbb{P}$-a.s., and for any $t\ge 0$, 
   \begin{align}
   u(t)& = u_0 +  \int_0^t \Delta u(s) \,{\rm d}s + \int_{0}^t  u(s) \log|u(s)| \,{\rm d}s +  \int_0^t \int_{|z| > 0}\eta(u(s);z)\widetilde{N}({\rm d}z,{\rm d}s)  \quad \text{in}~~ W^{-1,2}(D)\,. \label{eq:defi-weak-form}
  \end{align}
   \end{itemize}
  \end{defi}
In situations where the drift and diffusion operators lack Lipschitz continuity, establishing the existence of a strong solution might not be feasible, and therefore one needs to consider the concept of a weak solution.
  \begin{defi}[Weak solution]
  We say that a $6$-tuple $\bar{\pi}=\big( \bar{\Omega}, \bar{\mathcal{F}}, \bar{\mathbb{P}}, \{\bar{\mathcal{F}}_t\}, \bar{N}, \bar{u}\big)$ is a weak solution of \eqref{eq:log-nonlinear}, if
 \begin{itemize}
  \item [(i)] $(\bar{\Omega}, \bar{\mathcal{F}},\bar{\mathbb{P}}, \{\bar{\mathcal{F}}_t\}_{t\ge 0} )$ is a complete stochastic basis,
  \item[ii)] $\bar{N}$ is a time-homogeneous Poisson random measure on $\pmb{E}$ with intensity measure $m({\rm d}z)$ defined on  $(\Bar{\Omega}, \Bar{\mathcal{F}}, \Bar{\mathbb{P}}, \{\Bar{\mathcal{F}}_t\}_{t\ge 0} )$,
   \item[(iii)] $\bar{u}$ is a $L^2(D)$-valued $\{\bar{\mathcal{F}}_t\}$-predictable process such that for any $T>0$
   \begin{itemize}
\item[a)] $\bar{u}(0)=u_0$ in $L^2(D)$,
\item[b)]  $ \bar{u}\in L^2(\bar{\Omega}; \mathbb{D}([0,T];L^2(D)))\cap L^2(\bar{\Omega};L^2([0,T]; W_0^{1,2}(D))) \cap L^2(\bar{\Omega}; L^{\infty}([0,T];L^{2}(D)))$,
\item[c)] $\bar{\mathbb{P}}$-a.s., and for any $t\ge 0$,  \eqref{eq:defi-weak-form} holds.
   \end{itemize}
 \end{itemize}
 \end{defi}
We will study the well-posedness theory for problem \eqref{eq:log-nonlinear} under the following assumptions.
 \begin{Assumptions}
\item \label{A1} Initial data $u_0$ is deterministic and $u_0 \in L^2(D)$. 
\item \label{A2} There exist non-negative constants $K_1, K_2 $ and a non-negative function ${\tt h}\in L^\infty(\pmb{E}, m)\cap L^1(\pmb{E},m)$ such that for all $ u, v \in \R$ and $z \in \pmb{E}$
\begin{align*}
    | \eta(u;z)-\eta(v;z)| \le \Big\{ K_1 |u-v| + K_2 |u-v| \big(\log _{+} (|u| \vee |v|) \big)^{\frac{1}{2}} \Big\} {\tt h}(z) .
\end{align*}

\item  \label{A3} $\eta$ satisfies the following sub-linear growth condition: there exist non-negative constants $M_1, M_2$ and 
$\theta \in [0,1)$ such that 
\begin{align*}
    | \eta(u;z)| \le \Big\{ M_1 + M_2 |u|^{\theta} \Big\} {\tt h}(z) \quad \text{ for all $ u\in \R$ and $z \in \pmb{E}$}.
\end{align*} 


\end{Assumptions}



Now, we state one of the main results of our article.

 \begin{thm}[Well-posedness of strong solution]\label{thm:existence-strong}
 Let the assumptions \ref{A1}-\ref{A3} hold true. Then there exists a unique global strong solution of \eqref{eq:log-nonlinear} in the sense of Definition \ref{defi:strong-solun}. Moreover, for any $T>0$ and $p\ge 2$, there exists a positive constant $C=C(p,\theta,\|u_0\|_{L^2(D)},T)$ such that
\begin{equation}\label{esti:bound-weak-solun}
 \mathbb{E}\bigg[\sup_{0\le t\le T} \|u(t)\|_{L^2(D)}^p + \int_0^T \|u(t)\|_{L^2(D)}^{p-2}\| u(t)\|_{W_0^{1,2}(D)}^2\,{\rm d}t \bigg] \le C(p,\theta,\|u_0\|_{L^2(D)},T)\,.
\end{equation}
 \end{thm}
 \subsection{Large deviation principle}
 We recall some fundamental definitions and results of the Freidlin–Wentzell-type large deviation theory. Let $\{Y^\epsilon\}_{\epsilon > 0}$  be a sequence of random variables defined on a given probability space $(\Omega, \mathcal{F}, \mathbb{P})$ taking values in some Polish space $\mathbb{O}$. It is well-known that the large deviation principle characterizes the exponential decay
of the remote tails of some sequences of probability distributions. The rate of such decay is described by
the “rate function”.
\begin{defi}[Rate Function]
  A function ${\tt I} : \mathbb{O} \rightarrow [0,\infty]$ is called a rate function on $\mathbb{O}$ if it is lower semi-continuous. If the level set $\{x \in \mathbb{O} : {\tt I}(x) \leq M \}$ is compact for each  $M < \infty$, then ${\tt I}$ is called a good rate function.
\end{defi}
\begin{defi} [Large deviation principle]
The sequence  $\{Y^\epsilon\}_{\epsilon > 0}$  is said to satisfy the large deviation principle on  $\mathbb{O}$ with rate function ${\tt I}$ if for any Borel subset $O$ of $\mathbb{O}$ 
\begin{equation*}
    \begin{aligned}
        - \underset{x\in O^0}{\inf}  {\tt I}(x) \leq \underset{\epsilon \longrightarrow 0}{\liminf}\, \epsilon^{2}\log\big(\mathbb{P}(Y^\epsilon \in O)\big) & \leq    \underset{\epsilon \longrightarrow 0}{\limsup}\,\epsilon^{2} \log \big( \mathbb{P}(Y^\epsilon \in O)\big)
        \leq - \underset{x\in \bar{O}}{\inf} {\tt I}(x),
    \end{aligned}
\end{equation*}
where $O^0$ and $\bar{O}$ are the interior and closure of $O$ in $\mathbb{O}$ respectively. 
\end{defi}
 Let $\mathcal{O}$ be a locally compact Polish space and $\mathcal{M}(\mathcal{O})$ be the space of all measures ${\tt m}$ on $(\mathcal{O}, \mathcal{B}(\mathcal{O}))$ such that ${\tt m}(K)<\infty$ for every compact $K$ in $\mathcal{O}$. Let $\mathcal{T}(\mathcal{M}(\mathcal{O}))$ be the weakest topology on $\mathcal{M}(\mathcal{O})$ such that for every  $f \in C_c(\mathcal{O})$
 $$\mathcal{M}(\mathcal{O}) \ni {\tt m}\mapsto \int_{\mathcal{O}} f(u) \,{\tt m}(du)\in \R$$
 is continuous, where $C_c(\mathcal{O})$ denotes the space of continuous functions with compact support. It is well-known that $(\mathcal{M}(\mathcal{O}), \mathcal{T}(\mathcal{M}(\mathcal{O})))$ is a Polish space; cf.~\cite{Budhiraja-2011}.
 \vspace{0.1cm}
 
 Fix a $\sigma$-finite measure $m$ on 
 $(\pmb{E}, \mathcal{B}(\pmb{E}))$ as mentioned in the introduction. Set ${\tt M}_T:=\mathcal{M}([0,T]\times \pmb{E})$ and $\mathbb{M}_T:=\mathcal{M}([0,T]\times \pmb{E} \times [0,\infty))$ for fixed $T>0$. Both $({\tt M}_T, \mathcal{T}({\tt M}_T))$ and $(\mathbb{M}_T, \mathcal{T}(\mathbb{M}_T))$ are Polish spaces. In this part, we specify the probability space 
 $$ \Omega:= \mathbb{M}_T,\quad \mathcal{F}:=\mathcal{T}(\mathbb{M}_T).$$
 For the given $m\in \mathcal{M}(\pmb{E})$, by \cite[Section 1.8]{Watanabe-1981},  there exists a unique probability measure $\mathbb{P}$ on $(\Omega, \mathcal{F})$ such that the canonical map
 $$ N: \Omega \ni \nu\mapsto \nu\in \mathbb{M}_T$$
 is a Poisson random measure~(PRM) on $[0,T]\times\pmb{E} \times [0,\infty)$ with intensity measure $\lambda_T\otimes m \otimes \lambda_\infty$, where $\lambda_T$ and $\lambda_\infty$ are Lebesgue measure on $[0,T]$ and $[0,\infty)$ respectively. For each $t\in [0,T]$, define the $\sigma$-algebra $\mathcal{G}_t$ as
 $$\mathcal{G}_t:= \sigma\{N((0, s] \times A): 0 \leq s \leq t, A \in \mathcal{B}(\pmb{E} \times [0,\infty))\}.$$
 We take $\mathbb{F}:=\{\mathcal{F}_t\}_{t\in [0,T]}$ as the $\mathbb{P}$-completion of $\{\mathcal{G}_t\}_{t\in [0,T]}$. By $\mathcal{P}_T$, we denote the $\mathbb{F}$-predictable $\sigma$-field on $[0,T]\times \Omega$. The corresponding compensated Poisson random measure is denoted by $\widetilde{N}$.  Denote 
 $$ \mathcal{A}:=\Big\{ \varphi: [0,T]\times\pmb{E}\times  \Omega\rightarrow  [0,\infty): ~\text{ $\varphi$ is $(\mathcal{P}_T\otimes \mathcal{B}(\pmb{E}))/\mathcal{B}([0,\infty)])$ measurable}\Big\}.$$
 For $\varphi \in \mathcal{A}$, define a counting process $N^{\varphi}$ on $[0,T]\times \pmb{E}$ by
 \begin{align}
 N^{\varphi}((0, t] \times A):=\int_{(0, t] \times A} \int_{(0, \infty)} 1_{[0, \varphi(s, z)]}(r)\, N({\rm d}s, {\rm d} z, {\rm d} r), \quad t \in[0, T],~ A \in \mathcal{B}(\pmb{E})\,. \label{eq:controlled-random-measure}
 \end{align}
 $N^\varphi$ can be viewed as a controlled random measure, with $\varphi$ selecting the intensity. Analogously, $\widetilde{N}^\varphi$ is defined by replacing $N$ with $\widetilde{N}$ in \eqref{eq:controlled-random-measure}. When $\varphi \equiv {\tt a} \in(0, \infty)$, we write $N^{\varphi}=N^{\tt a}$. In particular, for any ${\tt a}>0$, $N^{\tt a}$ is a Poisson random measure on $[0,T]\times \pmb{E}$ with intensity measure $\lambda_T({\rm d}t)\otimes {\tt a} m({\rm d}z)$, and $\widetilde{N}^{\tt a}$ is
equal to the corresponding compensated Poisson random measure; cf.~\cite[Proposition 5.1]{Brzezniak-2023}. 
\vspace{0.1cm}

Let us define 
\begin{align*}
\mathcal{H}_2:=\Big\{ {\tt h}: \pmb{E}\rightarrow \R^+:~\text{${\tt h}$ is Borel measurable and there exists $\delta \in (0,\infty)$ such that} \\
\int_{E} \exp\{\delta\, {\tt h}^2(z)\}\,m({\rm d}z) < +\infty ~\text{for all $E\in \mathcal{B}(\pmb{E})$ with $m(E)< +\infty$}\Big\}\,.
\end{align*}
To study the large deviation principle, we consider the following assumptions on the diffusion coefficient $\eta(u;z)$.
\begin{Assumptions2} 
\item \label{B1} There exist ${\tt h}_1,~{\tt h}_2 \in \mathcal{H}_2\cap L^1(\pmb{E},m)\cap L^2(\pmb{E},m)$ with ${\tt h}_2 \le 1$ such that
\begin{align}
& \|\eta (u;z)-\eta(v;z)\|_{L^2(D)} \le {\tt h}_1(z)\|u-v\|_{L^2(D)}~~ \text{for all $ u,v~\in L^2(D)$ and $z \in \pmb{E}$}\,, \label{cond:lipschitz-ldp} \\
&\|\eta(u,z)\|_{L^2(D)} \le {\tt h}_2(z)\big( 1 + \|u\|_{L^2(D)}^\theta\big)~~\text{for all $ u~\in L^2(D)$ and $z \in \pmb{E}$}\,. \label{cond:linear-growth-ldp}
\end{align}
\end{Assumptions2}
\begin{rem}\label{rem:assumption}
Thanks to \cite[Remark $3.1$]{Wu-2024-1}, one has for  any $\beta\ge 0$,
 $$L^2(\pmb{E},m)\cap \mathcal{H}_2 \subset L^{\beta + 2}(\pmb{E}, m)\cap \mathcal{H}_2.$$
\end{rem}

\begin{rem}\label{rem:regarding-condition}
If ${\tt h}\in \mathcal{H}_2$, then for every $\delta \in (0,\infty)$ and all $E\in \mathcal{B}(\pmb{E})$ with $m(E)< +\infty$, 
$$ \int_{E} \exp(\delta {\tt h}(z))\,m({\rm d}z) < + \infty.$$

\end{rem}
For $\epsilon>0$, consider the following SPDE on the probability space $(\Omega, \mathcal{F}, \mathbb{P}, \mathbb{F})$ as described above. 
\begin{equation}
 \label{eq:log-LDP}
      \begin{aligned}
            du_{\epsilon} - \Delta u_{\epsilon} \, {\rm d}t&= u_{\epsilon} \log |u_{\epsilon}| \,{\rm d}t + \epsilon \int_{|z|>0}\eta(u_{\epsilon};z)\widetilde{N}^{\epsilon^{-1}}({\rm d}z,{\rm d}t),~~(t,x)\in (0,T]\times D\,, \\
            u_{\epsilon}(0,x) &= u_0(x),~~~x\in D\,.
        \end{aligned}
\end{equation}
In view of Remark \ref{rem:assumption} and Theorem \ref{thm:existence-strong}, 
under the assumptions \ref{A1} and \ref{B1}, equation \eqref{eq:log-LDP} has a unique strong solution $u_\epsilon$ whose trajectories a.s. belong to the space $\mathcal{E}_T:=\mathbb{D}([0,T];L^2(D))\cap L^2([0,T];W_0^{1,2}(D))$.  Moreover, in view of Yamada-Watanabe theorem \cite[Theorem 8]{Zhao-2014}, there exists a family $\{\mathcal{G}^\epsilon\}_{\epsilon>0}$, where $\mathcal{G}^\epsilon: {\tt M}_T\rightarrow \mathcal{E}_T$ is a measurable map, such that $u_\eps=\mathcal{G}^\epsilon(\epsilon N^{\epsilon^{-1}})$. 

 \vspace{0.1cm}
 
Like in the Brownian noise case, we need to introduce so called skeleton equation. For any Borel measurable function $\varphi: [0,T]\times \pmb{E} \rightarrow [0,\infty)$, define 
$$
L_T(\varphi):=\int_{[0,T]\times \pmb{E}} \big(\varphi(t, z) \log(\varphi(t,z))-\varphi(t,z)+1\big)\,m({\rm d}z)\,{\rm d}t\,.$$
For any $N\in \mathbb{N}$, define 
$$ S_N = \left \{ g: [0,T]\times \pmb{E} \rightarrow [0,\infty) : L_{T}(g) \leq N  \right \}, \quad \mathbb{S}:=\cup_{N\ge 1} S_N\,.$$
Any $g \in S_N$ can be identified with a measure $\nu_{T}^{g} \in {\tt M}_T$, defined by
$$\nu_{T}^{g}(A) = \int_{A} g(s,z)\, m({\rm d}z)\,{\rm d}s, \quad~~A \in \mathcal{B}([0,T]\times \pmb{E}).$$
This identification induces a topology $\tau(S_N)$ on $S_N$ in which $(S_N, \tau(S_N))$ is a compact space; see \cite[Appendix]{Budhiraja-2013}. Define the following set
$$ \mathcal{U}_N: = \left \{\phi \in \mathcal{A} : \phi(\omega) \in S_N , \mathbb{P} \,\text{-a.e.}\, \omega\in \Omega \right \}. $$
Fix an increasing sequence $\{\pmb{E}_n\}_{n\in \mathbb{N}}$ of compact sets in $\pmb{E}$ such that $\pmb{E}=\bigcup_{n=1}^{\infty} \pmb{E}_n$. Define
\begin{align*}
\tilde{\mathcal{A}}&:= \bigcup_{n=1}^\infty \Big\{ \varphi \in \mathcal{A}:~~\varphi(t,z,\omega)\in [\frac{1}{n},n]~\text{if $(t,z,\omega)\in [0,T]\times \pmb{E}_n\times \Omega$} \notag \\
&\hspace{2cm} \text{and $\varphi(t,z,\omega)=1$ if $(t,z,\omega)\in [0,T]\times \pmb{E}_n^c\times \Omega$}\Big\}\,, \notag  \\
\tilde{\mathcal{U}}_N&:=\left \{\varphi \in \tilde{\mathcal{A}} : \varphi(\cdot,\cdot,\omega) \in S_N , \mathbb{P} \,\text{-a.e.}\, \omega\in \Omega \right \}.
\end{align*}
For any $g\in \mathbb{S}$, consider the following deterministic equation (called skeleton equation):
\begin{equation} \label{eq:log-skeleton}
\begin{aligned}
du_{g}(t,x)  - \Delta u_{g} \, {\rm d}t&= u_{g} \log |u_{g}| \,{\rm d}t + \int_{\pmb{E}}\eta(u_{g}(t,x);z)\, (g(t,z)-1) \, m({\rm d}z)\,{\rm d}t,~~(t,x)\in (0,T]\times D\,, \\
 u_{g}(0,x) &= u_0(x),~~x\in D,\quad u_g(t,x)=0,~~(t,x)\in [0,T]\times \partial D\,.
\end{aligned}
\end{equation}
\begin{thm}\label{thm:skeleton}
Let the assumptions \ref{A1} and \ref{B1} hold true. Then there exists a unique solution $u_g \in C([0,T]; L^2(D))\cap L^2([0,T]; W_0^{1,2}(D))$ of the skeleton equation \eqref{eq:log-skeleton}. Moreover, for fixed $N\in \mathbb{N}$, there exists a constant $C_N>0$ such that
\begin{align}
\sup_{g\in S_N} \Big\{ \sup_{t\in [0,T]}\|u_g(t)\|_{L^2(D)}^2 + \int_0^T\|u_g(t)\|_{W_0^{1,2}(D)}^2\,{\rm d}t \Big\} \le C_N\,. \label{esti:uni-skeleton}
\end{align}
\end{thm}
This allow us to define a measurable map $\mathcal{G}^0: \mathbb{S}\ni g\mapsto u_g\in \mathcal{E}_T$. 
\vspace{0.1cm}

We now state the main result regarding the large deviation principle. 
\begin{thm}\label{thm:ldp-log-laplace-levy}
Let the assumptions \ref{A1} and \ref{B1} hold true. Then the family $\{u_\epsilon\}_{\epsilon >0}$ satisfies a large deviation principle
on $\mathcal{E}_T$ with the good rate function ${\tt I}: \mathcal{E}_T\rightarrow [0, + \infty]$ given by
 $$ {\tt I}(\phi):=\inf\Big\{ L_T(g):~~~\text{ $g\in \mathbb{S}$ with $u_g=\phi$}\Big\},\quad \phi\in \mathcal{E}_T$$
 with the convention that $\inf(\emptyset)= \infty$, 
 where for $g\in \mathbb{S}$, $u_g$ is the unique solution of \eqref{eq:log-skeleton}. 
\end{thm}

\section{Well-posedness theory for strong solution of \eqref{eq:log-nonlinear}}\label{sec:existence-weak-solu}

In this section, we establish the well-posedness of equation \eqref{eq:log-nonlinear}.  In order to achieve this, we initially use the  Galerkin methods to construct approximate solutions. Following this, we derive its necessary {\it a-priori} bounds. Firstly, we show the existence of a martingale solution and then prove its global solution.

\subsection{Galerkin approximating solutions}
Let $\{e_n\}$ be an orthogonal basis of $V:=W_0^{1,2}(D)$ and orthonormal basis of $L^2(D)$ consists of the eigen-functions of $-\Delta$ operator with a Dirichlet boundary condition corresponding to the eigenvalue $\lambda_n$. Note that for each $n\in \mathbb{N}$, $e_n\in L^\infty(D)$. Let $L_{n}$ denote the $n$-dimensional subspace of $L^2(D)$ spanned by $\left\{e_1, \ldots, e_n\right\}$. Let $P_n: V^* \rightarrow L_n$ be the projection operator defined by
\begin{align} \label{def-projection}
    P_n h:=\sum_{j=1}^n \left\langle h, e_j\right\rangle e_j\,, \quad h\in V^*\,.
\end{align}
For any  $n \in \mathbb{N}$, we consider the following stochastic differential equation in $L_{n}$:
\begin{align} 
    \begin{cases}
   \displaystyle {\rm d}u_n(t)=\Delta u_n(t) {\rm d}t + P_n\left[u_n(t) \log \left|u_n(t)\right|\right] {\rm d}t+  \int_{\pmb{E}} P_n [\eta(u_n;z)] \widetilde{N}({\rm d}z,{\rm d}t), \quad t> 0\,, \\
u_n(0)=P_n u_0\,,
\end{cases} \label{eq:finite-dim}
\end{align}
such that
\begin{align} \label{eq:finite-dim-def} 
   u_n(t)=\sum_{j=1}^n h_{j n}(t) e_j \,.
\end{align}
Note that $u_n$ solves \eqref{eq:finite-dim} if and only if $\left\{h_{j n}\right\}_{j=1}^n$ solves the system of SDE
\begin{align} \label{eq:galerkin}
{\rm d}h_{i n}(t)
&=  -\lambda_i h_{i n}(t) \,{\rm d}t+\left(\sum_{j=1}^n h_{j n}(t) e_j \log \left|\sum_{j=1}^n h_{j n}(t) e_j\right|, e_i\right) \,{\rm d}t  \notag \\
& + \int_{\pmb{E}} \left(\eta \bigg(\sum_{j=1}^n h_{j n}(t) e_j ;z\bigg), e_i\right) \widetilde{N}({\rm d}z,{\rm d}t) , \quad i=1,2, \ldots, n\,.
\end{align}
To present the existence and uniqueness results for the system \eqref{eq:galerkin}, we introduce the following functions $F_i$ and $H_i, i=1, \ldots, n$, on $\mathbb{R}^n$:
\begin{align*}
F_i\left(y_1, \ldots, y_n\right) & :=\displaystyle \int_D e_i(x)\left(\sum_{j=1}^n y_j e_j(x)\right) \log \left|\sum_{j=1}^n y_j e_j(x)\right| \,{\rm d}x\,, \\
H_i\left(y_1, \ldots, y_n;z\right) & :=\int_D e_i(x)\,   \eta \bigg(\sum_{j=1}^n y_j e_j(x); z \bigg)  \,{\rm d}x\,.
\end{align*}
To proceed further, we first recall the following estimates for the difference between two logarithmic terms. For its proof, we refer to see \cite[Lemmas $3.1$ $\&$ $3.2$]{Shang-2022}.
\begin{lem} \label{lem:result-1}
     For any $\xi, \zeta \in V, \epsilon>0$, and $\alpha \in(0,1)$, we have
\begin{align*}
   & (\xi \log |\xi|-\zeta \log |\zeta|, \xi-\zeta) \\
& \le \epsilon \, \|\xi-\zeta\|_{W_0^{1,2}(D)}^2+\left(1+\frac{d}{4} \log \frac{1}{\epsilon}\right)\|\xi-\zeta\|_{L^2(D)}^2 +\|\xi-\zeta\|_{L^2(D)}^2 \log \|\xi-\zeta\|_{L^2(D)} \\
&\quad +\frac{1}{2(1-\alpha) \mathrm{e}}\left(\|\xi\|_{L^2(D)}^{2(1-\alpha)}+\|\zeta\|_{L^2(D)}^{2(1-\alpha)}\right)\|\xi-\zeta\|_{L^2(D)}^{2 \alpha}\,. 
\end{align*}
 \end{lem}

\begin{lem} \label{lem:result-2}
     For any $\xi, \zeta \in V, \epsilon>0$, and $\alpha \in(0,1)$, we have
\begin{align*}
& \int_D|\xi(x)-\zeta(x)|^2 \log _{+}(|\xi(x)| \vee |\zeta(x)|) \, {\rm d}x \\
& \le \epsilon \, \|\xi-\zeta\|_{W_0^{1,2}(D)}^2+\left(\frac{d}{4} \log \frac{1}{\epsilon}\right)\|\xi-\zeta\|_{L^2(D)}^2+\|\xi-\zeta\|_{L^2(D)}^2 \log \|\xi-\zeta\|_{L^2(D)} \\
&\quad  +\frac{1}{2(1-\alpha) \mathrm{e}}\left(\|\xi\|_{L^2(D)}^{2(1-\alpha)}+\|\zeta\|_{L^2(D)}^{2(1-\alpha)}\right)\|\xi-\zeta\|_{L^2(D)}^{2 \alpha} \\
& \qquad +\frac{1}{2(1-\alpha) \mathrm{e}}(4 \lambda(D))^{1-\alpha}\|\xi-\zeta\|_{L^2(D)}^{2 \alpha}\,,
\end{align*}
where $\lambda(D)$ is the Lebesgue measure of the domain $D$.
\end{lem}
For any vector $v=\left(v_1, \ldots, v_n\right) \in \mathbb{R}^n$, we denote the length of $v$ as $|v|$. 
We provide some essential estimates for the functions $F_i$ and $G_i,~i=1, \ldots, n$.

\begin{lem} \label{est:finite-log}
Under the assumptions \ref{A1}-\ref{A3}, the following estimates hold. 
    \begin{itemize}
        \item [(i)] There exist non-negative constants $\widetilde{K_1}, \widetilde{K_2}, \widetilde{K_3}$ and $\delta>0$ such that for any $v, w \in \mathbb{R}^n$ with $|v-w| \leq \delta$,
        \begin{align*}
            & \left|F_i\left(v_1, \ldots, v_n\right)-F_i\left(w_1, \ldots, w_n\right)\right| \notag \\
\leq & \widetilde{K_1}|v-w|+\widetilde{K_2}|v-w| \log _{+}(|v| \vee|w|)+\widetilde{K_3}|v-w| \log \frac{1}{|v-w|}\,.
        \end{align*}
        \item [(ii)] There exist  constants $\widetilde{C_1}, \widetilde{C_2} \ge 0$ such that for any $v \in \mathbb{R}^n$,
        \begin{align*}
            & \left|F_i\left(v_1, \ldots, v_n\right)\right| 
\leq  \widetilde{C_1} + \widetilde{C_2}|v| \log _{+}|v|\,.
        \end{align*}
        \item [(iii)] There exist non-negative constants $\widetilde{K_4}, \widetilde{K_5}$ such that for any $v, w \in \mathbb{R}^n$,
        \begin{align*}
            & \int_{\pmb{E}} \left|H_i\left(v_1, \ldots, v_n;z\right)-H_i\left(w_1, \ldots, w_n;z\right)\right|^2 \,m({\rm d}z) \notag \\
\leq & \widetilde{K_4}|v-w|^2+\widetilde{K_5}|v-w| \log _{+}(|v| \vee|w|)\,.
\end{align*}
\item [(iv)] There exist constants $\widetilde{C_3}, \widetilde{C_4}\ge 0$ such that for any $v \in \mathbb{R}^n$,
        \begin{align*}
            & \int_{\pmb{E}} \left|H_i\left(v_1, \ldots, v_n;z\right)\right|^2 \,m({\rm d}z)
\leq  \widetilde{C_3} + \widetilde{C_4}|v|^2 (\log _{+}|v|)\,.
        \end{align*}
    \end{itemize}
\begin{proof}
    The assertions {\rm (i)} and {\rm (ii)} follow from \cite[Lemma 4.1]{Shang-2022}. To prove {\rm (iii)}, we first set
    \begin{align*}
        y_1(x) := \sum_{j=1}^n v_j e_j(x), \quad y_2(x) := \sum_{j=1}^n w_j e_j(x).
    \end{align*}
    Note that
    \begin{align} \label{eq:est-y1}
        |y_1(x)| = \left | \sum_{j=1}^n v_j e_j(x) \right | \le |v| \left ( \sum_{j=1}^n  \|e_j\|_{L^{\infty}(D)}^2 \right )^\frac{1}{2} \le (|v| \vee |w|) \left ( \sum_{j=1}^n  \|e_j\|_{L^{\infty}(D)}^2 \right )^\frac{1}{2}\,.
    \end{align} 
    In view of the assumption \ref{A2}, we get
    \begin{align}
       \int_{\pmb{E}} & |H_i\left(v_1, \ldots, v_n;z \right)-H_i\left(w_1, \ldots, w_n;z\right)|^2 \,m({\rm d}z) \notag \\
       & = \int_{\pmb{E}} \left|\int_{D}^{} e_i(x) \bigg(\eta(y_1(x);z) - \eta(y_2(x);z) \bigg) {\rm d}x \right|^2 \,m({\rm d}z) \notag \\
       & \le C \int_{D} e_i^2(x) \left[\int_{\pmb{E}} \left| \eta(y_1(x);z) - \eta(y_2(x);z) \right|^2 \,m({\rm d}z) \right] {\rm d}x  \notag \\
        & \le C \|e_i \|_{L^\infty(D)}^2  \int_{D} \int_{\pmb{E}}\Big\{ 2\,K_1^{2} |y_1-y_2|^2 + 2\,K_2^{2} |y_1-y_2|^2 \log _{+} (|y_1| \vee |y_2|)\Big\}  {\tt h}^2(z) \,m({\rm d}z) {\rm d}x  \notag \\
        & \le C \|e_i \|_{L^\infty(D)}^2 \left[ \int_{D}^{} |y_1-y_2|^2 \, {\rm d}x + \int_{D}^{} |y_1-y_2|^2 \big(\log _{+} (|y_1| \vee |y_2|) \big)\, {\rm d}x \right]\,. \notag
    \end{align}
   By using \eqref{eq:est-y1} and using similar argument~(with the cosmetic modifications) as in the proof of \cite[Lemma $4.1$, pp. 97-99]{Shang-2022} to conclude the assertion ${\rm (iii)}$. The proof of assertion {\rm (iv)} is similar to the proof of assertion {\rm (iii)}.
\end{proof}
\end{lem}
Thanks to Lemma \ref{est:finite-log}, we can directly utilize Theorems $2.2$ and $2.4$ of \cite{Fubao-2019}~(see also \cite{Yuchao-2018,Situ-2005}) to obtain the following theorem.
\begin{thm}
    Under \ref{A1}-\ref{A3}, the equation \eqref{eq:finite-dim} admits a unique strong solution.
\end{thm}

\subsection{\bf {\it A-priori} estimates} In this subsection, we wish to derive the necessary uniform bounds for $\{u_n\}$. To begin with, we state the following logarithmic Sobolev inequality; see \cite{Shang-2022}.
\begin{lem}[Logarithmic Sobolev inequality]
For any $\eps>0$ and $u\in W_0^{1,2}(D),$ the following inequalities hold:
\begin{align}
& {i)}~~\int_{D} |u(x)|^2 \log(|u(x)|)\,{\rm d}x \le \eps \|u\|_{W_0^{1,2}(D)}^2 + \frac{d}{4} \log(\frac{1}{\eps})\|u\|_{L^2(D)}^2 + \|u\|_{L^2(D)}^2 \log(\|u\|_{L^2(D)})\,.\label{inq:log-sov-1} \\
& {ii)}~~\int_{D} |u(x)|^2 \log_{+}(|u(x)|)\,{\rm d}x \le \eps \|u\|_{W_0^{1,2}(D)}^2 + \frac{d}{4} \log(\frac{1}{\eps})\|u\|_{L^2(D)}^2 + \|u\|_{L^2(D)}^2 \log(\|u\|_{L^2(D)}) + \frac{1}{2e} \lambda(D)\,. \notag
\end{align}
\end{lem}
We frequently use the following nonlinear versions of Gronwall's inequalities, whose proof can be found in \cite{Mitrinovic-1991,Shang-2022} 
\begin{lem}\label{lem:nonlinear-gronwall}
Let $y(\cdot), f(\cdot)$ and $g(\cdot)$ be non-negative functions on $\R$ such that $y(\cdot)$ satisfies the integral inequality 
\begin{align*}
y(t) \le C + \int_{t_0}^t \big( f(s) y(s) + g(s) y^\alpha(s)\big)\,{\rm d}s,~~~t\ge t_0\ge 0,
\end{align*}
where $C\ge 0$ and $\alpha\in [0,1)$ are given constants. Then, for any $t\ge t_0\ge 0$,
\begin{align*}
y(t)\le \Bigg\{ C^{1-\alpha} \exp\Big( (1-\alpha) \int_{t_0}^t f(s)\,{\rm d}s\Big) + (1-\alpha) \int_{t_0}^t g(s)  \exp\Big( (1-\alpha) \int_{s}^t f(r)\,{\rm d}r\Big)\,{\rm d}s \Bigg\}^\frac{1}{1-\alpha}\,.
\end{align*}
\end{lem}
\begin{lem}\label{lem:log-Gronwall}
Let $y(\cdot), f(\cdot), g(\cdot), h(\cdot)$ and $a(\cdot)$ be non-negative functions on $[0,\infty)$, and $h(\cdot)$ be increasing function with $h(0)\ge 1$. Assume that, for any $t\ge 0$, the following inequality
\begin{align*}
y(t) + a(t) \le h(t) + \int_0^t f(s)y(s)\,{\rm d}s + \int_0^t g(s) y(s)\log(y(s))\,{\rm d}s
\end{align*}
holds and all the integrals are finite. Then for any $t\ge 0$, one has
\begin{align*}
y(t)+ a(t)\le \Big( h(t)\Big)^{\exp(G(t))} \exp\Big\{ \exp(G(t))\int_0^t f(s)\exp(-G(s))\,{\rm d}s\Big\}\,, 
\end{align*}
where the function $G(t)$ is given by
$$ G(t):=\int_0^t g(s)\,{\rm d}s\,.$$
\end{lem}

\begin{lem} \label{lem:est-galerkin}
Under assumptions \ref{A1}-\ref{A3}, the following estimate holds: for any $p \ge 2$,
    \begin{align}
       \underset{n}\sup\, \mathbb{E} & \Big[  \underset{t\in [0,T_p]}\sup\,\| u_n(s)\|_{L^2(D)}^p + \int_{0}^{T_p} \| u_n(s)\|_{L^2(D)}^{p-2} \| u_n(s)\|_{W_0^{1,2}(D)}^2 \,{\rm d}s \Big] \notag \\
       & \le C_{p,\theta} \big( 1+ \| u_0\|_{L^2(D)}^{\frac{p^2}{p-1+\theta}} \big) < \infty\,, \label{esti:uniform-galerkin}
    \end{align}
   for some constant $C_{p,\theta}$, where $T_p := \log \frac{p}{p-1+\theta}$ and $\theta$ is the constant in the assumption \ref{A3}. Moreover, $T_p$ is decreasing in $p \in [2,\infty)$.
   \end{lem}
    \begin{proof}
        For any $n \in \mathbb{N}, R > 0,$ we define stopping times
        \begin{align*}
            \tau_{R}^{n} := \inf \Big \{t>0:\|u_n(t)\|_{L^2(D)} >R \Big \} \wedge T_p.
        \end{align*}
        Because $u_n$ has no explosion, as $R \rightarrow \infty$ we have, 
        \begin{align*}
            \tau_{R}^{n} \rightarrow T_p ~~ \text{for}~~ \ \mathbb{P}-a.s. 
        \end{align*}
   By employing the It\^{o}-L\'{e}vy formula, one has, for $t \le \tau_{R}^{n}$ 
   \begin{align*}
      {\rm d} \|u_n(t) \|_{L^2(D)}^{2} & = -2\| \nabla u_{n}(t)\|_{L^2(D)}^{2} \, {\rm d}t + 2 \big(u_n(t) \log|u_n(t)|, u_n(t)\big) \, {\rm d}t \\
      & + \int_{\pmb{E}}  \Big ( \|  u_n(t) + P_{n} \eta(u_n(t);z) \|_{L^2(D)}^2 -  \|  u_n(t)\|_{L^2(D)}^2 \Big) \, \widetilde{N}({\rm d}z,{\rm d}t) \\
      & +  \int_{\pmb{E}} \|P_{n} \eta(u_n(t);z) \|_{L^2(D)}^2 m({\rm d}z)\,{\rm d}t.
   \end{align*}
   Again, using It\^{o}-L\'{e}vy formula, we have 
   \begin{align*}
       \|u_n(t) \|_{L^2(D)}^{p} & = \|u_0\|_{L^2(D)}^{p} - p  \int_{0}^{t} \|u_n(s) \|_{L^2(D)}^{p-2} \|u_n(s) \|_{W_{0}^{1,2}(D)}^{2} \,{\rm d}s \\
       & + p  \int_{0}^{t} \|u_n(s) \|_{L^2(D)}^{p-2} \big(u_n(s) \log|u_n(s)|, u_n(s)\big)  \,{\rm d}s \\
       & +  \int_{0}^{t} \int_{\pmb{E}} \Big ( \|  u_n(s) + P_{n} \eta(u_n(s);z) \|_{L^2(D)}^p -  \|  u_n(s)\|_{L^2(D)}^p  \Big) \, \widetilde{N}({\rm d}z,{\rm d}s) \\
       & + \int_{0}^{t} \int_{\pmb{E}} \Big ( \|  u_n(s) + P_{n} \eta(u_n(s);z) \|_{L^2(D)}^p - \| u_n(s)\|_{L^2(D)}^p \\
       & \hspace{3cm} - p \|u_n(s) \|_{L^2(D)}^{p-2} \big( P_{n} \eta(u_n(s);z), u_n(s) \big) \Big) \,  m({\rm d}z)\,{\rm d}s. 
   \end{align*}
   We use Taylor's formula and the logarithmic Sobolev inequality \eqref{inq:log-sov-1} with $\epsilon = \frac{1}{2}$ to have
   \begin{align*}
       \| & u_n(t) \|_{L^2(D)}^{p} + p \int_{0}^{t} \|u_n(s) \|_{L^2(D)}^{p-2} \|u_n(s) \|_{W_{0}^{1,2}(D)}^{2} \,{\rm d}s \\
       & \le \|u_0\|_{L^2(D)}^{p}   + p  \int_{0}^{t} \|u_n(s) \|_{L^2(D)}^{p-2} \bigg( \frac{1}{2} \|u_n(s) \|_{W_{0}^{1,2}(D)}^{2} + \frac{d}{4}\log(2) \|u_n(s) \|_{L^2(D)}^2 \\
       & + \|u_n(s) \|_{L^2(D)}^2 \log \|u_n(s) \|_{L^2(D)}  \bigg)  \,{\rm d}s 
        + p  \int_{0}^{t} \int_{\pmb{E}} \|u_n(s) \|_{L^2(D)}^{p-2} \big( P_{n} \eta(u_n(s);z), u_n(s) \big)  \,   \widetilde{N}({\rm d}z,{\rm d}s) \\
       & +p(p-1)  \int_{0}^{t} \int_{\pmb{E}} \int_0^1 (1-\gamma) \Big ( \| u_n(s) + \gamma P_{n} \eta(u_n(s);z) \|_{L^2(D)}^{p-2} \| P_{n}\eta(u_n(s);z)\|_{L^2(D)}^2  \Big)\,d\gamma \,N({\rm d}z,{\rm d}s), 
   \end{align*} 
   and hence we get
   \begin{align} \label{eq:gron-1}
& \|  u_n(t) \|_{L^2(D)}^{p} + \frac{p}{2} \int_{0}^{t} \|u_n(s) \|_{L^2(D)}^{p-2} \|u_n(s) \|_{W_{0}^{1,2}(D)}^{2} \,{\rm d}s \notag \\
& \le  \|u_0\|_{L^2(D)}^{p} + C \int_{0}^{t}  \|u_n(s) \|_{L^2(D)}^{p} \,{\rm d}s  + \int_{0}^{t}  \|u_n(s) \|_{L^2(D)}^{p} \log  \|u_n(s) \|_{L^2(D)}^p  \,{\rm d}s \notag \\
& +  p \underset{r \in [0,t]}\sup \left |\int_{0}^{r} \int_{\pmb{E}} \|u_n(s) \|_{L^2(D)}^{p-2} \big(P_n \eta(u_n(s);z), u_n(s) \big)  \,   \widetilde{N}({\rm d}z,{\rm d}s)   \right | \notag \\
& + C_p \int_{0}^{t} \int_{\pmb{E}} \Big ( \| u_n(s)\|_{L^2(D)}^{p-2} \| \eta(u_n(s);z) \|_{L^2(D)}^{2} 
       + \|\eta(u_n(s);z)\|_{L^2(D)}^p  \Big) \,N({\rm d}z,{\rm d}s)\notag \\
       & \le {\tt M}(t) + C \int_{0}^{t}  \|u_n(s) \|_{L^2(D)} \,{\rm d}s  + \int_{0}^{t}  \|u_n(s) \|_{L^2(D)}^{p} \log  \|u_n(s) \|_{L^2(D)}^{p}  \,{\rm d}s,
   \end{align}
   where 
   \begin{align*}
       {\tt M}(t) := & \|u_0\|_{L^2(D)}^{p} + p \underset{r \in [0,t]}\sup \left |\int_{0}^{r} \int_{\pmb{E}} \|u_n(s) \|_{L^2(D)}^{p-2} \big( P_n\eta(u_n(s);z), u_n(s) \big)  \,   \widetilde{N}({\rm d}z,{\rm d}s)   \right |  \\
       &+ C_p  \int_{0}^{t} \int_{\pmb{E}} \Big( \| u_n(s)\|_{L^2(D)}^{p-2} \| \eta(u_n(s);z) \|_{L^2(D)}^{2} 
       + \|\eta(u_n(s);z)\|_{L^2(D)}^p  \Big) \,N({\rm d}z,{\rm d}s).
   \end{align*}
   Note that these constants $C, C_p$ are independent of $n$. We use log-Gronwall's inequality  i.e., Lemma \ref{lem:log-Gronwall} to \eqref{eq:gron-1} to have
   \begin{align}
       \| & u_n(t) \|_{L^2(D)}^{p} + \frac{p}{2} \int_{0}^{t} \|u_n(s) \|_{L^2(D)}^{p-2} \|u_n(s) \|_{W_{0}^{1,2}(D)}^{2} \,{\rm d}s 
        \le \big( 1 \vee {\tt M}(t) \big)^{e^{t}} \times e^{C(e^t -1)}. \label{eq:gron-2}
   \end{align}
Hence, from \eqref{eq:gron-2}, we get 
\begin{align}\label{eq:est-p1}
 \mathbb{E} & \bigg[\underset{s \in [0,t \wedge \tau_{R}^{n}]}\sup  \| u_n(s)\|_{L^2(D)}^{p}  \bigg] + \frac{p}{2} \mathbb{E} \Big[\int_{0}^{t \wedge \tau_{R}^{n}} \|u_n(s) \|_{L^2(D)}^{p-2} \|u_n(s) \|_{W_{0}^{1,2}(D)}^{2} \,{\rm d}s \Big] \notag \\
      & \le e^{C(e^{T_p} -1)} \mathbb{E} \bigg[ \big({\tt M}(t \wedge \tau_{R}^{n}) + 1 \big)^{e^{T_p}}  \bigg]  \le C \mathbb{E} \bigg[ \big({\tt M}(t \wedge \tau_{R}^{n}) + 1 \big)^{\frac{p}{p-1+\theta}} \bigg] \notag \\
      & \le C \bigg( 1 + \| u_0\|_{L^2(D)}^{\frac{p^2}{p-1+\theta}} \bigg) + C \mathbb{E} \Bigg[\underset{r \in [0,t \wedge \tau_{R}^{n}]}\sup \left |\int_{0}^{r} \int_{\pmb{E}} \|u_n(s) \|_{L^2(D)}^{p-2} \big(P_n \eta(u_n(s);z), u_n(s) \big)  \,   \widetilde{N}({\rm d}z,{\rm d}s)   \right | ^{\frac{p}{p-1+\theta}}\Bigg] \notag \\
      & \quad + C \Bigg|\mathbb{E} \Big[ \int_{0}^{t\wedge \tau_{R}^{n}} \int_{\pmb{E}} \Big( \| u_n(s)\|_{L^2(D)}^{p-2}\| \eta(u_n(s);z) \|_{L^2(D)}^{2} + \| \eta(u_n(s);z) \|_{L^2(D)}^{p} \Big)  \,N({\rm d}z,{\rm d}s)\Big] \Bigg|^{\frac{p}{p-1+\theta}}\notag \\
      & \equiv  \sum_{i=1}^3\mathcal{K}_{i}.
\end{align}
   One can easily see that using the assumption \ref{A3} and H\"{o}lder's inequality to have
   \begin{align} \label{est1}
        \int_{0}^{t} \int_{\pmb{E}} \| \eta(u(s);z) \|_{L^2(D)}^{\gamma}  m({\rm d}z)\,{\rm d}s \le C \int_{0}^{t} \big( 1 + \| u\|_{L^2(D)}^{\gamma \theta} \big)  \,{\rm d}s, \quad \forall \gamma \ge 2. 
   \end{align}
 By applying Burkholder-Davis-Gundy inequality~(BDG inequality), Young's inequality and \cite[Corollary $2.4$]{Liu-2019} along with \eqref{est1}, we obtain 
 \begin{align}
     \mathcal{K}_2 & \le C \mathbb{E} \Bigg[ \bigg(\int_{0}^{t\wedge \tau_{R}^{n}} \int_{\pmb{E}} \|u_n(s) \|_{L^2(D)}^{2p-2} \| P_{n}\eta(u_n(s);z)\|_{L^2(D)}^2 \,N({\rm d}z,{\rm d}s) \bigg) ^{\frac{p}{2(p-1+\theta)}}\Bigg] \notag  \\
      & \le  C \mathbb{E} \bigg[ \underset{s \in [0,t \wedge \tau_{R}^{n}]}\sup  \|u_n(s) \|_{L^2(D)}^{\frac{p(p-1)}{p-1+\theta}}  \bigg( \int_{0}^{t\wedge \tau_{R}^{n}} \int_{\pmb{E}} \| \eta(u_n(s);z)\|_{L^2(D)}^2 \,N({\rm d}z,{\rm d}s) \bigg)^{\frac{p}{2(p-1+\theta)}} \bigg] \notag \\
       & \le  \epsilon_1 \mathbb{E} \bigg[ \underset{s \in [0,t \wedge \tau_{R}^{n}]}\sup  \|u_n(s) \|_{L^2(D)}^{p} \bigg] + C_{\epsilon_1} \mathbb{E} \Bigg[\bigg(\int_{0}^{t\wedge \tau_{R}^{n}} \int_{\pmb{E}} \| \eta(u_n(s);z)\|_{L^2(D)}^2 \,N({\rm d}z,{\rm d}s) \bigg)^{\frac{p}{2\theta}}\Bigg] \notag \\
       &  \le  \epsilon_1 \mathbb{E} \bigg[ \underset{s \in [0,t \wedge \tau_{R}^{n}]}\sup  \|u_n(s) \|_{L^2(D)}^{p} \bigg] + C(\epsilon_1) \mathbb{E} \bigg[\int_{0}^{t\wedge \tau_{R}^{n}} \int_{\pmb{E}} \| \eta(u_n(s);z)\|_{L^2(D)}^{\frac{p}{\theta}} \,m({\rm d}z)\,{\rm d}s \bigg] \notag \\
       & \hspace{4cm}+ C(\epsilon_1) \mathbb{E} \bigg[ \bigg(\int_{0}^{t\wedge \tau_{R}^{n}} \int_{\pmb{E}} \| \eta(u_n(s);z)\|_{L^2(D)}^2 \,m({\rm d}z)\,{\rm d}s \bigg)^{\frac{p}{2\theta}}\bigg] \notag \\
       & \le  C_{T_p} +  \epsilon_1 \mathbb{E} \bigg[ \underset{s \in [0,t \wedge \tau_{R}^{n}]}\sup  \|u_n(s) \|_{L^2(D)}^{p} \bigg] + C(\epsilon_1)  \mathbb{E} \bigg[\int_{0}^{t\wedge \tau_{R}^{n}} \| u_n(s)\|_{L^2(D)}^{p} \,{\rm d}s \bigg]  \notag \\
       & \hspace{2cm}+  C(\epsilon_1)  \mathbb{E} \bigg[ \Big( \int_{0}^{t\wedge \tau_{R}^{n}} \| u_n(s)\|_{L^2(D)}^{2\theta} \,{\rm d}s \Big)^{\frac{p}{2\theta}}\bigg] \notag \\
       & \le  C_{T_p} + \epsilon_1 \mathbb{E} \bigg[ \underset{s \in [0,t \wedge \tau_{R}^{n}]}\sup  \|u_n(s) \|_{L^2(D)}^{p} \bigg] + C(\epsilon_1)  \mathbb{E} \bigg[\int_{0}^{t\wedge \tau_{R}^{n}} \| u_n(s)\|_{L^2(D)}^{p} \,{\rm d}s \bigg] \notag \\
       & \le  C_{T_p} + \epsilon_1 \mathbb{E} \bigg[ \underset{s \in [0,t \wedge \tau_{R}^{n}]}\sup  \|u_n(s) \|_{L^2(D)}^{p} \bigg] + C(\epsilon_1) \int_{0}^{t}\mathbb{E} \bigg[ \underset{r \in [0,s \wedge \tau_{R}^{n}]}\sup  \|u_n(r) \|_{L^2(D)}^{p} \bigg] \,{\rm d}s.\label{esti:k2}
 \end{align}
 We use \cite[Corollary $2.4$]{Liu-2019} along with \eqref{est1} to estimate $\mathcal{K}_3$ as follows.
 \begin{align*}
     \mathcal{K}_{3}  & \le C_p  \mathbb{E} \Bigg[\bigg( \int_{0}^{t\wedge \tau_{R}^{n}} \int_{\pmb{E}} \| u_n(s)\|_{L^2(D)}^{p-2} \| \eta(u_n(s);z) \|_{L^2(D)}^{2}   \,N({\rm d}z,{\rm d}s) \bigg)^{\frac{p}{p-1+\theta}}\Bigg] \\
      & \quad +  C_p  \mathbb{E} \Bigg[ \bigg(\int_{0}^{t\wedge \tau_{R}^{n}} \int_{\pmb{E}} \| P_{n}\eta(u_n(s);z)\|_{L^2(D)}^p  \,N({\rm d}z,{\rm d}s) \bigg)^{\frac{p}{p-1+\theta}}\Bigg] \\
      & \le C_p  \mathbb{E} \Bigg[ \int_{0}^{t\wedge \tau_{R}^{n}} \int_{\pmb{E}} \bigg( \| u_n(s)\|_{L^2(D)}^{p-2} \| \eta(u_n(s);z) \|_{L^2(D)}^{2} \bigg)^{\frac{p}{p-1+\theta}}   \, m({\rm d}z)\,{\rm d}s \Bigg] \\
     & \quad + C_p  \mathbb{E} \Bigg[ \bigg(\int_{0}^{t\wedge \tau_{R}^{n}} \int_{\pmb{E}} \| u_n(s)\|_{L^2(D)}^{p-2} \| \eta(u_n(s);z) \|_{L^2(D)}^{2} \, m({\rm d}z)\,{\rm d}s \bigg)^{\frac{p}{p-1+\theta}}\Bigg] \\ 
     & \quad + C_p  \mathbb{E} \Bigg[ \int_{0}^{t\wedge \tau_{R}^{n}} \int_{\pmb{E}} \| P_{n}\eta(u_n(s);z)\|_{L^2(D)}^{\frac{p^2}{p-1+\theta}}  \, m({\rm d}z)\,{\rm d}s \Bigg] \\
      & \quad + C_p  \mathbb{E} \Bigg[\bigg( \int_{0}^{t\wedge \tau_{R}^{n}} \int_{\pmb{E}} \| P_{n}\eta(u_n(s);z)\|_{L^2(D)}^p  \, m({\rm d}z)\,{\rm d}s \bigg)^{\frac{p}{p-1+\theta}}\Bigg]  \equiv  \sum_{i=1}^4\mathcal{K}_{3,i}.
 \end{align*}
 One can use Young's inequality along with \eqref{est1} to have
 \begin{align*}
     \mathcal{K}_{3,1} & \le C \mathbb{E}\Bigg[ \sup_{s\in [0, t\wedge \tau_R^n]} \|u_n(s)\|_{L^2(D)}^{\frac{p(p-2)}{p-1+\theta}} 
      \Big(\int_{0}^{t\wedge \tau_{R}^{n}} \int_{\pmb{E}} \|\eta(u_n(s);z)\|_{L^2(D)}^{\frac{2p}{p-1+\theta}}\Big)\Bigg ]\notag\\
     & \le C + \epsilon_{2} \mathbb{E}  \bigg[ \underset{s \in [0,t \wedge \tau_{R}^{n}]}\sup  \|u_n(s) \|_{L^2(D)}^{p} \bigg] + C_{p,\epsilon_{2}} \mathbb{E} \bigg[\int_{0}^{t\wedge \tau_{R}^{n}} \| u_n(s)\|_{L^2(D)}^{\frac{2p\theta}{p-1+\theta}} \,{\rm d}s \bigg]^{\frac{p-1+\theta}{1+\theta}} \\
     & \le C +  \epsilon_{2} \mathbb{E}  \bigg[ \underset{s \in [0,t \wedge \tau_{R}^{n}]}\sup  \|u_n(s) \|_{L^2(D)}^{p} \bigg] + C_{p,\epsilon_{2}} \int_{0}^{t}\mathbb{E} \bigg[  \underset{r \in [0,s \wedge \tau_{R}^{n}]}\sup  \|u_n(r) \|_{L^2(D)}^{p} \bigg] \,{\rm d}s.
 \end{align*}
 Similarly, we have
 \begin{align}
     \mathcal{K}_{3,2} \le C +  \epsilon_{3} \mathbb{E}  \bigg[ \underset{s \in [0,t \wedge \tau_{R}^{n}]}\sup  \|u_n(s) \|_{L^2(D)}^{p} \bigg] + C_{p,\epsilon_{3}} \int_{0}^{t}\mathbb{E} \bigg[  \underset{r \in [0,s \wedge \tau_{R}^{n}]}\sup  \|u_n(r) \|_{L^2(D)}^{p} \bigg] \,{\rm d}s. \notag
 \end{align}
 Now, we use \eqref{est1} along with Young's inequality and the fact $\frac{p\theta}{p-1+\theta} < 1$ to have
 \begin{align}
     \mathcal{K}_{3,3} \le C_{T_p} +  C \mathbb{E} \bigg[\int_{0}^{t\wedge \tau_{R}^{n}} \| u_n(s)\|_{L^2(D)}^{\frac{p^{2}\theta}{p-1+\theta}} \,{\rm d}s \bigg] \le C_{T_p} + \int_{0}^{t}\mathbb{E} \bigg[  \underset{r \in [0,s \wedge \tau_{R}^{n}]}\sup  \|u_n(r) \|_{L^2(D)}^{p} \bigg] \,{\rm d}s. \notag
 \end{align}
 Similarly, we get
 \begin{align}
     \mathcal{K}_{3,4}  \le C_{T_p} + \int_{0}^{t}\mathbb{E} \bigg[  \underset{r \in [0,s \wedge \tau_{R}^{n}]}\sup  \|u_n(r) \|_{L^2(D)}^{p} \bigg] \,{\rm d}s. \notag
 \end{align}
Hence we have
 \begin{align}
\mathcal{K}_3 \le C + (\epsilon_2+ \epsilon_3)  \mathbb{E}  \bigg[ \underset{s \in [0,t \wedge \tau_{R}^{n}]}\sup  \|u_n(s) \|_{L^2(D)}^{p} \bigg] + C(\epsilon_2, \epsilon_3, p,\theta) \int_{0}^{t}\mathbb{E} \bigg[  \underset{r \in [0,s \wedge \tau_{R}^{n}]}\sup  \|u_n(r) \|_{L^2(D)}^{p} \bigg] \,{\rm d}s\,. \label{esti:k3}
 \end{align}
 Putting the inequalities \eqref{esti:k2} and \eqref{esti:k3} in \eqref{eq:est-p1} and choose $\epsilon_{1},\epsilon_{2},\epsilon_{3} > 0$ such that $1-\epsilon_{1}-\epsilon_{2}-\epsilon_{3} > 0$, we arrive at
 \begin{align}
     \mathbb{E}  \bigg[\underset{s \in [0,t \wedge \tau_{R}^{n}]}\sup  \| u_n(s)\|_{L^2(D)}^{p}  \bigg]   \le  C \bigg( 1 + \| u_0\|_{L^2(D)}^{\frac{p^2}{p-1+\theta}} \bigg) + C \int_{0}^{t}\mathbb{E} \bigg[  \underset{s \in [0,t \wedge \tau_{R}^{n}]}\sup  \|u_n(s) \|_{L^2(D)}^{p} \bigg] \,{\rm d}s. \notag
 \end{align}
 Using the Gr{\"o}nwall's inequality we obtain that
 \begin{align}
     \mathbb{E}  \bigg[\underset{s \in [0,T_p \wedge \tau_{R}^{n}]}\sup  \| u_n(s)\|_{L^2(D)}^{p}  \bigg] + \frac{p}{2} \mathbb{E} \int_{0}^{T_p \wedge \tau_{R}^{n}} \|u_n(s) \|_{L^2(D)}^{p-2} \|u_n(s) \|_{W_{0}^{1,2}(D)}^{2} \,{\rm d}s \le C \bigg( 1 + \| u_0\|_{L^2(D)}^{\frac{p^2}{p-1+\theta}} \bigg). \notag  
 \end{align}
Hence, the assertion follows once we send $ R \rightarrow \infty$ and use Fatou's lemmma.
\end{proof}

We would like to have a strong convergence of $\{u_n\}$ in some appropriate Banach space. The uniform bound \eqref{esti:uniform-galerkin} guarantees only weak convergence of the sequence $\{u_n\}$. To get strong convergence, one may derive the uniform bound of $\{u_n\}$ in some appropriate fractional Sobolev space and use compactness theorem as in \cite[Theorem 2.1]{Flandoli-1995}. 
To proceed further, we need some preparation. Assume that $(\mathbb{Z}, \left \| \cdot \right \|_{\mathbb{Z}})$ be a separable metric space. For any $r > 1$ and $\alpha \in (0,1)$, let $W^{\alpha,r}([0,T];\mathbb{Z})$ be the Sobolev space of all functions $u \in L^{r}([0,T];\mathbb{Z})$ such that
\begin{align} \label{norm def1}
   \int_{0}^{T} \int_{0}^{T} \frac{\left \| u(t) - u(s) \right \|_{\mathbb{Z}}^{r}}{\left | t-s \right |^{1+\alpha r}} \, {\rm d}t \, {\rm d}s < + \infty \,. 
\end{align}
with the norm 
\[ \left \| u \right \|_{W^{\alpha,r}([0,T];\mathbb{Z})}^{r} = \int_{0}^{T}  \left \| u(t) \right \|_{\mathbb{Z}}^{r} \, {\rm d}t + \int_{0}^{T} \int_{0}^{T} \frac{\left \| u(t) - u(s) \right \|_{\mathbb{Z}}^{r}}{\left | t-s \right |^{1+\alpha r}} \, {\rm d}t \, {\rm d}s.  \]
\begin{lem}\label{lem:uniform-frac-galerkin}
     The following estimation holds: for $\alpha \in (0, \frac{1}{2})$ and $p \in (1,2)$
     \begin{align*}
         \underset{n\in \mathbb{N}}\sup\, \mathbb{E} \left\{ \| u_n\|_{W^{\alpha,p}([0,T_2];{W^{-1,2}(D)})}^p  \right\}  < \infty,
     \end{align*}
     where $T_2$ is given in Lemma \ref{lem:est-galerkin}.
\end{lem}
\begin{proof}
From Lemma \ref{lem:est-galerkin}, we see that 
$$ \sup_{n\in \mathbb{N}}\mathbb{E}\Big[ \int_0^{T_2} \|u_n(s)\|_{W_0^{1,2}(D)}^2\,{\rm d}s\Big] \le C_{2,\theta} \big(1 + \|u_0\|_{L^2(D)}^\frac{4}{1+\theta}\big) < +\infty.$$
Hence, thanks to the Sobolev embedding $W_0^{1,2}(D)\hookrightarrow W^{-1, 2}(D)$, Young's inequality and the above estimate, we get, for $1<p<2$,
\begin{align*}
& \mathbb{E}\Big[\int_0^{T_2} \|u_n(s)\|_{W^{-1,2}(D)}^p\,{\rm d}s\Big] \le C \mathbb{E}\Big[\int_0^{T_2} \|u_n(s)\|_{W^{1,2}_0(D)}^p\,{\rm d}s\Big] \le C \Big\{ 1+ \mathbb{E}\Big[\int_0^{T_2} \|u_n(s)\|_{W_0^{1,2}(D)}^2\,{\rm d}s\Big]\Big\} \le C\,.
\end{align*}
Note that $u_n(\cdot)$ satisfies the following integral equation. 
        \begin{align*}
        u_n(t) & = u_n(0) + \int_{0}^{t} \Delta u_n(s) \, {\rm d}s + \int_{0}^{t} P_n[u_n(s) \log |u_n(s)|]\, {\rm d}s + \int_{0}^{t}  \int_{\pmb{E}} P_n[\eta(u_n(s);z)] \,\widetilde{N}({\rm d}z,{\rm d}s)\notag \\
        &\equiv \sum_{i=0}^3\mathcal{K}_i^n(t).
    \end{align*}
W.L.O.G.\,, we assume that $s < t$. $\mathcal{K}_{0}^{n}(\cdot)$ fulfills the estimation \eqref{norm def1} for any  $\alpha \in (0,\frac{1}{2})$ as it is independent of time. Following the line of argument as invoked in the proof of \cite[Lemma $5.2$]{Shang-2022}, we get
\begin{align*}
  \mathbb{E} \left[\int_{0}^{T_2}\int_{0}^{T_2} \frac{\left \|\mathcal{K}_{i}^n(t) - \mathcal{K}_{i}^n(s)   \right \|_{W^{-1,2}(D)}^{2}}{\left | t-s \right |^{1+ \alpha p}} \, {\rm d}t \, {\rm d}s\right] \leq C_i\quad \text{for}~~i=1,2\,. 
\end{align*}
Using the assumption \ref{A2}, the maximal inequality \cite{Liu-2019} and the uniform estimate \eqref{esti:uniform-galerkin}, we have
\begin{align*}
     &\mathbb{E}\Big[ \| \mathcal{K}_{3}^n(t) -\mathcal{K}_{3}^n(s)\|_{W^{-1,2}(D)}^{p}\Big]    = \mathbb{E} \left \| \int_{s}^{t} \int_{\pmb{E}} P_n[\eta(u_n(r);z)]  \widetilde{N}({\rm d}z,{\rm d}r)  \right \|_{W^{-1,2}(D)}^{p} \\
     & \le C_p \, \mathbb{E} \Bigg[\bigg( \int_{s}^{t} \int_{\pmb{E}} \| \eta(u_n(r);z)\|_{L^2(D)}^{2}  m({\rm d}z){\rm d}r \bigg)^{\frac{p}{2}}\Bigg]
      \le C\,  \mathbb{E} \Bigg[\bigg( \int_{s}^{t} \big( 1+ \| u_n(r)\|_{L^2(D)}^{2 \theta}\big)  {\rm d}r \bigg)^{\frac{p}{2}}\Big] \notag \\
     & \le C\,  \mathbb{E} \bigg[ 1 + \underset{r \in [0,T_2]}\sup \| u_n(r) \|_{L^2(D)}^{2} \bigg] (t-s)^{\frac{p}{2}}.
\end{align*}
Hence 
\begin{align*}
  \mathbb{E} \left[\int_{0}^{T_2}\int_{0}^{T_2} \frac{\left \|\mathcal{K}_{3}^n(t) - \mathcal{K}_{3}^n(s)   \right \|_{W^{-1,2}(D)}^{2}}{\left | t-s \right |^{1+ \alpha p}} \, {\rm d}t \, {\rm d}s\right] \leq C\,. 
\end{align*}
We combine all the above estimates to arrive at the required assertion. 
\end{proof}
\subsection{ Tightness of the sequence $\{\mathcal{L}(u_n)\}$:} We now discuss the tightness of the law of the sequence $\{u_n\}$, denoted by $\{\mathcal{L}(u_n)\}$, on some appropriate spaces of c\`{a}dl\`{a}g functions. We first recall the following well-known lemma. 
\begin{lem} \cite[Theorem 2.1]{Flandoli-1995}
\label{lem:cpt}
Let $\mathbb{X} \subset \mathbb{Y} \subset \mathbb{X^*}$ be Banach spaces, $\mathbb{X}$ and $\mathbb{X^*}$ reflexive, with 
compact embedding of $\mathbb{X}$ in $\mathbb{Y}$. For any $q \in (1,\infty)$ and $\alpha \in (0,1)$, the  embedding of  $L^{q}([0,T];\mathbb{X}) \cap W^{\alpha,q}([0,T];\mathbb{X^*}) $  equipped with natural norm in $L^{q}([0,T];\mathbb{Y})$ is compact.
\end{lem}
In view of Lemma \ref{lem:cpt} along with the estimations in Lemmas \ref{lem:est-galerkin} and \ref{lem:uniform-frac-galerkin}, one case use similar line of argument as invoked in \cite[Corollary 4.2]{Kavin-2024}  to conclude that $\left \{ \mathcal{L}(u_n) \right \}$ is tight on $L^{p}([0,T_2];L^{2}(D))$ for any $p\in (1,2)$; see also in \cite[Corollary $4.8$ ]{Kavin-2023LDP}.
\vspace{0.1cm}

To pass to the limit in the Galerkin approximations, tightness of $\left \{ \mathcal{L}(u_n) \right \}$ on the space $L^{p}([0,T_2];L^{2}(D))$ is not enough. We may require for example the tightness of $\left \{ \mathcal{L}(u_n) \right \}$ on the space 
of c\`{a}dl\`{a}g functions $u:[0,T_2]\goto W^{-1,2}(D)$ with extended Skorokhod topology \cite{Skorohod-56,Billingsley-99}. Taking motivation from \cite{Brzezniak-2013,Majee-2020, Kavin-2023nonlinear, Metivier-1988} and looking at the estimations in Lemma \ref{lem:est-galerkin}, we show the tightness of $\{\mathcal{L}(u_n)\}$ on the functional space
\begin{align*}
 \mathcal{Z}_{T_2}:= \mathbb{D}([0,T_2]; W^{-1,2}(D))\cap \mathbb{D}([0,T_2]; L^2_{w}(D))\cap L^2_w([0,T_2]; W^{1,2}(D)) \cap L^2([0,T_2]; L^2(D))
\end{align*}
  equipped with the supremum topology $\mathcal{T}_{T_2}$ ; see \cite{Majee-2023} for its notations. To proceed further, we recall the definition of Aldous condition. 

\begin{defi}[Aldous condition]\label{defi:aldous-condition}
Let  $(\mathbb{B},\|\cdot\|_{\mathbb{B}})$ be a separable Banach space. We say that a sequence of $\{\mathcal{F}_t\}_{0\le t\le T}$-adapted, $\mathbb{B}$-valued c\`{a}dl\`{a}g processes $({\tt Z}_n)_{n\in \mathbb{N}}$ satisfy the {\it Aldous condition}
if for every $\delta,~m>0$, there exists $\eps>0$ such that
\begin{align*}
 \sup_{n\in \mathbb{N}}\sup_{0< \beta \le \eps} \mathbb{P}\big\{\|{\tt Z}_n(\kappa_n + \beta)-{\tt Z}_n(\kappa_n)\|_{\mathbb{B}}\ge m\big\} \le \delta\,
\end{align*} 
holds for every $\{\mathcal{F}_t\}$-adapted sequence of stopping times $(\kappa_n)_{n\in \mathbb{N}}$  with 
$\kappa_n\le T$.
\end{defi}
 We mention a sufficient condition to satisfy Aldous condition; cf.~ \cite[Lemma 9]{Motyl-2013}. 
\begin{lem}
 The sequence $({\tt Z}_n)_{n\in \mathbb{N}}$ as mentioned in Definition \ref{defi:aldous-condition} be a sequence 
 of c\`{a}dl\`{a}g stochastic processes satisfies the Aldous condition if there exist positive constants $\alpha,\, \zeta$ and $C$ such that for any $\beta >0$,
\begin{align}
 \mathbb{E}\Big[\|{\tt Z}_n(\kappa_n + \beta)-{\tt Z}_n(\kappa_n)\|^\alpha_{\mathbb{B}}\Big] \le C \beta^\zeta \label{inq:Aldous-condition}
\end{align}
holds for every sequence of $\{\mathcal{F}_t\}$-stopping times $(\kappa_n)$ with $\kappa_n\le T$.
\end{lem}
Following the line arguments of \cite[Lemma $3.3$]{Brzezniak-2013}, \cite[Theorem $2$, Lemma $7$,~$\&$~Corollary $1$]{Motyl-2013}, we now state~(without proof) tightness criterion of the family of laws of some process in $\mathcal{Z}_{T_2}$.
\begin{thm}\label{thm:for-tightness}
 Let $({\tt X}_{n})_{n \in \mathbb{N}}$ be a sequence of $\mathbb{F}$-adapted, $W^{-1,2}(D)$-valued c\`{a}dl\`{a}g stochastic process such that 
  \begin{itemize}
  \item[(i)]
   $({\tt X}_{n})_{n \in \mathbb{N}}$ satisfies the Aldous condition in $W^{-1,2}(D)$,
 \item[(ii)] for some constant $C>0$, 
  \begin{align*}
   \sup_{n \in \mathbb{N}} \mathbb{E}\Big[ \sup_{t\in [0,T_2]}\|{\tt X}_{n}(t)\|_{L^2(D)}\Big] \le C, \quad 
   \sup_{n \in \mathbb{N}} \mathbb{E}\Big[\int_0^{T_2}\|{\tt X}_{n}(t)\|_{W_0^{1,2}(D)}^2\,{\rm d}t\Big] \le C\,.
  \end{align*}
  \end{itemize}
Then the sequence  $\big\{\mathcal{L}({\tt X}_{n})\big\}_{n \in \mathbb{N}}$ is tight on $(\mathcal{Z}_{T_2}, \mathcal{T}_{T_2})$ .
\end{thm}

\begin{lem}\label{lem:tightness}
 The sequence $\big\{\mathcal{L}(u_n)\big\}_{n \in \mathbb{N}}$ is tight on $(\mathcal{Z}_{T_2}, \mathcal{T}_{T_2})$. 
\end{lem}
\begin{proof} We follow the same line of argument as invoked in \cite{Majee-2020}. By Lemma \ref{lem:est-galerkin}, the assertion ${\rm (ii)}$ of Theorem \ref{thm:for-tightness} holds true for the sequences $\{u_n\}_{n \in \mathbb{N}}$. Thus, it remains to show that $\{u_n\}_{n \in \mathbb{N}}$ satisfies the Aldous condition in $W^{-1,2}(D)$. Let $\big\{\kappa_m \big\}_{m \in \mathbb{N}}$ be a sequence of stopping times such that $0 \le \kappa_m \le T_2$.  By \eqref{eq:finite-dim}, we have 
\begin{align}
 u_n(t)&= P_n u_0 + \int_0^t  \Delta u_n(s) \,{\rm d}s + \int_0^t P_n\left[u_n(s) \log \left|u_n(s)\right|\right] \,{\rm d}s + \int_0^t \int_{\pmb{E}} P_n [\eta(u_n(s);z)] \widetilde{N}({\rm d}z,{\rm d}s)  \notag \\
  &\equiv  {\tt T}_1^{n}(t) + {\tt T}_2^{n}(t) + {\tt T}_3^{n}(t) + {\tt T}_4^{n}(t)\,. \label{eq:time-cont-p-laplace}
\end{align}
We show that each term in \eqref{eq:time-cont-p-laplace} satisfies \eqref{inq:Aldous-condition} for certain choices of
$\alpha$ and $\zeta$. Clearly, $ {\tt T}_1^{n}(\cdot) $ satisfies \eqref{inq:Aldous-condition} for any $\alpha, \zeta$.
For ${\tt T}_2^{n}(t)$, we have, for any  $\beta >0$
\begin{align*}
& \mathbb{E}\Big[  \big\| {\tt T}_2^{n}(\kappa_m + \beta)-{\tt T}_2^{n}(\kappa_m)\big\|_{W^{-1,2}(D)}\Big] 
 \le  C  \mathbb{E}\Big[ \int_{\kappa_m}^{\kappa_m + \beta} \|u_n(s) \big\|_{W_{0}^{1,2}(D)}\,{\rm d}s\Big] \notag \\
 & \le C \beta^{\frac{1}{2}} \Big( \mathbb{E}\Big[ \int_0^{T_2} \|u_n(s)\|_{W_{0}^{1,2}(D)}^2\,{\rm d}s\Big] \Big)^{\frac{1}{2}} \le C  \beta^{\frac{1}{2}}\,,
\end{align*}
where $(\kappa_m)$ is a sequence of $\mathbb{F}$- stopping times  with $\kappa_m \le T_2$.
Let us recall the following embedding: 
\begin{align}\label{embedding-1}
L^{q^*}(D)\hookrightarrow W^{-1,2}(D),~\text{where}~q^* \in \begin{cases} [\frac{2d}{d+2},2)~~\text{if}~d>2\,, \\
(1,2)~~\text{if}~d=1,2\,.
\end{cases}
\end{align}
Note that, for any given $\epsilon >0$, there exists a constant $C(\epsilon)$ such that for any $a\ge 0$
\begin{align}
\big|a \log |a|\big| \le C(\epsilon) \big(1+ a^{1+\epsilon} \big)\,. \label{inq:important-log}
\end{align}
 Using \eqref{embedding-1}, \eqref{inq:important-log} with $\epsilon >0$ satisfying $(1+\eps)\le 2$ and $(1+\eps)q^*\le 2$, and H\"{o}lder's inequality,  we have 
\begin{align*}
& \mathbb{E}\Big[  \big\| {\tt T}_3^{n}(\kappa_m + \beta)-{\tt T}_3^{n}(\kappa_m)\big\|_{W^{-1,2}(D)}\Big] 
 \le \mathbb{E}\Big[\int_{\kappa_m}^{\kappa_m + \beta}\big\| u_n(s) \log |u_n(s)| \big\|_{L^{q^*}(D)}\,{\rm d}s\Big] \\
 & \le C\,\mathbb{E}\Big[ \int_{\kappa_m}^{\kappa_m + \beta} \Big( \int_{D}\{ 1 + |u_n(s)|^{(1+\epsilon)q^*}\}\,{\rm d}x\Big)^\frac{1}{q^*}\,{\rm d}s\Big] 
  \le C\, \mathbb{E}\Big[\int_{\kappa_m}^{\kappa_m + \beta} \Big( 1 + \Big\{ \int_D |u_n(s)|^2\,{\rm d}x\Big\}^{1+\eps}\Big)\,{\rm d}s\Big] \notag \\
 & \le C\, \Big( 1 + \sup_n\mathbb{E}\Big[\sup_{t\in [0,T_2]}\|u_n(t)\|_{L^2(D)}^2\Big]\Big)\beta\,.
\end{align*}
Thus, \eqref{inq:Aldous-condition} is satisfied by  ${\tt T}_3^{n}(t)$ for $\alpha=1$ and $\zeta=1$.
Furthermore, by using the embedding  $W_0^{1,2}\hookrightarrow L^2(D)\hookrightarrow W^{-1,2}(D)$, It\^{o}-L\'{e}vy isometry and the assumption \ref{A3}, we obtain
\begin{align*}
 & \mathbb{E}\Big[\big\| {\tt T}_4^{n}(\kappa_m +\beta)-{\tt T}_4^{n}(\kappa_m)\big\|_{W^{-1,2}}^2\Big] 
  \le C \mathbb{E}\Big[\Big\|  \int_{\kappa_m}^{\kappa_m + \beta} \int_{\pmb{E}}  P_n [\eta(u_n;z)] \widetilde{N}({\rm d}z,{\rm d}s)\Big\|_{L^2(D)}^2\Big] \\
 & \le C\, \mathbb{E}\Big[ \int_{\kappa_m}^{\kappa_m + \beta} \int_{\pmb{E}} \big\|  P_n [\eta(u_n;z)]\big\|_{L^2(D)}^2 \,m({\rm d}z)\,{\rm d}s\Big]  \le C\, \mathbb{E}\Big[ \int_{\kappa_m}^{\kappa_m + \beta}\big( 1+ \| u_n(s)\|_{L^2(D)}^{2 \theta}\big)\,{\rm d}s\Big] \notag \\
 &\le C\, \Big( 1 + \sup_n\mathbb{E}\Big[\sup_{t\in [0,T_2]}\|u_n(t)\|_{L^2(D)}^2\Big]\Big)\beta\,.
\end{align*}
This shows that ${\tt T}_4^{\kappa}(t)$ satisfies \eqref{inq:Aldous-condition} for $\alpha=2$ and $\zeta=1$.
Thus, the sequence $\big\{\mathcal{L}(u_n)\big\}_{n \in \mathbb{N}}$ is tight on $(\mathcal{Z}_{T_2}, \mathcal{T}_{T_2})$.
\end{proof}

Let $M_{\bar{\mathbb{N}}}(\pmb{E} \times [0,T])$ be set of all $\bar{\mathbb{N}}:=\mathbb{N}\cup \{ \infty\}$-valued measures on $(\pmb{E} \times[0,T],\mathcal{B}(\pmb{E}\times[0,T]))$ endowed with the $\sigma$-field $\mathcal{M}_{\bar{\mathbb{N}}}(\pmb{E} \times [0,T])$ generated by the projection maps 
$i_B: M_{\bar{\mathbb{N}}}(\pmb{E} \times [0,T])\ni \mu\mapsto \mu(B)\in \bar{\mathbb{N}}\quad \forall B \in \mathcal{B}(\pmb{E} \times [0,T])$. Observe that $M_{\bar{\mathbb{N}}}(\pmb{E} \times [0,T])$ is a separable metric space. For $n>0$, define $N_{n}({\rm d}z,{\rm d}t):=N({\rm d}z,{\rm d}t)$. Then, by \cite[Theorem 3.2]{Parthasarathy1967}, we arrive at the following conclusion.
\begin{lem} \label{lem:tightness-levy}
The laws of the family  $\{ N_{n}({\rm d}z,{\rm d}t)\}$ is tight on $M_{\bar{\mathbb{N}}}(\pmb{E} \times [0,T_2])$.
\end{lem}
\subsection{Existence of weak solution of \eqref{eq:log-nonlinear} on the interval $[0,T_2]$}\label{subsec:existence-weak-local}
In this subsection, we show the existence of a weak solution of the underlying problem \eqref{eq:log-nonlinear} on the interval $[0,T_2]$. 
\vspace{0.2cm}

Thanks to Lemmas \ref{lem:tightness}-\ref{lem:tightness-levy}, we infer that the family of the laws of sequence $\big\{(u_n, N_{n}): n \in \mathbb{N}\big\}$ is tight on $\mathcal{X}_{T_2}$, where the space $\mathcal{X}_{T_2}$ is defined as 
$$ \mathcal{X}_{T_2}:=\big[ L^{p}([0,T_2];L^{2}(D)) \cap \mathcal{Z}_{T_2} \big] \times  M_{\bar{\mathbb{N}}}(\pmb{E} \times [0,T]), \quad p \in (1,2).$$ We then apply Jakubowski's version of Skorokhod theorem together with \cite[Corollary 2]{Motyl-2013}, and \cite[Theorem $D1$]{Erika-2009} to arrive at the following proposition. 
\begin{prop}\label{prop:skorokhod-representation}
There exist a new probability space $(\bar{\Omega}, \bar{\mathcal{F}}, \bar{\mathbb{P}})$, a subsequence of $\{n\}$, still denoted by $\{n\}$, and $\mathcal{X}_{T_2}$-valued 
random variables $(u_{n}^*, N_{n}^*)$ and $(u_*,N_*)$ defined on  $(\bar{\Omega}, \bar{\mathcal{F}}, \bar{\mathbb{P}})$ satisfying the followings:
\begin{itemize}
 \item [a).] $\mathcal{L}({u}_{n}^*,N_{n}^*)=\mathcal{L}(u_n,N_{n})$ for all $n \in \mathbb{N}$,
 \item [b).] $({u}_{n}^*, N_{n}^*)\goto ({u}_*, N_*)$ in $\mathcal{X}_{T_2}\quad \bar{\mathbb{P}}$-a.s. as $n \goto \infty$,
 \item [c).] $N_{n}^*(\bar{\omega})= N_*(\bar{\omega})$ for all $\bar{\omega}\in \bar{\Omega}$.
\end{itemize}
\end{prop}
Moreover, there exists a sequence of  perfect functions 
$\phi_{n}:\bar{\Omega}\to\Omega$ such that
\begin{align}
  {u}_{n}^*=u_n\circ\phi_{n}\,, \quad \mathbb{P}=\bar{\mathbb{P}}\circ \phi_{n}^{-1}\,; \label{eq:perfect-function}
 \end{align}
 see e.g., \cite[Theorem $1.10.4$ \&\ Addendum $1.10.5$]{Wellner}. Furthermore, by \cite[Section 9]{Erika-2009} we infer that $N_{n}^*$ and $N_*$ are time-homogeneous Poisson random measures on $\pmb{E}$ with the intensity measure $m({\rm d}z)$ over the stochastic basis
$(\bar{\Omega}, \bar{\mathcal{F}}, \bar{\mathbb{P}},\bar{\mathbb{F}})$, where  the  filtration
 $\bar{\mathbb{F}}:= \big( \bar{\mathcal{F}}_t\big)_{t\in [0,T_2]}$ is defined by
\begin{align*}
 \bar{\mathcal{F}}_t:= \sigma\big\{({u}_{n}^*(s), \,N_{n}^*(s),\, {u}_*(s), N_*(s)): 0\le s\le t\big\}, \quad t\in [0,T_2]. 
\end{align*}
With the help of \eqref{eq:perfect-function} and Proposition \ref{prop:skorokhod-representation}, we see that $u_n^*$ satisfies the following integral equation in $W^{-1,2}(D)$: for all $t\in [0,T_2]$,
\begin{align}
        u_n^*(t) & = P_n u_0 + \int_{0}^{t} \Delta u_n^*(s) \, {\rm d}s + \int_{0}^{t} P_n[u_n^*(s) \log |u_n^*(s)|]\, {\rm d}s + \int_{0}^{t}  \int_{\pmb{E}} P_n[\eta(u_n^*(s);z)] \,\widetilde{N}_*({\rm d}z,{\rm d}s)\,. \label{eq:galerkin-new-prob}
    \end{align}

  Following the same calculations as invoked in the proof of Lemma \ref{lem:est-galerkin}, we derive the following uniform bounds
  for $\{u_n^*\}$: for $r\in (1,2)$
 \begin{equation}\label{a-priori-3}
 \begin{aligned}
& \underset{n}\sup\, \bar{\mathbb{E}}\bigg[  \underset{t\in [0,T_2]}\sup\,\| {u}_{n}^*(t)\|_{L^2(D)}^2 + \int_{0}^{T_2} \| {u}_{n}^*(s)\|_{W_0^{1,2}(D)}^2 \,{\rm d}s  \bigg] \le C, \\
& \bar{\mathbb{E}}\bigg[ \int_0^{T_2} \| P_n[u_n^*(t)\log(|u_n^*(t)|)]\|_{W^{-1,2}(D)}^r\,{\rm d}t\bigg]\le C \Bigg( 1+  
\bar{\mathbb{E}}\bigg[ \underset{t\in [0,T_2]}\sup\,\| {u}_{n}^*(t)\|_{L^2(D)}^2\bigg]\Bigg)\,, \\
& \bar{\mathbb{E}}\bigg[ \int_0^{T_2}\| \Delta u_n^*(t)\|_{W^{-1,2}(D)}^2\,{\rm d}t\bigg] \le C \bar{\mathbb{E}}\bigg[ \int_0^{T_2}\| u_n^*(t)\|_{W^{1,2}_0(D)}^2\,{\rm d}t\bigg]\,.
\end{aligned}
\end{equation}
By employing the fact that $\|\cdot \|_{L^2(D)}$ and $\|\cdot \|_{W_{0}^{1,2}(D)}$ are lower semi-continuous on $W^{-1,2}(D)$, $P_n$ is a projection operator from $W^{-1,2}(D)$ to $L_n$, together with Fatou’s lemma, the uniform bound \eqref{a-priori-3},  and Proposition \ref{prop:skorokhod-representation}, we get
\begin{align} \label{eq:cgt-lim-est1}
    & \bar{\mathbb{E}}\bigg[ \underset{t\in [0,T_2]}\sup\,\| {u}_{*}(t)\|_{L^2(D)}^2 \bigg]  \le \bar{\mathbb{E}}\bigg[ \underset{t\in [0,T_2]}\sup\, \underset{m \rightarrow \infty}\liminf \, \| P_m{u}_*(t)\|_{L^2(D)}^2 \bigg]\notag \\
    & \le \underset{m \rightarrow \infty}\liminf \, \underset{n \rightarrow \infty}\liminf \, 
 \bar{\mathbb{E}}\bigg[\underset{t\in [0,T]}\sup\, \|P_m {u}_{n}^*(t)\|_{L^2(D)}^2 \bigg]  \le \underset{n}\sup\, \bar{\mathbb{E}}\Big[\underset{t\in [0,T_2]}\sup\,\| {u}_{n}^*(t)\|_{L^2(D)}^2 \Big]< + \infty\,.
\end{align}
Similarly, we have
\begin{align}\label{eq:cgt-lim-est2}
    \bar{\mathbb{E}}\bigg[ \int_{0}^{T_2} \| {u}_{*}(s)\|_{W_0^{1,2}(D)}^2 \,{\rm d}s  \bigg] < \infty\,.
\end{align}

In view of Proposition \ref{prop:skorokhod-representation} and the uniform moment estimate \eqref{a-priori-3}, we arrive at the following lemma.
\begin{lem}\label{lem:conv-martingale-2}
There exists a subsequence of $\{{u}_{n}^*\}$, still denoted by $\{{u}_{n}^* \}$ and $\{\bar{\mathcal{F}}_t\}$- adapted process 
${u}_{*} \in  \mathbb{D}([0,T_2]; W^{-1,2}(D))\cap \, L^{2}([0,T];W_{0}^{1,2}(D)) \, \cap \, L^{\infty}([0,T];L^{2}(D))$ such that
\begin{itemize}
\item[(a)] $u_n^*(\bar{\omega},t,x)\rightarrow u_* (\bar{\omega},t,x)~~\text{for $\bar{\mathbb{P}}$-a.s., and for a.e. $(t,x)\in [0,T_2]\times D$}$.
\item[(b)] ${u}_{n}^* \rightarrow {u}_{*}$ in $L^r\big(\bar{\Omega}; L^r([0,T_2];L^{2}(D))$,~~$r\in (1,2)$,
\item[(c)]  ${u}_{n}^* \rightharpoonup {u}_{*} $ in $L^2\big(\bar{\Omega}; L^2([0,T_2];W_0^{1,2}(D)),$
\item[(d)]  $\Delta{u}_{n}^* \rightharpoonup \Delta{u}_{*} $ in $ L^2\big(\bar{\Omega}; L^2([0,T_2];W^{-1,2}(D)),$
\item[(e)] $P_n\left[{u}_{n}^*\log \left|{u}_{n}^*\right|\right]\rightarrow {u}_{*} \log |{u}_{*}| $ in $L^r\big(\bar{\Omega}; L^r([0,T_2];W^{-1,2}(D))$. 
\end{itemize}
\end{lem}
We now discuss about the convergence of the stochastic integral. Let $L^2((\Gamma,\mathcal{G}, \tau);\R)$ denote the square integrable predictable integrands for It\^{o}-L\'{e}vy integrals with respect to the compensated Poisson random measure $\widetilde{N}_*({\rm d}z,{\rm d}s)$, where 
$$ \Gamma=\bar{\Omega}\times [0,T_2]\times \pmb{E},~~~\mathcal{G}=\mathcal{P}_{T_2}\times \mathcal{B}(\pmb{E}),~~~\tau=\bar{\mathbb{P}}\otimes\lambda_t \otimes m({\rm d}z)$$ with $\mathcal{P}_{T_2}$ being the predictable $\sigma$-field on $\bar{\Omega}\times [0,T_2]$ with respect to 
$\{\bar{\mathcal{F}}_t\}$ and $\lambda_t$ the Lebesgue measure on $[0,T_2]$. Since the It\^{o}-L\'{e}vy integral from $L^2((\Gamma,\mathcal{G}, \tau);\R)$ to $L^2((\bar{\Omega}, \bar{\mathcal{F}}_{T_2});\R)$ is a linear isometry operator and any isometry between two Hilbert spaces preserves the weak convergence, one can easily see that any weakly converging sequence of integrands $ \{\chi_n(t,z)\} \subset L^2((\Gamma,\mathcal{G}, \tau);\R)$, the corresponding sequence of It\^{o}-L\'{e}vy integrals with respect to $\widetilde{N}_*({\rm d}z,{\rm d}s)$ also converge weakly in $L^2((\bar{\Omega}, \bar{\mathcal{F}}_{T_2});\R)$. 

Thanks to the sublinear growth of $\eta$, the uniform moment estimate \eqref{a-priori-3} and $a)$ of Lemma \ref{lem:conv-martingale-2}, we see that 
$\chi_n(t,x):= \langle P_n(\eta(u_n^*(t);z)), \phi \rangle$ for $\phi\in W_0^{1,2}(D)$ converges weakly to 
$\chi(t,z):=\langle \eta(u_*(t);z)), \phi \rangle$ in $L^2((\Gamma,\mathcal{G}, \tau);\R)$. As a consequence of the above discussion, the It\^{o}-L\'{e}vy integrals 
$$\int_{0}^{T_2}\int_{\pmb{E}} \chi_n(s,z)\widetilde{N}_*({\rm d}z,{\rm d}s) \rightharpoonup \int_{0}^{T_2}\int_{\pmb{E}} \chi(s,z)\widetilde{N}_*({\rm d}z,{\rm d}s)~~~\text{in}~~L^2((\bar{\Omega}, \bar{\mathcal{F}}_{T_2});\R).$$ In other words, we get
\begin{align}
 \int_{0}^{\cdot}\int_{\pmb{E}} P_n \eta({u}_{n}^*(s);z) \widetilde{N}_*({\rm d}z,{\rm d}s) \rightharpoonup \int_{0}^{\cdot}\int_{\pmb{E}} \eta({u}_{*}(s);z)\, \widetilde{N}_*({\rm d}z,{\rm d}s)~~\text{in}~\,L^{2}([0,T_2];L^{2}(\bar{\Omega},W^{-1,2}(D))\,. \label{conv:stoc-new-prob}
\end{align}

Thanks to Lemma \ref{lem:conv-martingale-2} and the convergence result in \eqref{conv:stoc-new-prob}, we pass to the limit 
in \eqref{eq:galerkin-new-prob} to conclude that $u_*$ is indeed a weak solution of \eqref{eq:log-nonlinear} on the interval $[0,T_2]$. Moreover, in view of  \eqref{eq:cgt-lim-est1} and \eqref{eq:cgt-lim-est2} and 
by \cite[Theorem 2]{Krylov-1982}, $u_*\in \mathbb{D}([0,T_2];L^2(D))$. 

\subsection{Uniqueness of solution of \eqref{eq:log-nonlinear}} \label{subsec:uniqueness}

In this subsection, we wish to show the path-wise uniqueness of solutions for \eqref{eq:log-nonlinear} via the standard $L^2$-contraction method. Let $u_1, u_2$ be two solutions of equation \eqref{eq:log-nonlinear} defined on the same stochastic basis 
$(\Omega, \mathcal{F}, \mathbb{P},\{\mathcal{F}_t\}, N)$. For any $R>0$ and $\delta \in (0,1]$, we define
\begin{align*}
\tau_R & :=\inf \left\{t>0:\|u_1(t)\|_{L^2(D)}^2 \vee \|u_2(t)\|_{L^2(D)}^2>R\right\}, \\
\tau_R^{\prime} & :=\inf \left\{t>0: \int_0^t\|u_1(s)\|_{W_0^{1,2}(D)}^2 \,{\rm d}s>R\right\} \wedge \inf \left\{t>0: \int_0^t\|u_2(s)\|_{W_0^{1,2}(D)}^2 \,{\rm d}s>R\right\}, \\
\tau^\delta & :=\inf \{t>0:\|u_1(t)-u_2(t)\|_{L^2(D)}>\delta\}, \\
\tau_R^\delta & :=\tau_R \wedge \tau_R^{\prime} \wedge \tau^\delta .
\end{align*}
We apply It\^{o}-L\'{e}vy formula to the function $x\mapsto \|x\|_{L^2(D)}^2$ on $u(t):=u_1(t)-u_2(t)$, and use the integration by parts formula to have
\begin{align}
 & \| u(t \wedge \tau_R^\delta)\|_{L^2(D)}^2 + 2 \int_{0}^{t \wedge \tau_R^\delta}  \| u(s)\|_{W_0^{1,2}(D)}^2 \,{\rm d}s \notag \\
 & = 2 \int_{0}^{t \wedge \tau_R^\delta} \big( u_1(s) \log |u_1(s)|- u_2(s)\log |u_2(s)|, u_1(s)-u_2(s) \big) \,{\rm d}s
 \notag \\
 & \quad + 2 \int_{0}^{t \wedge \tau_R^\delta} \int_{\pmb{E}} \Big ( \|  u(s) + \eta(u_1(s);z) - \eta(u_2(s);z) \|_{L^2(D)}^2 -  \|  u(s)\|_{L^2(D)}^2 \Big) \, \widetilde{N}({\rm d}z,{\rm d}s) \notag \\
 & \quad + \int_{0}^{t \wedge \tau_R^\delta} \int_{\pmb{E}} \| \eta(u_1(s);z) - \eta(u_2(s);z) \|_{L^2(D)}^2\, m({\rm d}z)\,{\rm d}s  \equiv \sum_{i=1}^3\mathcal{B}_i \,.\label{eq:uniqueness-1}
\end{align}

In view of the assumption \ref{A2}, one has
\begin{align}
\mathcal{B}_3 \le & 2K_{1}^{2} \int_{0}^{t \wedge \tau_R^\delta} \|u(s)\|_{L^2(D)}^2 \,{\rm d}s + \mathcal{B}_{3,1}\,, \notag
\end{align}
where $\mathcal{B}_{3,1}$ is given by 
\begin{align*}
\mathcal{B}_{3,1}:= 2 K_2^2 \int_0^{t \wedge \tau_R^\delta} \int_D|u_1(s, x)-u_2(s, x)|^2 \log _{+}(|u_1(s, x)| \vee|u_2(s, x)|) \,{\rm d}x \,{\rm d}s\,.
\end{align*}
By using Lemma \ref{lem:result-1} with $\epsilon=\frac{1}{4}$, we estimate $\mathcal{B}_1$ as follows.
\begin{align*}
\mathcal{B}_1 \le & \frac{1}{2} \int_0^{t \wedge \tau_R^\delta}\|u(s)\|_{W_0^{1,2}(D)}^2 \,{\rm d}s + 2(1+ \frac{d}{4} \log 4)  \int_0^{t \wedge \tau_R^\delta} \|u(s)\|_{L^2(D)}^2 \,{\rm d}s \notag \\
& + 2  \int_0^{t \wedge \tau_R^\delta} \|u(s)\|_{L^2(D)}^2 \log(\|u(s)\|_{L^2(D)})\,{\rm d}s \notag \\
& \qquad  + \frac{1}{(1-\alpha)e} \int_0^{t \wedge \tau_R^\delta} \big( 
\|u_1(s)\|_{L^2(D)}^{2(1-\alpha)} + \|u_2(s)\|_{L^2(D)}^{2(1-\alpha)}\big) \|u(s)\|_{L^2(D)}^{2\alpha} \,{\rm d}s\,.
\end{align*}
Again we use  Lemma \ref{lem:result-2} with $\epsilon=\frac{1}{4 K_2^2}$ to estimate $\mathcal{B}_{3,1}$:
\begin{align*}
\mathcal{B}_{3,1} \le & \frac{1}{2} \int_0^{t \wedge \tau_R^\delta}\|u(s)\|_{W_0^{1,2}(D)}^2 \,{\rm d}s + 2K_2^2 \frac{d}{4} \log (4K_2^2)  \int_0^{t \wedge \tau_R^\delta} \|u(s)\|_{L^2(D)}^2 \,{\rm d}s \notag \\
& + 2K_2^2  \int_0^{t \wedge \tau_R^\delta} \|u(s)\|_{L^2(D)}^2 \log(\|u(s)\|_{L^2(D)})\,{\rm d}s \notag \\
& \quad  + \frac{K_2^2}{(1-\alpha)e} \int_0^{t \wedge \tau_R^\delta} \big( 
\|u_1(s)\|_{L^2(D)}^{2(1-\alpha)} + \|u_2(s)\|_{L^2(D)}^{2(1-\alpha)}\big) \|u(s)\|_{L^2(D)}^{2\alpha} \,{\rm d}s \notag \\
& \qquad + \frac{K_2^2}{(1-\alpha)e} \int_0^{t \wedge \tau_R^\delta} \big( 4\lambda(D)\big)^{1-\alpha}\|u(s)\|_{L^2(D)}^{2\alpha}\,.
\end{align*}
Hence, from \eqref{eq:uniqueness-1}, we have 
\begin{align}
 & \| u(t \wedge \tau_R^\delta)\|_{L^2(D)}^2 +  \int_{0}^{t \wedge \tau_R^\delta}  \| u(s)\|_{W_0^{1,2}(D)}^2 \,{\rm d}s \notag \\
& \leq C \int_0^{t \wedge \tau_R^\delta}  \| u(s)\|_{L^2(D)}^2 \,{\rm d}s +C \int_0^{t \wedge \tau_R^\delta}\| u(s)\|_{L^2(D)}^2 \log \| u(s)\|_{L^2(D)} \,{\rm d}s \notag \\
& +\frac{1+K_2^2}{(1-\alpha) \mathrm{e}} \int_0^{t \wedge \tau_R^\delta}\left(\|u_1(s)\|_{L^2(D)}^{2(1-\alpha)}+\|u_2(s)\|_{L^2(D)}^{2(1-\alpha)}\right)\| u(s)\|_{L^2(D)}^{2 \alpha} \,{\rm d}s \notag \\
& +\frac{K_2^2}{(1-\alpha) \mathrm{e}} \int_0^{t \wedge \tau_R^\delta}(4 \lambda(D))^{1-\alpha}\| u(s)\|_{L^2(D)}^{2 \alpha}\,{\rm d}s  \notag\\
& + 2 \int_{0}^{t \wedge \tau_R^\delta} \int_{\pmb{E}} \Big ( \|  u(s) + \eta(u_1(s);z) - \eta(u_2(s);z) \|_{L^2(D)}^2 -  \|  u(s)\|_{L^2(D)}^2 \Big) \, \widetilde{N}({\rm d}z,{\rm d}s)\,. \label{inq:uni-2}
\end{align}
Recalling the definition of $\tau_R^\delta$, we have from  \eqref{inq:uni-2}, after taking expectation
\begin{align}
    \mathbb{E} \big[\| u(t \wedge \tau_R^\delta)\|_{L^2(D)}^2\big] & \le  C  \int_0^{t \wedge \tau_R^\delta}  \mathbb{E}\big[\| u(s)\|_{L^2(D)}^2\big] \,{\rm d}s \notag \\
    & + \frac{2\left(1+K_2^2\right) R^{1-\alpha}+K_2^2(4 \lambda(D))^{1-\alpha}}{(1-\alpha) \mathrm{e}} \int_0^{t \wedge \tau_R^\delta} \mathbb{E}\big[ \| u(s)\|_{L^2(D)}^{2\alpha}\big] \,{\rm d}s \notag
\end{align}
Setting ${\tt u}(t):=\mathbb{E} \big[\| u(t \wedge \tau_R^\delta)\|_{L^2(D)}^2\big]$, and applying Lemma \ref{lem:nonlinear-gronwall}, we get
\begin{align*}
      {\tt u}(t) & \le \left\{\frac{2\left(1+K_2^2\right) R^{1-\alpha}+K_2^2(4 \lambda(D))^{1-\alpha}}{\mathrm{e}} \int_0^{t}  e^{C(1-\alpha)(t-s)} \,{\rm d}s \right\}^{\frac{1}{1-\alpha}} \notag \\
     & = \left[\frac{2\left(1+K_2^2\right) R^{1-\alpha}+K_2^2(4 \lambda(D))^{1-\alpha}}{\mathrm{e}}\right]^{\frac{1}{1-\alpha}} \times\left(\int_0^{t} e^{C(1-\alpha) s} \,{\rm d}s \right)^{\frac{1}{1-\alpha}} \notag \\
     & \leq \frac{1}{2}\left\{\left[\frac{4\left(1+K_2^2\right) t^\alpha}{\mathrm{e}}\right]^{\frac{1}{1-\alpha}}R +\left[\frac{2 K_2^2 t^\alpha}{\mathrm{e}}\right]^{\frac{1}{1-\alpha}} \times 4 \lambda(D)\right\} \times\left(\int_0^{t} e^{C s} \,{\rm d}s \right).   
\end{align*}
Setting $T^*:=\left(\frac{\mathrm{e}}{4\left(1+K_2^2\right)}\right)^2$, and then letting $\alpha \rightarrow 1$, we obtain
\begin{align*}
    {\tt u}(t) =0, \quad \forall \, 0 \leq t \leq T^*\,.
\end{align*}
Note that $T^*$ is independent of the initial value. Hence, one can use the same argument repeatedly on $[T^*, 2T^*]$, $[2T^*, 3T^*]$ and so on to deduce that  ${\tt u}(t)=0$ for any $t \geq 0$. This means
\begin{align*}
    \mathbb{E}\Big[\left\| (u_1 - u_2) \left(t \wedge \tau_R \wedge \tau_R^{\prime} \wedge \tau^\delta\right)\right\|_{L^2(D)}^2\Big]=0 \quad \forall t \geq 0.
\end{align*}
Observe that $\mathbb{P}$-a.s.,  $\tau_R,~\tau_R^{\prime} \rightarrow \infty$ as $R \rightarrow \infty$. Hence, we have
\begin{align*}
    \mathbb{E}\Big[\left\| (u_1 - u_2) \left(t \wedge \tau^\delta\right)\right\|_{L^2(D)}^2\Big]=0, \quad \forall~ t \geq 0\,.
\end{align*}
In view of the definition of $\tau^{\delta}$ and the above equality, one can easily conclude that 
\begin{align*}
   u_1(t)=u_2(t), \quad \mathbb{P} \text {-a.s., } \quad \forall t \geq 0 . 
\end{align*}
This completes the proof.
\subsection{Proof of Theorem \ref{thm:existence-strong}}
In view of Subsections \ref{subsec:existence-weak-local} and \ref{subsec:uniqueness}, equation \eqref{eq:log-nonlinear} has a unique weak solution $u$ 
 on $[0,T_2]$. Hence, the existence of a unique strong solution of \eqref{eq:log-nonlinear} follows from \cite[Theorem 2]{Ondrejat-2004}. Moreover, 
 $u$ satisfies the {\it a-priori}  estimate as in Lemma \ref{lem:est-galerkin}. To prove the existence of a global strong solution, we closely follow the argument of \cite{Shang-2022}. For any $t\ge 0$ and $A \in \mathcal{B}(\pmb{E})$, define 
 $$ N_t(A)=N((0,t]\times A).$$
 One can easily check that,  for any ${\tt s}\ge 0$, $N_t^{\tt s}(\cdot):= N_{t+{\tt s}}(\cdot)-N_{{\tt s}}(\cdot)$ is a time-homogeneous Poisson random measure with respect to the filtration $\{\mathcal{F}_t^{{\tt s}}:=\mathcal{F}_{t+{\tt s}}\}_{t\ge 0}$. Moreover, $N_t^{\tt s}(\cdot)$ and 
 $N_t(\cdot)$ have the same distribution. Keeping these in mind, one can follow the proof of \cite[Step $2$, Theorem $5.4$]{Shang-2022} to conclude the existence of a probabilistic strong solution $u$ to the equation \eqref{eq:log-nonlinear} on the interval $[0,T]$ for any given $T>0$.
 Moreover, $u$ satisfies the moment estimate \eqref{esti:bound-weak-solun}. Furthermore, since $T>0$ is arbitrary and pathwise uniqueness holds, the solution $u$ is global. This finishes the proof of Theorem \ref{thm:existence-strong}.

\section{Skeleton equation: wellposedness} \label{sec:log-skeleton}
In this section, we prove the wellposedness theory for the skeleton equation \eqref{eq:log-skeleton} under the assumptions \ref{A1} and \ref{B1}.
We start with the following lemmas whose proof can be found in \cite{Budhiraja-2013} and \cite{Yang-2015}. 

\begin{lem} \label{lem:identity-eta}
Let the assumption \ref{B1} hold. Then the following are true. 
    \begin{itemize}
        \item [(a)] For $i=1,2$ and every $N \in \mathbb{N}$,
\begin{align*}
& C_{i, 2}^{N}:=\underset{g \in S_N} \sup \int_0^T \int_{\pmb{E}} {\tt h}_i^2(z)(g(s, z)+1) \,m({\rm d}z) {\rm d}s< +\infty, \\
& C_{i, 1}^{N}:= \underset{g \in S_N} \sup \int_0^T \int_{\pmb{E}} {\tt h}_i(z) |g(s, z)-1| \,m({\rm d}z) {\rm d} s< +\infty .
\end{align*}
\item[(b)] For every $\eta>0$, there exists $\tilde{\delta}>0$ such that for any $A \subset[0, T]$ satisfying $\lambda_T(A)<\tilde{\delta}$,
\begin{align*}
    \underset{g \in S_N}\sup \int_{A} \int_{\pmb{E}} {\tt h}_i(z)|g(s, z)-1| \, m({\rm d}z) {\rm d}s \leq \eta .
\end{align*}
\end{itemize}
\end{lem}
\begin{lem} \label{lem:technical-control}
Under the assumption \ref{B1}, the following holds:
    \begin{itemize}
        \item[(i)] If $\underset{t \in [0,T]}\sup \|y(t)\|_{L^2(D)}<\infty$, then for any $g \in \mathbb{S}$
\begin{align*}
\int_{\pmb{E}} \eta( y(\cdot); z)\,(g(\cdot, z)-1) \, m({\rm d}z) \in L^{1}([0, T]; L^2(D))\,.
\end{align*}
\item[(ii)] Let $\{y_n\}_{n\in \mathbb{N}}$ be a family of mapping from $[0,T]$ to $L^2(D)$ such that  $$\underset{n\in \mathbb{N}}\sup\, \underset{s \in[0, T]}\sup\, \left\|y_n(s)\right\|_{L^2(D)} < \infty.$$
Then, for any $N\in \mathbb{N}$, there exists a finite constant $ \tilde{C}_{N}>0$ such that
\begin{align*}
\tilde{C}_{N}:= \sup_{g\in S_N}\sup_{n\in \mathbb{N}} \int_0^T \Big\|\int_{\pmb{E}} \eta(y_n(s); z)(g(s, z)-1) \, m({\rm d}z) \Big\|_{L^2(D)} {\rm d}s  <\infty \,.
\end{align*}
\end{itemize}
\end{lem}

\subsection{Galerkin approximations of \eqref{eq:log-skeleton} and {\it a-priori}  estimate:}
To show the existence of a solution for \eqref{eq:log-skeleton}, as a first step, we demonstrate the existence of an approximate solutions
(via Galerkin method), and then derive necessary {\it a-priori} estimates. Using the {\it a-priori}  estimates, we prove certain strong convergence result of approximate solutions, which then leads to the existence of a solution of the control equation \eqref{eq:log-skeleton}.
\vspace{0.2cm}

Let $g\in \mathbb{S}$ be fixed. Let $P_n$ be the projection operator given in \eqref{def-projection} in Section \ref{sec:existence-weak-solu}. For each fixed $n\in \mathbb{N}$, consider the following ODE in the finite-dimensional space $L_n$:
\begin{align} 
    \begin{cases}
   \displaystyle {\rm d}{\tt u}_n(t) - \Delta {\tt u}_n(t) {\rm d}t = P_n\left[{\tt u}_n(t) \log \left|{\tt u}_n(t)\right|\right] {\rm d}t+ P_n \Big( \int_{\pmb{E}} \eta({\tt u}_n;z)\, (g(t,z)-1) \, m({\rm d}z)\Big)\,{\rm d}t,~~t>0, \\
{\tt u}_n(0)=P_n u_0\,.
\end{cases} \label{eq:log-skeleton-finite-dim}
\end{align}

One can follow a similar line of argument~(under cosmetic change) as in Section \ref{sec:existence-weak-solu} and make a minor adjustment as in the proof of \cite[Theorem $2.1$]{Shizan-2005} to arrive at the following theorem.
\begin{thm}
    Let the assumptions \ref{A1} and \ref{B1} hold true. For any $n \in \mathbb{N}$, there exists a unique global solution ${\tt u}_n$ to equation \eqref{eq:log-skeleton-finite-dim}.
\end{thm}
We wish to analyze the strong convergence of the family $\{{\tt u}_n \}$ on some appropriate space. To do so, we need to derive some essential {\it a-priori} estimates. 

\begin{lem}\label{lem:apriori-skeleton-galerkin}
    Under assumptions \ref{A1} and \ref{B1}, the following estimates on $\{{\tt u}_n\}$ hold.
    \begin{itemize}
    \item[a)] There exists a constant $C>0$, independent of $n$ such that
    \begin{align}
        \underset{n\in \mathbb{N}}\sup\, \bigg[ \underset{s\in [0,T]}\sup\,\| {\tt u}_n(s)\|_{L^2(D)}^2 + \int_{0}^{T}  \| {\tt u}_n(s)\|_{W_0^{1,2}(D)}^2 \,{\rm d}s \bigg] \le C\,. \label{esti:apriori-skeleton-galerkin}
    \end{align}
    \item[b)] For $\alpha \in (0, \frac{1}{2})$, there exists $C_{\alpha} > 0$ such that
    \begin{align*}
        \underset{n\in \mathbb{N}}\sup\,\bigg[  \| {\tt u}_n\|_{W^{\alpha,2}([0,T];{W^{-1,2}(D)})}\bigg] \le C_{\alpha} \,.
    \end{align*}
    \end{itemize}
\end{lem}
\begin{proof}
    We use the chain-rule, Cauchy-Schwartz inequality, Poincar{\'e} inequality and the logarithmic Sobolev inequality \eqref{inq:log-sov-1} with $\epsilon = \frac{1}{2}$ to have
    \begin{align}
        &  \left \| {\tt u}_n(t) \right \|_{L^2(D)}^{2}  +  \int_{0}^{t} \| {\tt u}_{n}(t)\|_{W_{0}^{1,2}(D)}^{2} \, {\rm d}s  \notag \\ 
         & \le \left \| u_0 \right \|_{L^2(D)}^{2} + C\int_0^t \Phi_2(s)\,{\rm d}s  +  C \int_0^t (1+ \Phi_2(s))\|{\tt u}_n(s)\|_{L^2(D)}^2\,{\rm d}s \notag \\
        & \qquad + 2 \int_0^t \|{\tt u}_n(s)\|_{L^2(D)}^2 \log \big(\|{\tt u}_n(s)\|_{L^2(D)}\big)\,{\rm d}s\,,\notag 
    \end{align} 
    where $\Phi_2(s)$ is given by
    $$\Phi_2(s):=\int_{\pmb{E}} {\tt h}_2(z) |g(s,z)-1|\,m({\rm d}z).$$
    One can use Gronwall's lemma with logarithmic nonlinearity together with Lemma \ref{lem:identity-eta} to arrive at the uniform estimate \eqref{esti:apriori-skeleton-galerkin}. To prove the assertion ${\rm b)}$, we write ${\tt u}_n$ as 
    \begin{align*}
{\tt u}_n(t)&= P_n u_0 + \int_0^t \Delta {\tt u}_n(s)\,{\rm d}s + \int_0^t P_n\big( {\tt u}_n(s)\log(|{\tt u}_n(s)|)\big)\,{\rm d}s + \int_0^t  P_n \Big( \int_{\pmb{E}} \eta({\tt u}_n;z)\, (g(t,z)-1) \, m({\rm d}z)\Big)\,{\rm d}s \notag \\
& \equiv J_1 + J_2(t) + J_3(t) + J_4(t)\,.
    \end{align*}
    Following the calculation as in the proof of \cite[Lemma 4.6]{Pan-2022}, we have
    \begin{align*}
\|J_i\|_{W^{\alpha,2}([0,T];{W^{-1,2}(D)})}\le C_i,~~~(1\le i\le 3)\,.
    \end{align*}
    In view of \eqref{cond:linear-growth-ldp}, Lemma \eqref{lem:identity-eta} and the uniform estimate \eqref{esti:apriori-skeleton-galerkin} we have, for any $s,t \in [0,T]$,
\begin{align*}
\| J_4(t)-J_4(s)\|_{W^{-1,2}(D)}^2 & \le C \Big( \int_s^t \int_{\pmb{E}}\| \eta({\tt u}_n(r);z)\|_{L^2(D)}|g(r,z)-1|\,m({\rm d}z)\,{\rm d}r\Big)^2 \\
& \le C \Big( 1 + \sup_{t\in [0,T]}\|{\tt u}_n(t)\|_{L^2(D)}\Big)^2 \Big( \int_s^t \int_{\pmb{E}} {\tt h}_2(z)|g(r,z)-1|\,\,m({\rm d}z)\,{\rm d}r\Big)^2 \\
& \le C \int_0^T\Phi_2(s)\,{\rm d}s \int_s^t \Phi_2(r)\,{\rm d}r\,.
\end{align*}
This shows that, again in view of Lemma \eqref{lem:identity-eta},
$$ \int_0^T \| J_4(t)\|_{W^{-1,2}(D)}^2\,{\rm d}r \le C.$$
A simple application of Fubini's theorem reveals that there exists a constant $C>0$ such that for any $\alpha \in (0, \frac{1}{2})$
\begin{align*}
\int_0^T\int_0^T \int_s^t \frac{\Phi_2(r)}{|t-s|^{1+ 2 \alpha}}\,{\rm d}r\,{\rm d}s\,{\rm d}t \le C \int_0^T \Phi_2(r)\,{\rm d}r\,.
\end{align*}
Thus, there exists a constant $C_\alpha$, independent of $n$, such that
\begin{align*}
\|J_i\|_{W^{\alpha,2}([0,T];{W^{-1,2}(D)})}\le C_\alpha\,.
\end{align*}
  This completes the proof.   
\end{proof}
\subsection{Existence proof: skeleton equation \eqref{eq:log-skeleton}}\label{subsec:existence-skeleton}
Thanks to Lemmas \ref{lem:cpt} and \ref{lem:apriori-skeleton-galerkin}, we have the following lemma.
\begin{lem}\label{lem:conv-1-skeleton-galerkin}
 There exists a sub-sequence of $\{{\tt u}_n \}$, still denoted by $\{{\tt u}_n\}$, and $u_g \in L^{2}([0,T];L^{2}(D))$ such that
 \begin{itemize}
\item[i)] ${\tt u}_n \rightarrow u_g$ in $ L^{2}([0,T];L^{2}(D))$, and ${\tt u}_n(s,x)\rightarrow u_g(s,x)$ for a.e. $(s,x)\in [0,T]\times D$.
\item[ii)] $P_n\left[{\tt u}_{n} \log \left|{\tt u}_{n} \right|\right]\rightarrow u_g \log |u_g| $ in $L^r([0,T];W^{-1,2}(D))$ for $r\in (1,2)$. 
 \end{itemize}
\end{lem}
Again, in view of {\it a-priori}  estimates in Lemma  \ref{lem:apriori-skeleton-galerkin}, the following weak convergence results hold.
\begin{align}\label{conv:weak-skeleton-galerkin}
\begin{cases}
{\tt u}_n \rightharpoonup u_g~\text{in}~L^2([0,T];W_0^{1,2}(D)), \\
{\tt u}_n \overset{*}{\rightharpoonup} u_g~\text{in}~L^{\infty}([0,T];L^{2}(D))\,,\\
\Delta{\tt u}_{n} \rightharpoonup \Delta u_g~\text{in}~L^2([0,T];W^{-1,2}(D))\,.
\end{cases}
\end{align}
One can easily get 
\begin{align}
    \underset{t \in [0,T]}\sup\, \|u_g(t)\|_{L^2(D)}^{2} \le C\,, \quad \int_{0}^{T} \| u_g(t)\|_{W_0^{1,2}(D)}^2 \,{\rm d}t \le C\,. \label{esti:uni-skeletin}
\end{align}
\begin{lem} The following convergence holds in $L^{\infty}([0,T]; W^{-1,2}(D)):$ 
\begin{align}
\int_0^{\cdot}P_n \Big( \int_{\pmb{E}}\eta({\tt u}_{n};z)(g(s,z)-1) m({\rm d}z)\Big)\,{\rm d}s \rightarrow  \int_0^{\cdot}\int_{\pmb{E}}\eta(u_g;z)(g(s,z)-1)\, m({\rm d}z)\,{\rm d}s\,. \label{conv:corection-term-skeleton-galerkin}
\end{align}
\end{lem}
\begin{proof}
For any $\varepsilon >0$, define the set 
$$A_{n, \varepsilon}:=\{ t\in [0,T]:~ \|{\tt u}_n(t)-u_g(t)\|_{L^2(D)}> \varepsilon\}.$$
Since ${\tt u}_n$ strongly converges to $u_g$ in $L^2([0,T];L^2(D))$, by Chebyshev inequality, we see that
$$ \lim_{n\goto \infty} \lambda_T (A_{n,\varepsilon}) \le \lim_{n\goto \infty} \frac{1}{\varepsilon^2} \int_0^T  \|{\tt u}_n(t)-u_g(t)\|_{L^2(D)}^2\,{\rm d}t=0.$$
Thus, by part $(b)$ of Lemma \ref{lem:identity-eta}, one has
\begin{align}
 \limsup_{n\goto \infty} \sup_{g\in S_N} \int_{A_{n,\varepsilon}}\int_{\pmb{E}} {\tt h}_1(z)|g(t,z)-1|\,m({\rm d}z)\,{\rm d}t < \varepsilon.\label{inq:conv-inter}
 \end{align}
Observe that, in view of \eqref{cond:lipschitz-ldp} and the uniform estimates \eqref{esti:apriori-skeleton-galerkin} and \eqref{esti:uni-skeletin}
\begin{align*}
& \sup_{t\in [0,T]} \Big\| \int_0^t P_n \Big( \int_{\pmb{E}}\eta({\tt u}_{n};z)(g(s,z)-1) m({\rm d}z)\Big)\,{\rm d}s -\int_0^{t}\int_{\pmb{E}}\eta(u_g;z)(g(s,z)-1)\, m({\rm d}z)\,{\rm d}s\Big\|_{W^{-1,2}(D)} \notag \\
& \le 2 \int_{0}^T \int_{\pmb{E}} \|{\tt u}_n -u_g\|_{L^2(D)} {\tt h}_1(z)|g(s,z)-1|\,m({\rm d}z)\,{\rm d}s \notag \\
&\le C \int_{A_{n,\varepsilon}}\int_{\pmb{E}} {\tt h}_1(z)|g(s,z)-1|\,m({\rm d}z)\,{\rm d}s + C \varepsilon \int_0^T\int_{\pmb{E}} {\tt h}_1(z)|g(s,z)-1|\,m({\rm d}z)\,{\rm d}s \le C(N)\varepsilon\,,
\end{align*}
where in the last inequality, we have used part $b)$ of Lemma \ref{lem:identity-eta} and the inequality \eqref{inq:conv-inter}. Since $\varepsilon>0$ is arbitrary, \eqref{conv:corection-term-skeleton-galerkin} holds as well. 
\end{proof}
\begin{rem}\label{rem:converence-related}
The above proof yields the following convergence result: let $x_n\rightarrow x$ in $L^2([0,T];L^2(D))$ and $\mathbf{h}\in \mathcal{H}_2 \cap L^2(\pmb{E},m)$. Then 
$$ \lim_{n\goto \infty} \sup_{\mathrm{k}\in S_N} \int_0^T \int_{\pmb{E}} \|x_n(s)-x(s)\|_{L^2(D)} \mathbf{h}(z)|\mathrm{k}(s,z)-1|\,m({\rm d}z)\,{\rm d}s=0.$$
\end{rem}
We use the convergence results in \eqref{conv:weak-skeleton-galerkin} and \eqref{conv:corection-term-skeleton-galerkin} together with Lemma \ref{lem:conv-1-skeleton-galerkin}, and pass to the limit as $n\rightarrow \infty$ in the equation satisfied by ${\tt u}_n$ to see that $u_g\in L^{\infty}([0,T];L^{2}(D))\cap L^2([0,T];W_0^{1,2}(D))$ satisfies the following equation in $W^{-1,2}(D)$:
\begin{align}
 \displaystyle u_g(t) = u_{0} + \int_{0}^{t} \Delta {u_g(s)}\, {\rm d}s + \int_{0}^{t} u_g(s) \log (|u_g(s)|) \, {\rm d}s + \int_{0}^{t}
 \int_{\pmb{E}}\eta(u_g(s);z)\, (g(s,z)-1) \, m({\rm d}z)\,{\rm d}s\,.\notag 
 \end{align}
 In other words, $u_g$ is a solution of the skeleton equation \eqref{eq:log-skeleton}. Note that, in view of Lemma \ref{lem:technical-control}, and 
the uniform estimate \eqref{esti:uni-skeletin},
\begin{align*}
\frac{d u_g}{{\rm d}t}\in L^2([0,T]; W^{-1,2}(D)) + L^1([0,T]; L^2(D))\,.
\end{align*}
This implies that $u_g\in C([0,T];L^2(D))$. 
\subsection{Uniqueness proof: skeleton equation}\label{subsec:uniqueness-skeleton}
Let $u_1$ and $u_2$ be two solutions of skeleton equation \eqref{eq:log-skeleton}. Set ${\tt u}_g:=u_1-u_2$. Then, using chain-rule, we get
\begin{align*}
& \|{\tt u}_g(t\wedge \tau_\delta)\|_{L^2(D)}^2 + 2 \int_{0}^{t\wedge \tau_\delta} \|{\tt u}_g(s)\|_{W_0^{1,2}(D)}^2\,{\rm d}s
= 2 \int_{0}^{t\wedge \tau_\delta} \Big\langle u_1 \log(|u_1|)-u_2 \log(|u_2|), {\tt u}_g\Big\rangle\,{\rm d}s \notag \\
&\quad + 2 \int_{0}^{t\wedge \tau_\delta} \int_{|z|>0}  \Big\langle \big\{ \eta(u_1;z)-\eta(u_2;z)\big\} (g(s,z)-1), {\tt u}_g\Big\rangle\, m({\rm d}z)\,{\rm d}s\equiv \mathcal{B}_4 + \mathcal{B}_5\,,
\end{align*}
where $\tau_\delta:=\inf\{ t\in (0,T]: \|u_1(t)-u_2(t)\|_{L^2(D)}> \delta\},~~\delta\in (0,1)$. 
Thanks to Cauchy-schwartz inequality and \eqref{cond:lipschitz-ldp}, we see that
\begin{align*}
\mathcal{B}_5 \le 2 \int_{0}^{t\wedge \tau_\delta} \Big(\int_{\pmb{E}} {\tt h}_1(z)|g(s,z)-1|\,m({\rm d}z)\Big)\|{\tt u}_g(s)\|_{L^2(D)}^2\,{\rm d}s\,.
\end{align*}
Thus, one can follow the same line of argument as invoked in the proof of \cite[Theorem $4.3$]{Shang-2022}~(see also subsection \ref{subsec:uniqueness}) together with Lemma \ref{lem:identity-eta} to conclude that ${\tt u}_g(t)=0$ for all $t\in [0,T]$. In other words, equation \eqref{eq:log-skeleton} has a unique solution. 
\subsection{Proof of Theorem \ref{thm:skeleton}} In view of subsections \ref{subsec:existence-skeleton} and \ref{subsec:uniqueness-skeleton}, the deterministic skeleton equation \eqref{eq:log-skeleton}  has a unique solution $u_g \in C([0,T];L^2(D))\cap L^2([0,T]; W_0^{1,2}(D))$. Moreover, following the same argument as invoked to achieve \eqref{esti:apriori-skeleton-galerkin}, one can prove the estimate \eqref{esti:uni-skeletin}.
\section{Large Deviation Principle} \label{sec:log-LDP}
This section provides the proof of Theorem \ref{thm:ldp-log-laplace-levy}.  According to \cite[Theorem 4.4]{Zhang-2023} and \cite[Theorem 3.2]{Matoussi-2021} which is an adaptation of the original results given in 
\cite[Theorems 2.3 and 2.4]{Brzezniak-2013}, to prove Theorem \ref{thm:ldp-log-laplace-levy}, it is sufficient to prove the following two conditions:
\begin{Assumptions3}
    \item \label{LDP1} For any $N \in \mathbb{N}$, let $\{\varphi_\epsilon:~\epsilon>0\} \subset \tilde{\mathcal{U}}_N$. Then, for any $\delta > 0$ 
        \begin{align}
            \underset{\epsilon \rightarrow 0}\lim \mathbb{P} \left\{\rho_T \left(\mathcal{G}^{\epsilon} (\epsilon N^{\epsilon^{-1}\,\varphi_{\epsilon}}), \mathcal{G}^{0}(\nu_{T}^{\varphi_{\epsilon}})
            \right) > \delta  \right\} = 0\,, \notag
        \end{align}
        where the metric $\rho_T$ is defined in \eqref{defi:metric}.
        \item \label{LDP2} For any given $N \in \mathbb{N}$, let $g_n, g \in S_N$ be such that $g_n \rightarrow g$ as $n \rightarrow \infty$. Then
\begin{align*} 
   \mathcal{G}^0 \big(\nu_{T}^{g_n} \big) \rightarrow \mathcal{G}^0 \big(\nu_{T}^{g} \big).
\end{align*}
\end{Assumptions3}
 Before proving these conditions, we state a technical lemma the proof of which can be found in  \cite[Lemma 3.11]{Budhiraja-2013}.
 Let $m_T:=\lambda_T \otimes m$. Define the space
\begin{align*}
\mathcal{H}_2([0,T]\times \pmb{E}):=\Big\{ \text{ $\mathsf{h}\in L^2([0,T]\times \pmb{E}, m_T)$ such that for all $\delta \in (0,\infty)$ and for all $E\in \mathcal{B}([0,T]\times \pmb{E})$} \\
 \text{with $m_T(E)< +\infty$, 
$\int_{E} \exp(\delta |\mathsf{h}(s,z)|)\,m({\rm d}z)\,{\rm d}s < + \infty $}\Big\}
\end{align*}
\begin{lem}\label{lem:technical-control-1} 
Let $\mathsf{h}(\cdot,\cdot)\in \mathcal{H}_2([0,T]\times \pmb{E})$. For fixed $N\in \mathbb{N}$, let $g, g_n \in S_N$ be such that 
$g_n \goto g$ as $n\rightarrow \infty$. Then 
\begin{align*}
    \underset{n \rightarrow \infty}\lim \int_0^T \int_{\pmb{E}} \mathsf{h}(s, z)\left(g_n(s, z)-1\right)\, m({\rm d}z) \, {\rm d}s= \int_0^T \int_{\pmb{E}} \mathsf{h}(s, z)(g(s, z)-1)\, m({\rm d}z) \, {\rm d}s .
\end{align*}
\end{lem}

\subsection{Proof of the condition \ref{LDP2}}
In this subsection, we prove condition \ref{LDP2}. To prove the condition it is enough to show that, for each fixed $N >0$, if $g_{n} \rightarrow g$ in $S_N$ then $v_n \rightarrow u_g$ in $C([0,T];L^{2}(D))\cap L^2([0,T];W_0^{1,2}(D))$, where $v_n$ and $u_g$ are the unique solution of \eqref{eq:log-skeleton} corresponding to $g_n$ and $g$ respectively.
\vspace{0.1cm}

To proceed further, we first derive {\it a-priori} estimate for $v_n$. One may follow the similar argument as done in the proof of Lemma \ref{lem:apriori-skeleton-galerkin} to find constants $C_{1,N}, C_{2,N}$ and $C_{\alpha, N}$, independent of $n$, such that 
\begin{align}\label{esti:apriori-c2}
\begin{cases}
  \displaystyle \underset{n\in \mathbb{N}}\sup\, \bigg[ \underset{s\in [0,T]}\sup\,\| v_n(s)\|_{L^2(D)}^2 + \int_{0}^{T}  \| v_n(s)\|_{W_0^{1,2}(D)}^2 \,{\rm d}s \bigg] \le C_{1,N}\,, \\
 \underset{n\in \mathbb{N}}\sup\,\bigg[  \| v_n\|_{W^{\alpha,2}([0,T];{W^{-1,2}(D)})}\bigg] \le C_{\alpha,N}\,,~\alpha\in (0,\frac{1}{2})\,. 
\end{cases}
\end{align}
Moreover, repeating the same argument as in subsection \ref{subsec:existence-skeleton}, we have the following: there exits a sub-sequence of $\{v_n\}$, still denoted by $\{v_n\}$, and ${\tt v}\in C([0,T]; L^2(D))\cap L^2([0,T];W_0^{1,2}(D))$ such that 
\begin{equation}\label{conv:c2-1}
\begin{aligned}
&v_n\rightarrow {\tt v} ~\text{in}~~L^2([0,T];L^2(D)), \quad v_n \overset{*}{\rightharpoonup} {\tt v}~\text{in}~L^{\infty}([0,T];L^{2}(D))\,, \\
& v_n \log(|v_n|) \rightarrow {\tt v}\log(|{\tt v}|)~\text{in}~L^r([0,T]; W^{-1,2}(D))~\text{for}~r\in (1,2)\,, \\
& \Delta v_n \rightharpoonup \Delta {\tt v}~\text{in}~L^2([0,T];W^{-1,2}(D)), \quad v_n \rightharpoonup {\tt v}~\text{in}~L^2([0,T];W^{1,2}_0(D))\,.
\end{aligned}
\end{equation}
We now prove that the limit function ${\tt v}$ is indeed a solution of \eqref{eq:log-skeleton}. Regarding the control term, we have the following lemma.
\begin{lem}\label{lem:conv-c2-2}
For any $\phi \in W_0^{1,2}(D)$, there holds
\begin{align}
 \lim_{n\goto \infty} \int_{0}^{T}\int_{\pmb{E}} \langle  \eta(v_n(t);z), \phi\rangle (g_n(t,z)-1)\,m({\rm d}z)\,{\rm d}t =\int_0^T\int_{\pmb{E}} \langle \eta({\tt v};z), \phi \rangle (g(t,z)-1)\,m({\rm d}z)\,{\rm d}t\,.\notag 
\end{align}
\end{lem}
\begin{proof}
By re-writing, we see that, for any $\phi \in W_0^{1,2}(D)$
\begin{align*}
 &\int_{0}^{T}\int_{\pmb{E}} \langle  \eta(v_n(t);z), \phi\rangle (g_n(t,z)-1)\,m({\rm d}z)\,{\rm d}t \notag \\
 &=\int_{0}^{T}\int_{\pmb{E}} \langle  \big(\eta(v_n(t);z)-\eta({\tt v}(t);z)\big), \phi\rangle (g_n(t,z)-1)\,m({\rm d}z)\,{\rm d}t \notag \\
 & \quad + \int_{0}^{T}\int_{\pmb{E}} \langle  \eta({\tt v}(t);z), \phi\rangle \big(g_n(t,z)-g(t,z)\big)\,m({\rm d}z)\,{\rm d}t 
 + \int_0^T\int_{\pmb{E}} \langle \eta({\tt v};z), \phi \rangle (g(t,z)-1)\,m({\rm d}z)\,{\rm d}t \notag \\
&\equiv \mathcal{J}_{n,1} + \mathcal{J}_{n,2} + \int_0^T\int_{\pmb{E}} \langle \eta({\tt v};z), \phi \rangle (g(t,z)-1)\,m({\rm d}z)\,{\rm d}t\,.
\end{align*}
To prove the lemma, we need to show that
$$ \mathcal{J}_{n,1}, \mathcal{J}_{n,2}\rightarrow 0~~~~\text{as $n \rightarrow \infty$ }.$$
Following the proof of \eqref{conv:corection-term-skeleton-galerkin}, it is easy to check that
$$ \mathcal{J}_{n,1} \rightarrow 0~~~~\text{as $n \rightarrow \infty$ }.$$
For $\mathcal{J}_{n,2}$ we proceed as follows. Let $\mathsf{h}(t,z):=\langle \eta({\tt v}(t);z), \phi\rangle$. In view of \eqref{cond:linear-growth-ldp}, Remark \ref{rem:regarding-condition} and the fact that ${\tt v}\in L^\infty([0,T];L^2(D))$, one can easily check that $\mathsf{h}(\cdot,\cdot) \in \mathcal{H}_2([0,T]\times \pmb{E})$. Hence an application of Lemma \ref{lem:technical-control-1} yields that 
$$ \mathcal{J}_{n,2}\rightarrow 0~~~~\text{as $n \rightarrow \infty$ }.$$
This completes the proof. 
\end{proof}

In view of \eqref{conv:c2-1} and Lemma \ref{lem:conv-c2-2}, one can pass to the limit in the equation satisfied by $v_n$ and conclude that 
${\tt v}$ is indeed a solution of \eqref{eq:log-skeleton}. Moreover, thanks to the uniqueness of the solution of the skeleton equation, we get that ${\tt v}=u_g$. To complete proof of condition \ref{LDP2}, we need to show that $v_n \rightarrow u_g $ in $C([0,T];L^{2}(D))\cap L^2([0,T];W_0^{1,2}(D))$. Applying chain-rule and Poincar{\'e} inequality for the equation satisfied by $v_n-u_g$, we get
\begin{align}
    & \frac{{\rm d}}{{\rm d}t}\left \| v_n(t) -u_g(t) \right \|_{L^2(D)}^{2} + 2 \left \| v_n(t) -u_g(t) \right \|_{W_{0}^{1,2}(D)}^{2} \notag \\
    & \quad = 2\, \bigg\langle  {v}_{n}(t) \log \left|{v}_{n}(t) \right| - u_{g}(t) \log \left|u_{g}(t) \right|, {v}_{n}(t) - u_{g}(t) \bigg\rangle \notag\\
    & \quad + 2\, \bigg\langle \int_{\pmb{E}} \big( \eta({v}_{n}(t);z)\, (g_{n}(t,z)-1) - \eta(u_{g}(t);z)\, (g(t,z)-1) \big) m({\rm d}z), {v}_{n}(t) - u_{g}(t) \bigg\rangle \,. \notag  
\end{align}
We use Lemma \ref{lem:result-1}, the monotone property of the logarithmic function, and the uniform bound \eqref{esti:apriori-c2} to have
\begin{align*}
     & \frac{{\rm d}}{{\rm d}t} \left \| v_n(t) -u_g(t) \right \|_{L^2(D)}^{2} + \left \| v_n(t) -u_g(t) \right \|_{W_{0}^{1,2}(D)}^{2}
    \le \Psi(t)\left \| v_n(t) -u_g(t) \right \|_{L^2(D)}^{2} + \mathcal{I}_{1,n}(t)\,,
\end{align*}
where 
\begin{align*}
&\Psi(t)= C+ \int_{\pmb{E}} {\tt h}_1(z)|g(t,z)-1|\,m({\rm d}z)\,, \\
&\mathcal{I}_{1,n}(t)=2\int_{\pmb{E}}  \bigg\langle \eta(v_n(t);z)\, \big((g_{n}(t,z)-1) - (g(t,z)-1) \big) , {v}_{n}(t) - u_{g}(t) \bigg\rangle \,m({\rm d}z)\,.
\end{align*}
Hence, by Gronwall's lemma and (a)-Lemma \ref{lem:identity-eta}, we have
\begin{align*}
& \sup_{t\in [0,T]}\| v_n(t) -u_g(t)\|_{L^2(D)}^{2} + \int_0^T \left \| v_n(t) -u_g(t) \right \|_{W_{0}^{1,2}(D)}^{2}\,{\rm d}t  \notag \\
&\le \exp\Big( \int_0^T \Psi(s)\,{\rm d}s\Big) \int_0^T |\mathcal{I}_{1,n}(s)|\,{\rm d}s \le C(N) \int_0^T |\mathcal{I}_{1,n}(s)|\,{\rm d}s\,.\notag
\end{align*}
 Thus, we need to show that 
\begin{align}
 \mathcal{A}_n(T):=\int_0^T |\mathcal{I}_{1,n}(s)|\,{\rm d}s \goto 0~\text{as}~n\goto \infty\,. \label{conv:c2-3}
 \end{align}
  Thanks to Cauchy-Schwartz inequality, the assumption \eqref{cond:linear-growth-ldp} and the moment estimate \eqref{esti:apriori-c2}, we see that
  \begin{align*}
\mathcal{A}_n(T)& \le C \int_0^T \int_{\pmb{E}} {\tt h}_2(z) \big\{|g_n(t,z)-1| + |g(t,z)-1|\big\} \|{v}_{n}(t) - u_{g}(t)\|_{L^2(D)}\,m({\rm d}z)\,{\rm d}t\notag \\
& \le C \sup_{\mathrm{k}\in S_N} \int_0^T \int_{\pmb{E}} \|{v}_{n}(t) - u_{g}(t)\|_{L^2(D)} {\tt h}_2(z) |\mathrm{k}(t,z)-1|\,m({\rm d}z)\,{\rm d}t\,.
  \end{align*}
  Hence by Remark \ref{rem:converence-related}, we conclude that the assertion \eqref{conv:c2-3} holds true.  In other words, $v_n \rightarrow u_g$ in $C([0,T]; L^2(D))\cap L^2([0,T]; W_0^{1,2}(D))$. This completes the proof of condition \ref{LDP2}. 

\subsection{Proof of condition \ref{LDP1}}
In this subsection, we prove condition \ref{LDP1}. To proceed further, we need a certain generalization of the Girsanov theorem. For $\epsilon>0$, let $\varphi_\epsilon \in \tilde{\mathcal{A}}$. Set $\Psi_\epsilon=\frac{1}{\varphi_\epsilon}$. Then $\Psi_\epsilon\in \tilde{\mathcal{A}}$. According to \cite[Theorem III.3.24]{Jacod-1987}~( see also \cite[Theorem $6.1$]{Brzezniak-2023}), the process $\{\mathcal{M}_t^\epsilon(\Psi_\epsilon):~t\in [0,T]\}$ defined by
\begin{align*}
\mathcal{M}_{t}^{\epsilon} (\Psi_{\epsilon}) := \exp  \bigg \{ & \int_{(0,t] \times \pmb{E} \times [0, \epsilon^{-1}\varphi_\epsilon(s,z)] } \log\big(\Psi_{\epsilon}(s,z)\big) N({\rm d}s, {\rm d}z,{\rm d}r)  \\
        &  + \int_{(0,t] \times \pmb{E} \times[0, \epsilon^{-1}\varphi_\epsilon(s,z)]} \big(-\Psi_{\epsilon}(s,z) + 1\big)\,m({\rm d}z)\,{\rm d}s\,{\rm d}r) \bigg \}
\end{align*}
is a martingale on $(\Omega, \mathcal{F}, \mathbb{P}, \mathbb{F})$. Define a new probability measure on $(\Omega, \mathcal{F})$ as 
$$ \mathbb{Q}(A)=\int_{A} \mathcal{M}_{T}^{\epsilon} (\Psi_{\epsilon})\,d\mathbb{P},~~~A\in \mathcal{F}.$$
Then the followings hold.
\begin{itemize}
\item[a)] $\mathbb{P}$ and $\mathbb{Q}$ are mutually absolutely continuous.
\item[b)] Law of $\epsilon N^{\epsilon^{-1}\varphi_{\epsilon}}$ under $\mathbb{Q}$ and law of $\epsilon N^{\epsilon^{-1}}$ under $\mathbb{P}$ are equal on ${\tt M}_T$.
\end{itemize}
We recall the $u_\epsilon=\mathcal{G}^\epsilon(\epsilon N^{\epsilon^{-1}})$ is the unique strong solution of \eqref{eq:log-LDP} on the probability space $(\Omega, \mathcal{F}, \mathbb{P}, \mathbb{F})$. Define
$$\tilde{u}_\epsilon:=\mathcal{G}^\epsilon \big(\epsilon N^{\epsilon^{-1}\varphi_{\epsilon}}\big).$$
It follows that $\tilde{u}_\epsilon$, on the probability space $(\Omega, \mathcal{F}, \mathbb{P}, \mathbb{F})$, is the unique solution of the following controlled stochastic differential equation: for $(t,x)\in (0,T]\times D$,
\begin{equation} \label{eq:epsilon-main-1}
\begin{aligned}
    d\tilde{u}_{\epsilon}(t,x) - \Delta \tilde{u}_{\epsilon} \, {\rm d}t&= \tilde{u}_{\epsilon} \log |\tilde{u}_{\epsilon}| \,{\rm d}t + \epsilon \int_{\pmb{E}}\eta(\tilde{u}_{\epsilon};z) \Big( N^{\epsilon^{-1} {\varphi}_{\epsilon}}({\rm d}z,{\rm d}t)- \epsilon^{-1}\,m({\rm d}z)\,{\rm d}t\Big)\,,\\
    \tilde{u}_{\epsilon}(0,x) &= u_0(x)\,.
\end{aligned}
\end{equation}
We will use the following version of equation \eqref{eq:epsilon-main-1}:
\begin{equation} \label{eq:epsilon-main}
\begin{aligned}
    d\tilde{u}_{\epsilon}(t,x) - \Delta \tilde{u}_{\epsilon} \, {\rm d}t&= \tilde{u}_{\epsilon} \log |\tilde{u}_{\epsilon}| \,{\rm d}t + \epsilon \int_{\pmb{E}}\eta(\tilde{u}_{\epsilon};z) \widetilde{N}^{\epsilon^{-1} {\varphi}_{\epsilon}}({\rm d}z,{\rm d}t) +  \int_{\pmb{E}}\eta(\tilde{u}_{\epsilon};z)\, (\varphi_{\epsilon}(t,z)-1)\,m({\rm d}z)\,{\rm d}t\,, \\
    \tilde{u}_{\epsilon}(0,x) &= u_0(x)\,.
\end{aligned}
\end{equation}

We first derive {\it a-priori} estimate for $\tilde{u}_{\epsilon}$. 
\begin{lem}
    There exist $\epsilon_0 \in (0,1)$, and a constant $\tilde{C}_N >0$ such that 
    \begin{align}
         \sup_{0<\epsilon < \epsilon_0}\mathbb{E} \bigg[ \underset{s\in [0,T]}\sup\,\| \tilde{u}_{\epsilon}(s)\|_{L^2(D)}^2 + \int_{0}^{T}  \| \tilde{u}_{\epsilon}(s)\|_{W_0^{1,2}(D)}^2 \,{\rm d}s \bigg] \le \tilde{C}_N\,. \label{esti:apriori-for-cond-c1}
    \end{align} 
    \end{lem}
\begin{proof}
For any $t\ge 0$  and $A\in \mathcal{B}(\pmb{E})$, define 
$$ N_t^\varphi(A):= N((0,t]\times A).$$ Then for any $s\ge 0$, $N_{t+s}^\varphi(\cdot)-N_s^\varphi(\cdot)$ is a time-homogeneous Poisson random measure with respect to the filtration $\mathcal{F}_t^s:=\mathcal{F}_{t+s},~t\ge 0$. Moreover, $N_{t+s}^\varphi(\cdot)-N_s^\varphi(\cdot)$ and $N_t^\varphi(\cdot)$ have the same distribution. Thus, in view of proof of \cite[Step $2$,Theorem $5.4$]{Shang-2022}, to prove the {\it a-priori}  estimate \eqref{esti:apriori-for-cond-c1}, we need to derive the following a-priori estimate: for any $p\ge 2$, there exists a constant $C=C(p,\theta, N, \|u_0\|_{L^2(D)})$ such that 
\begin{align}
\sup_{0<\epsilon < \epsilon_0} \mathbb{E}\Big[ \sup_{t\in [0, T_p]}\|\tilde{u}_\epsilon(t)\|_{L^2(D)}^p + \int_0^{T_p} \|\tilde{u}_\epsilon(s)\|_{L^2(D)}^{p-2} \|\tilde{u}_\epsilon(s)\|_{W_0^{1,2}(D)}^2\,{\rm d}s\Big]\le C(p,\theta, N, \|u_0\|_{L^2(D)})\,, \label{esti:apriori-for-cond-c1-intermidiate}
\end{align}
where $T_p$ is defined in  Lemma \ref{lem:est-galerkin}. To prove \eqref{esti:apriori-for-cond-c1-intermidiate}, 
we follow the similar argument as used in Lemma \ref{lem:est-galerkin}.  We apply It$\hat{o}$-L\'{e}vy formula to the functional $x\mapsto \|x\|_{L^2(D)}^p,~p\ge 2$ on $\tilde{u}_\epsilon$ and use Taylor's formula together with the logarithmic Sobolev inequality to have
\begin{align}
& \|\tilde{u}_\epsilon(t)\|_{L^2(D)}^p + \frac{p}{2} \int_0^t \|\tilde{u}_\epsilon(s)\|_{L^2(D)}^{p-2} \|\tilde{u}_\epsilon(s)\|_{W_0^{1,2}(D)}^2\,{\rm d}s \notag \\
& \le \|u_0\|_{L^2(D)}^p + C \int_0^t \|\tilde{u}_\epsilon(s)\|_{L^2(D)}^p\,{\rm d}s + \int_0^t \|\tilde{u}_\epsilon(s)\|_{L^2(D)}^p \log(\|\tilde{u}_\epsilon(s)\|_{L^2(D)}^p)\,{\rm d}s \notag \\
& \quad  + C_p \epsilon^2 \int_0^t \int_{\pmb{E}}\Big\{ \|\tilde{u}_\epsilon(s)\|_{L^2(D)}^{p-2}\|\eta(\tilde{u}_\epsilon(s);z)\|_{L^2(D)}^2 + 
\|\eta(\tilde{u}_\epsilon(s);z)\|_{L^2(D)}^p \Big\} N^{\epsilon^{-1} {\varphi}_{\epsilon}}({\rm d}z,{\rm d}s) \notag \\
&\qquad  + \epsilon\,p \sup_{r\in [0,t]} \Big| \int_0^r \int_{\pmb{E}}  \|\tilde{u}_\epsilon(s)\|_{L^2(D)}^{p-2} \langle \eta(\tilde{u}_\epsilon(s);z), \tilde{u}_\epsilon(s)\rangle \widetilde{N}^{\epsilon^{-1} {\varphi}_{\epsilon}}({\rm d}z,{\rm d}s)\Big| \notag \\
& \qquad \quad+ p \int_0^t \int_{\pmb{E}}  \|\tilde{u}_\epsilon(s)\|_{L^2(D)}^{p-2} \langle \eta(\tilde{u}_\epsilon(s);z), \tilde{u}_\epsilon(s)\rangle (\varphi_\epsilon(s,z)-1)\,m({\rm d}z)\,{\rm d}s\equiv \sum_{i=1}^6\mathcal{R}_i(t)\,. \label{esti:1-apriori-for-cond-c1-intermidiate}
\end{align}
In view of Cauchy-Schwartz inequality, the growth condition \eqref{cond:linear-growth-ldp} and Young's inequality, we have
\begin{align}
\mathcal{R}_6(t)& \le C\int_0^t \|\tilde{u}_\epsilon(s)\|_{L^2(D)}^{p-2} \big( 1 + \|\tilde{u}_\epsilon(s)\|_{L^2(D)}^2\big)\Psi_{2,\epsilon}(s)\,{\rm d}s \notag \\
& \le C \int_0^t \Psi_{2,\epsilon}(s)\,{\rm d}s + C\int_0^t \Psi_{2,\epsilon}(s)\|\tilde{u}_\epsilon(s)\|_{L^2(D)}^2\,{\rm d}s\,, \label{esti:2-apriori-for-cond-c1-intermidiate}
\end{align}
where $\Psi_{2,\epsilon}(s)$ is given by 
$$\Psi_{2,\eps}(s):= \int_{\pmb{E}} {\tt h}_2(z)|\varphi_{\epsilon}(s,z)-1|\,m({\rm d}z)\,.$$
Using \eqref{esti:2-apriori-for-cond-c1-intermidiate} in \eqref{esti:1-apriori-for-cond-c1-intermidiate}, and then applying log-Gronwall's inequality in Lemma \ref{lem:log-Gronwall} along with Lemma \ref{lem:identity-eta}, we have, after taking expectation
\begin{align}
&\mathbb{E}\Big[ \sup_{s\in [0, t\wedge \bar{\tau}_R]} \|\tilde{u}_\epsilon(s)\|_{L^2(D)}^p + \frac{p}{2} \int_0^{t\wedge \bar{\tau}_R} \|\tilde{u}_\epsilon(s)\|_{L^2(D)}^{p-2} \|\tilde{u}_\epsilon(s)\|_{W_0^{1,2}(D)}^2\,{\rm d}s \Big]\notag \\
& \le \mathbb{E}\Big[ (1+ \bar{M}(t\wedge \bar{\tau}_R))^{e^{T_p}} \exp \Big\{ C e^{T_p} \int_{0}^t \big(1+ \Psi_{2,\epsilon}(s)\big)\,{\rm d}s\Big\}\Big]\notag \\
& \le C(N, p, \theta) \mathbb{E}\Big[ \big( 1 + \bar{M}(t\wedge \bar{\tau}_R)\big)^\frac{p}{p-1+\theta}\Big]\,, \label{esti:3-apriori-for-cond-c1-intermidiate}
\end{align}
where 
\begin{align*}
\bar{M}(t):&= \|u_0\|_{L^2(D)}^p + C \int_0^t \Psi_{2,\epsilon}(s)\,{\rm d}s + \mathcal{R}_4(t)+ \mathcal{R}_5(t)\,, \\
\bar{\tau}_R:&= \inf \Big \{t>0:\|\tilde{u}_\epsilon(t)\|_{L^2(D)} >R \Big \} \wedge T_p\,.
\end{align*}
Observe that
\begin{align}
&  \mathbb{E}\Big[ \big( 1 + \bar{M}(t\wedge \bar{\tau}_R)\big)^\frac{p}{p-1+\theta}\Big] \notag \\
& \le C(p,\theta) \Big\{ 1 + \|u_0\|_{L^2(D)}^\frac{p^2}{p-1+\theta} + (C_{1,1}^N)^\frac{p}{p-1+\theta}\Big\} \notag \\
& + C(p,\theta) \,\epsilon^\frac{p}{p-1+\theta} \,\mathbb{E}\Bigg[\underset{r \in [0,t \wedge \bar{\tau}_{R}]}\sup \left |\int_{0}^{r} \int_{\pmb{E}} \|\tilde{u}_\epsilon(s) \|_{L^2(D)}^{p-2} \langle \eta(\tilde{u}_\epsilon(s);z), \tilde{u}_\epsilon(s) \rangle  \,   \widetilde{N}^{\epsilon^{-1}\varphi_\epsilon}({\rm d}z,{\rm d}s)   \right |^{\frac{p}{p-1+\theta}}\Bigg] \notag \\
& + C(p,\theta)\,\epsilon^\frac{2p}{p-1+\theta}\, \mathbb{E} \Bigg[\bigg( \int_{0}^{t\wedge \bar{\tau}_{R}} \int_{\pmb{E}} \| \tilde{u}_\epsilon(s)\|_{L^2(D)}^{p-2} 
\| \eta(\tilde{u}_\epsilon(s);z) \|_{L^2(D)}^{2}   \,N^{\epsilon^{-1}\varphi_\epsilon}({\rm d}z,{\rm d}s) \bigg)^{\frac{p}{p-1+\theta}}\Bigg] \notag \\
      & \quad +  C(p,\theta)\,\epsilon^\frac{2p}{p-1+\theta}\,  \mathbb{E} \Bigg[ \bigg(\int_{0}^{t\wedge \bar{\tau}_{R}} \int_{\pmb{E}} \| \eta(\tilde{u}_\epsilon(s);z)\|_{L^2(D)}^p  \,N^{\epsilon^{-1}\varphi_\epsilon}({\rm d}z,{\rm d}s) \bigg)^{\frac{p}{p-1+\theta}}\Bigg]\equiv 
      \sum_{i=7}^{10} \mathcal{R}_i\,. \label{esti:4-apriori-for-cond-c1-intermidiate}
\end{align}
In view of Burkholder-Davis-Gundy inequality, \cite[Corollary $2.4$]{Liu-2019}, and the Young's inequality
$$ ab \le \epsilon^{-{\tt s}} a^P + c(P,Q) \epsilon^\frac{{\tt s}(p-1)}{\theta} b^Q,~~\text{with $\frac{1}{P} + \frac{1}{Q}=1$ for $P=\frac{p-1+\theta}{p-1}$ and $1<{\tt s}<\frac{p}{p-1+\theta}$ },$$
 one can proceed with a similar calculation as done for the estimation of $\mathcal{K}_2$ to have
\begin{align}
\mathcal{R}_8 & \le \epsilon \mathbb{E}\Big[ \sup_{s\in [0, t\wedge \bar{\tau}_R]} \|\tilde{u}_\epsilon(s)\|_{L^2(D)}^p\Big] + C \,\epsilon\,
\mathbb{E}\Big[ \int_0^{t\wedge \bar{\tau}_R} \int_{\pmb{E}} \|\eta(\tilde{u}_\epsilon(s);z)\|_{L^2(D)}^\frac{p}{\theta} \varphi_\epsilon(s,z)\,m({\rm d}z)\,{\rm d}s\Big]\notag \\
& + C \epsilon\, \mathbb{E}\Big[ \Big(\int_0^{t\wedge \bar{\tau}_R} \int_{\pmb{E}} \|\eta(\tilde{u}_\epsilon(s);z)\|_{L^2(D)}^2\, \varphi_\epsilon(s,z)\,m({\rm d}z)\,{\rm d}s\Big)^\frac{p}{2\theta}\Big]\equiv \sum_{i=1}^3\mathcal{R}_{8,i}\,.\notag
\end{align}
Thanks to the assumption \eqref{cond:linear-growth-ldp}, Lemma \eqref{lem:identity-eta} and the fact that $0\le {\tt h}_2 \le 1 $, we have
\begin{align*}
\mathcal{R}_{8,2} &\le C\,\epsilon\, \mathbb{E}\Big[ \int_{0}^{t\wedge \bar{\tau}_{R}} \int_{\pmb{E}} \big( 1 + \| \tilde{u}_\epsilon(s)\|_{L^2(D)}^p
\big) |{\tt h}_2(z)|^2 \big( \varphi_\epsilon(s,z) +1\big)\,m({\rm d}z)\,{\rm d}s\Big] \notag \\
& \le  C\,\epsilon \,\mathbb{E}\Big[ \big( 1 + \sup_{s\in [0, t\wedge \bar{\tau}_R]} \|\tilde{u}_\epsilon(s)\|_{L^2(D)}^p\big) \int_{0}^{t\wedge \bar{\tau}_{R}} \int_{\pmb{E}}|{\tt h}_2(z)|^2 \big( \varphi_\epsilon(s,z) +1\big)\,m({\rm d}z)\,{\rm d}s\Big] \notag \\
& \le C(N, \epsilon) + C \,\epsilon\, \mathbb{E}\Big[ \sup_{s\in [0, t\wedge \bar{\tau}_R]} \|\tilde{u}_\epsilon(s)\|_{L^2(D)}^p\Big]\,, \notag \\
\mathcal{R}_{8,3} & \le C\,\epsilon\,  \mathbb{E}\Big[ \big( 1 + \sup_{s\in [0, t\wedge \bar{\tau}_R]} \|\tilde{u}_\epsilon(s)\|_{L^2(D)}^p\big) \Big(
\int_{0}^{t\wedge \bar{\tau}_{R}} \int_{\pmb{E}}|{\tt h}_2(z)|^2 \big( \varphi_\epsilon(s,z) +1\big)\,m({\rm d}z)\,{\rm d}s\Big)^\frac{p}{2\theta}\Big] \notag \\
& \le C(N,p,\theta,T) + C(N,p,\theta,T)\,\epsilon \,  \mathbb{E}\Big[ \sup_{s\in [0, t\wedge \bar{\tau}_R]} \|\tilde{u}_\epsilon(s)\|_{L^2(D)}^p\Big]\,.
\end{align*}
Hence, we have
\begin{align}
\mathcal{R}_8\le C\,\epsilon\, \Big\{ 1 + \mathbb{E}\Big[ \sup_{s\in [0, t\wedge \bar{\tau}_R]} \|\tilde{u}_\epsilon(s)\|_{L^2(D)}^p\Big]\Big\}\,. \label{est:r8}
\end{align}
Again, using \cite[Corollary $2.4$]{Liu-2019}, Young's inequality, the assumption \eqref{cond:linear-growth-ldp}, Lemma \eqref{lem:identity-eta} and the fact that $0\le {\tt h}_2 \le 1 $, one can easily deduce that
\begin{align}
\mathcal{R}_9 + \mathcal{R}_{10}\le C\,\epsilon\, \Big\{ 1 + \mathbb{E}\Big[ \sup_{s\in [0, t\wedge \bar{\tau}_R]} \|\tilde{u}_\epsilon(s)\|_{L^2(D)}^p\Big]\Big\}\,. \label{est:r9+r10}
\end{align}
Combining \eqref{est:r8}, \eqref{est:r9+r10} in \eqref{esti:4-apriori-for-cond-c1-intermidiate} and \eqref{esti:3-apriori-for-cond-c1-intermidiate}, 
we conclude the following: there exist $\epsilon_0 \in (0,1)$ and a constant $C=C(p,\theta, N,T, \|u_0\|_{L^2(D)})$ such that for all $0<\epsilon < \epsilon_0$,
$$ \mathbb{E}\Big[ \sup_{s\in [0, t\wedge \bar{\tau}_R]} \|\tilde{u}_\epsilon(s)\|_{L^2(D)}^p + \frac{p}{2} \int_0^{t\wedge \bar{\tau}_R} \|\tilde{u}_\epsilon(s)\|_{L^2(D)}^{p-2} \|\tilde{u}_\epsilon(s)\|_{W_0^{1,2}(D)}^2\,{\rm d}s \Big] \le C(p,\theta, N,T, \|u_0\|_{L^2(D)}).$$
This completes the proof.
\end{proof}

\subsubsection{Proof of condition \ref{LDP1}}
For any $N \in \mathbb{N}$, let $\{\varphi_\epsilon:~\epsilon \in (0, \epsilon_0)\} \subset \tilde{\mathcal{U}}_N$. For any $R, L>0$ and $\delta \in (0,1)$, we define
\begin{align*}
\tau_R^{\epsilon}  &:=\inf \left\{t>0:\|\tilde{u}_{\epsilon}(t)\|_{L^2(D)} >R\right\}, \quad 
\tau_L^{\epsilon} :=\inf \left\{t>0: \int_0^t\|\tilde{u}_{\epsilon}(s)\|_{W_0^{1,2}(D)}^2\,{\rm d}s>L\right\}, \\
\tau_{\delta}^{\epsilon} & :=\inf \{t>0:\|\tilde{u}_{\epsilon}(t)-u_{\varphi_\epsilon}(t)\|_{L^2(D)}>\delta\}, \quad 
\tau^{\epsilon}  :=\tau_R^{\epsilon} \wedge \tau_L^{\epsilon} \wedge \tau_{\delta}^{\epsilon}  \wedge T\,,
\end{align*}
where $\tilde{u}_\epsilon$ resp. $u_{\varphi_\epsilon}$ is the unique solution of \eqref{eq:epsilon-main} resp. \eqref{eq:log-skeleton} with $g$ replaced by $\varphi_\epsilon$. Note that, since $\varphi_\epsilon \in \widetilde{\mathcal{U}}_N$, by \eqref{esti:uni-skeleton} there exists a constant $C_N$, independent of $\epsilon$, such that $\mathbb{P}$-a.s., 
\begin{align}
\sup_{\epsilon \in (0,\epsilon_0)}\Bigg\{ \sup_{s\in [0,T]}\|u_{\varphi_\epsilon}(s)\|_{L^2(D)}^2 + \int_0^T\|u_{\varphi_\epsilon}(s)\|_{W_0^{1,2}(D)}^2\,{\rm d}s \Bigg\} \le C_N\,. \label{esti:apriori-for-cond-c1-control}
\end{align}
Observe that, in view of \eqref{esti:apriori-for-cond-c1} and Markov's inequality, one has for any fixed $T_* \le T$
\begin{equation}\label{conv:cond-c1-stoppingtimes}
\begin{aligned}
&\sup_{\epsilon \in (0,\epsilon_0)} \mathbb{P}\Big( \tau_R^{\epsilon} \le T_*\Big) \le \frac{1}{R^2} \sup_{\epsilon \in (0,\epsilon_0)} \mathbb{E}\Big[ \|\tilde{u}_\eps(T_* \wedge \tau_R^{\epsilon})\|_{L^2(D)}^2\Big]\le \frac{\tilde{C}_N}{R^2} \rightarrow 0~~\text{as $R\rightarrow \infty$} \\
&\sup_{\epsilon \in (0,\epsilon_0)} \mathbb{P}\Big( \tau_L^{\epsilon} \le T_*\Big)\le \frac{1}{L} \sup_{\epsilon \in (0,\epsilon_0)}\mathbb{E}\Big[ \int_0^T 
\|\tilde{u}_{\epsilon}(s)\|_{W_0^{1,2}(D)}^2\,{\rm d}s\Big] \le  \frac{\tilde{C}_N}{L} \rightarrow 0~~\text{as $L\rightarrow \infty$}\,.
\end{aligned}
\end{equation}
By applying It$\hat{o}$-L\'{e}vy formula to the functional $y\mapsto \|y\|_{L^2(D)}^2$ on $ \tilde{v}_\eps:=\tilde{u}_\epsilon - u_{\varphi_\epsilon}$, we get
\begin{align}
   & \|\tilde{v}_\epsilon(t\wedge\tau^{\epsilon})\|_{L^2(D)}^2 + 2 \int_{0}^{t\wedge\tau^{\epsilon}} \|\tilde{v}_{\epsilon}(s)\| _{W_0^{1,2}(D)}^2 \,{\rm d}s \notag \\
 & = 2 \int_{0}^{t\wedge\tau^{\epsilon}} \big\langle  \tilde{u}_{\epsilon}(s) \log |\tilde{u}_{\epsilon}(s)|- u_{\varphi_\epsilon}(s)\log |u_{\varphi_\epsilon}(s)|, \tilde{u}_{\epsilon}(s)-u_{\varphi_\epsilon}(s) \big\rangle \,{\rm d}s
 \notag \\
 & + 2 \int_{0}^{t\wedge\tau^{\epsilon}} \int_{\pmb{E}} \big\langle \eta(\tilde{u}_{\epsilon}(s);z) - \eta(u_{\varphi_\epsilon}(s);z) , \tilde{v}_{\epsilon}(s) \big\rangle (\varphi_{\epsilon}(s,z)-1) \, m({\rm d}z)\,{\rm d}s
 \notag \\
  & + 2\,\epsilon \int_{0}^{t\wedge\tau^{\epsilon}} \int_{\pmb{E}} \big\langle  \eta(\tilde{u}_{\epsilon}(s);z), \tilde{v}_{\epsilon}(s) \big\rangle  \big( N^{\epsilon^{-1}\,\varphi_{\epsilon}} ({\rm d}z,{\rm d}s) - \epsilon^{-1}\varphi_{\epsilon}\, m({\rm d}z)\,{\rm d}s \big)
 \notag \\
  & + \epsilon^2 \int_{0}^{t\wedge\tau^{\epsilon}} \int_{\pmb{E}} \|\eta(\tilde{u}_{\epsilon}(s);z)\|_{L^2(D)}^2 N^{\epsilon^{-1}\,\varphi_{\epsilon}} ({\rm d}z,{\rm d}s) \equiv \sum_{i=1}^{4}\mathcal{T}_{i}^{\epsilon}(t) \,. \label{eq:cond-c1-1}
\end{align}
We use Lemma \ref{lem:result-1}, uniform estimate \eqref{esti:apriori-for-cond-c1-control} and the definition of $\tau^{\epsilon}$ to get
\begin{align}
    \mathcal{T}_{1}^{\epsilon}(t) \le & \int_{0}^{t\wedge\tau^{\epsilon}} \|\tilde{v}_{\epsilon}(s)\| _{W_0^{1,2}(D)}^2 \,{\rm d}s + \frac{R^{2(1-\alpha)}}{(1-\alpha)e} \int_{0}^{t} \|\tilde{v}_\epsilon(s\wedge\tau^{\epsilon})\|_{L^2(D)}^{2\alpha} \,{\rm d}s 
     + C \int_{0}^{t} \|\tilde{v}_\epsilon(s\wedge\tau^{\epsilon})\|_{L^2(D)}^2\,{\rm d}s\,. \notag
     \end{align}
In view of Cauchy-Schwartz-inequality, one has
\begin{align}
     \mathcal{T}_{2}^{\epsilon}(t) \le &  C \int_{0}^{t\wedge\tau^{\epsilon}} \Psi_{1,\epsilon}(s)\|\tilde{v}_\epsilon(s)\|_{L^2(D)}^2{\rm d}s \le  C \int_{0}^{t} \Psi_{1,\epsilon}(s)\|\tilde{v}_\epsilon(s\wedge\tau^{\epsilon})\|_{L^2(D)}^2{\rm d}s\,, \notag
\end{align}
where $\Psi_{1,\epsilon}(s)$ is given by
$$ \Psi_{1,\eps}(s):= \int_{\pmb{E}} {\tt h}_1(z)|\varphi_{\epsilon}(s,z)-1|\,m({\rm d}z)\,.$$
Combining all the above inequalities in \eqref{eq:cond-c1-1}, and then using Lemma \ref{lem:nonlinear-gronwall}, we get, for all $t\le T$
\begin{align}
    \underset{s \in [0,t]} \sup \|\tilde{v}_\epsilon(s\wedge\tau^{\epsilon})\|_{L^2(D)}^2
    & \le \Bigg \{\left( \underset{s \in [0,T]} \sup |\mathcal{T}_{3}^{\epsilon}(s)| + \underset{s \in [0,T]} \sup |\mathcal{T}_{4}^{\epsilon}(s)|  \right )^{1-\alpha} 
    \times \exp \left ((1-\alpha) \int_{0}^{t} (1+ \Psi_{1,\epsilon}(s))\,{\rm d}s  \right) \notag \\
   & \qquad \qquad +  \frac{R^{2(1-\alpha)}}{e} \int_{0}^{t} \exp \left ((1-\alpha) \int_{s}^{t} (1+ \Psi_{1,\epsilon}(r))\,{\rm d}r \right) \,{\rm d}s \Bigg \}^{\frac{1}{1-\alpha}} \notag \\
        & \quad \le 2^{\frac{\alpha}{1-\alpha}} \Bigg \{\left( \underset{s \in [0,T]} \sup |\mathcal{T}_{3}^{\epsilon}(s)| + \underset{s \in [0,T]} \sup |\mathcal{T}_{4}^{\epsilon}(s)| \right )
    \times \exp \left ( \int_{0}^{t}(1+ \Psi_{1,\eps}(s))\,{\rm d}s  \right) \Bigg \} \notag \\
   & \qquad + 2^{\frac{\alpha}{1-\alpha}} \Bigg \{ \frac{R^{2(1-\alpha)}}{e} \int_{0}^{t} \exp \left ((1-\alpha) \int_{s}^{t} (1+ \Psi_{1,\epsilon}(r))\,{\rm d}r \right) \,{\rm d}s \Bigg \}^{\frac{1}{1-\alpha}} \notag \\
   & \le 2^{\frac{\alpha}{1-\alpha}} \left( \underset{s \in [0,T]} \sup |\mathcal{T}_{3}^{\epsilon}(s)| + \underset{s \in [0,T]} \sup |\mathcal{T}_{4}^{\epsilon}(s)| \right )\pmb{H}_1 + \pmb{H}_2\,, \label{inq:cond-c1-2}
\end{align}
where $\pmb{H}_1$ and $\pmb{H}_2$ are given by
\begin{align*}
\pmb{H}_1:&= \exp \left ( \int_{0}^{t}(1+ \Psi_{1,\eps}(s))\,{\rm d}s  \right)\,, \\
\pmb{H}_2:&= 2^{\frac{\alpha}{1-\alpha}} \Bigg \{ \frac{R^{2(1-\alpha)}}{e} \int_{0}^{t} \exp \left ((1-\alpha) \int_{s}^{t} (1+ \Psi_{1,\epsilon}(r))\,{\rm d}r \right) \,{\rm d}s \Bigg \}^{\frac{1}{1-\alpha}}\,.
\end{align*}
Since $\varphi_\epsilon \in \widetilde{\mathcal{U}}_N$, by Lemma \ref{lem:identity-eta}, we have $\mathbb{P}$-a.s.,
\begin{align}
& \pmb{H}_1 \le \exp\Big( t + C_{1,1}^N\Big), \quad 
 \pmb{H}_2 \le \Big( \frac{t2^\alpha}{e}\Big)^\frac{1}{1-\alpha} R^2 \exp\Big( T + C_{1,1}^N\Big)\,. \label{inq:cond-c1-3}
\end{align}
Taking expectation in both sides of \eqref{inq:cond-c1-2} and using \eqref{inq:cond-c1-3}, we get for $t\le T$
\begin{align}
    &\mathbb{E} \Big[  \underset{s \in [0,t]} \sup \|\tilde{v}_\epsilon(s\wedge\tau^{\epsilon})\|_{L^2(D)}^2 \Big] \notag \\
    & \quad \le 2^{\frac{\alpha}{1-\alpha}} \exp\Big( t + C_{1,1}^N\Big)\mathbb{E} \Big[ \underset{s \in [0,T]} \sup |\mathcal{T}_{3}^{\epsilon}(s)| + \underset{s \in [0,T]} \sup |\mathcal{T}_{4}^{\epsilon}(s)|  \Big] + \Big( \frac{t2^\alpha}{e}\Big)^\frac{1}{1-\alpha} R^2 \exp\Big( T + C_{1,1}^N\Big)\,. \label{inq:cond-c1-4}
\end{align}
Again, in view of Lemma \ref{lem:identity-eta}, we observe that
\begin{align}
\mathbb{E}\Big[\underset{s \in [0,T]} \sup |\mathcal{T}_{4}^{\epsilon}(s)| \Big] 
& \le 2\, \epsilon^2\, \mathbb{E}\Bigg[ \Big( 1 + \sup_{t\in [0,T]}\|\tilde{u}_\epsilon(t)\|_{L^2(D)}^2\Big) \int_0^T \int_{\pmb{E}} |{\tt h}_2(z)|^2 (\varphi_\epsilon +1)\,m({\rm d}z)\,{\rm d}s \Bigg] \notag \\
& \le 2\, \epsilon^2\, C_{0,2}^N \Big( 1 + \mathbb{E}\Big[ \sup_{t\in [0,T]}\|\tilde{u}_\epsilon(t)\|_{L^2(D)}^2\Big]\Big)\,. \label{inq:cond-c1-5-1}
\end{align}
Moreover, the BDG inequality, Cauchy-Schwartz and Young's inequalities implies that 
\begin{align}
\mathbb{E}\Big[\underset{s \in [0,T]} \sup |\mathcal{T}_{3}^{\epsilon}(s)| \Big] & \le 
C\, \epsilon\, \mathbb{E}\Big[ \Big( \int_0^T \int_{\pmb{E}} \|\eta(\tilde{u}_\epsilon(s);z)\|_{L^2(D)}^2 \|\tilde{v}_\epsilon(s)\|_{L^2(D)}^2 N^{\eps^{-1}\varphi_\epsilon}({\rm d}z,{\rm d}s)\Big)^\frac{1}{2}\Big] \notag \\
& \le C\, \epsilon^\frac{1}{2}\,\mathbb{E} \left [ \int_0^T \int_{\pmb{E}} |{\tt h}_2(z)|^2 \big( 1 + \|\tilde{u}_\epsilon(s)\|_{L^2(D)}^2\big) \varphi_\epsilon(s,z)\,m({\rm d}z)\,{\rm d}s \right ] \notag \\
&  \quad + \frac{1}{2} \epsilon^\frac{1}{2}\,\mathbb{E} \left [  \underset{s \in [0,t]} \sup \|\tilde{v}_\epsilon(s\wedge\tau^{\epsilon})\|_{L^2(D)}^2 \right ] \notag \\
& \le  \frac{1}{2}\, \epsilon^\frac{1}{2}\,\mathbb{E} \left [  \underset{s \in [0,t]} \sup \|\tilde{v}_\epsilon(s\wedge\tau^{\epsilon})\|_{L^2(D)}^2 \right ] + C\, \epsilon^\frac{1}{2}\, C_{0,2}^N \Bigg( 1 + \mathbb{E}\Big[ \|\tilde{u}_\epsilon(s)\|_{L^2(D)}^2\Big] \Bigg)\,, \label{inq:cond-c1-5}
\end{align}
where in the last inequality, we have used Lemma \ref{lem:identity-eta} as $\varphi_\epsilon \in \widetilde{\mathcal{U}}_N$. We use \eqref{inq:cond-c1-5-1} and \eqref{inq:cond-c1-5} in the inequality \eqref{inq:cond-c1-4} together with the uniform estimates \eqref{esti:apriori-for-cond-c1} and \eqref{esti:apriori-for-cond-c1-control} to obtain
\begin{align}
\mathbb{E}\Big[ \sup_{s\in [0,t]}\| \tilde{v}_\epsilon(s\wedge\tau^{\epsilon})\|_{L^2(D)}^2\Big] \le C(N, \alpha) \epsilon^\frac{1}{2} + 
\Big( \frac{t2^\alpha}{e}\Big)^\frac{1}{1-\alpha} R^2 \exp\Big( T + C_{1,1}^N\Big), \quad \forall~t\le T\,. \label{inq:cond-c1-6}
\end{align}
By using \eqref{inq:cond-c1-6} in \eqref{eq:cond-c1-1}, we have
\begin{align}
\mathbb{E}\Big[ \int_{0}^{t\wedge\tau^{\epsilon}} \|\tilde{v}_{\epsilon}(s)\| _{W_0^{1,2}(D)}^2 \,{\rm d}s\Big] \le
C_1(N, \alpha) \epsilon^\frac{1}{2}+ C(R, T, N) \Big( \frac{t2^\alpha}{e}\Big)^\frac{1}{1-\alpha} + CT \frac{R^{2(1-\alpha)}}{(1-\alpha)e}
, \quad \forall~t\le T\,. \notag
\end{align}
Thus, we have
\begin{align}
&  \underset{\epsilon \rightarrow 0}{\limsup}\Bigg\{ \mathbb{E}\Big[ \sup_{s\in [0,t]}\| \tilde{v}_\epsilon(s\wedge\tau^{\epsilon})\|_{L^2(D)}^2\Big]  + \mathbb{E}\Big[ \int_{0}^{t\wedge\tau^{\epsilon}} \|\tilde{v}_{\epsilon}(s)\| _{W_0^{1,2}(D)}^2 \,{\rm d}s\Big]\Bigg\} \notag \\
& \le  C(R, T, N) \Big( \frac{t2^\alpha}{e}\Big)^\frac{1}{1-\alpha} + CT \frac{R^{2(1-\alpha)}}{(1-\alpha)e}
, \quad \forall~t\le T\,. \label{inq:cond-c1-7}
\end{align}
Setting $T^*:=\frac{e^2}{16T} \wedge T$ and then letting $\alpha \rightarrow 1$ in \eqref{inq:cond-c1-7}, we obtain
\begin{align}
    \underset{\epsilon \rightarrow 0}{\lim} \Bigg\{ \mathbb{E} \left [ \|\tilde{v}_\epsilon(t\wedge\tau^{\epsilon})\|_{L^2(D)}^2 \right] + \mathbb{E} \Big[ \int_{0}^{t\wedge\tau^{\epsilon}} \|\tilde{v}_{\epsilon}(s)\| _{W_0^{1,2}(D)}^2 \,{\rm d}s  \Big]\Bigg\} = 0,  \quad \forall \, 0 \leq t \leq T^*\,. \label{inq:cond-c1-8}
\end{align}
By the definition of $\tau^{\epsilon}$ along with Chebyshev's inequality, \eqref{conv:cond-c1-stoppingtimes} and  \eqref{inq:cond-c1-8}, we see that 
\begin{align*}
& \mathbb{P}\Big( \rho_{T^*}^2(\tilde{u}_\eps, u_{\varphi_\epsilon})> 2 \delta^2\Big) \notag \\
& \le \mathbb{P} \Big[\underset{t \in [0,T^*]}{\sup} \|\tilde{v}_{\epsilon}(t)\|_{L^2(D)} > \delta  \Big] +  \mathbb{P} \Big[\int_{0}^{T^*} \|\tilde{v}_{\epsilon}(s)\| _{W_0^{1,2}(D)}^2 \,{\rm d}s > \delta^2 \Big] \notag \\
& \le \frac{1}{\delta^2} \Bigg\{  \mathbb{E} \left [ \|\tilde{v}_\epsilon(T^*\wedge\tau^{\epsilon})\|_{L^2(D)}^2 \right] + \mathbb{E} \Big[ \int_{0}^{T^*\wedge\tau^{\epsilon}} \|\tilde{v}_{\epsilon}(s)\| _{W_0^{1,2}(D)}^2 \,{\rm d}s  \Big]\Bigg\} \notag \\
& \qquad +2 \underset{\epsilon \in (0,\epsilon_0)}{\sup} \mathbb{P} \left ( \tau_R^{\epsilon} \le T^* \right)  + 2\underset{\epsilon \in (0,\epsilon_0)}{\sup} \mathbb{P} \left ( \tau_L^{\epsilon} \le T^* \right) \longrightarrow 0~~\text{as $\epsilon \rightarrow 0$}\,. 
\end{align*} 
In other words, the condition \ref{LDP1} holds true on the interval $[0,T^*]$. Again by considering the equations satisfied by $\tilde{u}_{\epsilon}$ and $u_{\varphi_\epsilon}$ on the interval $[T^*, T]$ with the initial values $\tilde{u}_{\epsilon}(T^*)$ and $u_{\varphi_\epsilon}(T^*)$ respectively, and noting the fact that $\|\tilde{u}_{\epsilon}(T^*)- u_{\varphi_\epsilon}(T^*)\| \overset{\mathbb{P}}{\rightarrow} 0$ as $\epsilon \rightarrow 0$, one can use the same argument as above to conclude that
\begin{align*}
    \mathbb{P} \Big(\rho_{T^*,2T^*\wedge T}(\tilde{u}_{\epsilon},u_{\varphi_\epsilon})^2 > 2\delta^2\Big) \rightarrow 0, \quad \text{as $ \epsilon \rightarrow 0$}\,.
\end{align*}
Similarly, for $n \ge 2$
\begin{align*}
    \mathbb{P}\Big(\rho_{(n-1)T^*,nT^*\wedge T}(\tilde{u}_{\epsilon},u_{\varphi_\epsilon})^2 > 2\delta^2 \Big) \rightarrow 0, \quad \text{as $ \epsilon \rightarrow 0$}\,.
\end{align*}
Since there exists some $n >0$ such that $nT^* \ge T$, we arrive at assertion that
\begin{align*}
    \mathbb{P} \Big(\rho_T(\tilde{u}_{\epsilon},u_{\varphi_\epsilon})^2 > \delta\Big) \rightarrow 0 , \quad \text{as $ \epsilon \rightarrow 0$}\,.
\end{align*}
This completes the proof of condition \ref{LDP1}.

\section{Declarations} The authors would like to make the following declaration statements.

\begin{itemize}
\item {\bf Funding:} The first author acknowledges the financial support by CSIR~(09/086(1440)/2020-EMR-I), India.
\vspace{0.2 cm}

\item{\bf Ethical Approval:} This declaration is ``not applicable".
\vspace{0.2 cm}

\item {\bf Availability of data and materials:}  No data sets were generated during the current study and therefore data sharing is not applicable to this article. 
\vspace{0.2 cm} 

\item{\bf Conflict of interest:} The authors have not disclosed any competing interests.
\end{itemize}

\bibliographystyle{alphaabbr}
\bibliography{Refernce}

\newcommand{\etalchar}[1]{$^{#1}$}
\begin{thebibliography}{WHL{\etalchar{+}}12}

\bibitem[AC17]{Alfaro-2017}
M.~Alfaro and R.~Carles.
\newblock Superexponential growth or decay in the heat equation with a
  logarithmic nonlinearity.
\newblock {\em Dyn. Partial Differ. Equ.}, 14(4):343--358, 2017.

\bibitem[BBM76]{Bialynicki-1976}
I.~Bialynicki-Birula and J.~Mycielski.
\newblock Nonlinear wave mechanics.
\newblock {\em Annals of Physics}, 100(1-2):62--93, 1976.

\bibitem[BCD13]{Budhiraja-2013}
A.~Budhiraja, J.~Chen, and P.~Dupuis.
\newblock Large deviations for stochastic partial differential equations driven
  by a {P}oisson random measure.
\newblock {\em Stochastic Process. Appl.}, 123(2):523--560, 2013.

\bibitem[BD00]{Budhiraja-2000}
A.~Budhiraja and P.~Dupuis.
\newblock A variational representation for positive functionals of infinite
  dimensional {B}rownian motion.
\newblock {\em Probab. Math. Statist.}, 20(1, Acta Univ. Wratislav. No.
  2246):39--61, 2000.

\bibitem[BDM08]{Budhiraja-2008}
A.~Budhiraja, P.~Dupuis, and V.~Maroulas.
\newblock Large deviations for infinite dimensional stochastic dynamical
  systems.
\newblock {\em Ann. Probab.}, 36(4):1390--1420, 2008.

\bibitem[BDM11]{Budhiraja-2011}
A.~Budhiraja, P.~Dupuis, and V.~Maroulas.
\newblock Variational representations for continuous time processes.
\newblock {\em Ann. Inst. Henri Poincar\'{e} Probab. Stat.}, 47(3):725--747,
  2011.

\bibitem[BGJ17]{Goldys-2017}
Z.~a. Brze\'zniak, B.~Goldys, and T.~Jegaraj.
\newblock Large deviations and transitions between equilibria for stochastic
  {L}andau-{L}ifshitz-{G}ilbert equation.
\newblock {\em Arch. Ration. Mech. Anal.}, 226(2):497--558, 2017.

\bibitem[BH09]{Erika-2009}
Z.~a. Brze\'{z}niak and E.~Hausenblas.
\newblock Maximal regularity for stochastic convolutions driven by {L}\'{e}vy
  processes.
\newblock {\em Probab. Theory Related Fields}, 145(3-4):615--637, 2009.

\bibitem[Bil99]{Billingsley-99}
P.~Billingsley.
\newblock {\em Convergence of probability measures}.
\newblock Wiley Series in Probability and Statistics: Probability and
  Statistics. John Wiley \& Sons, Inc., New York, second edition, 1999.
\newblock A Wiley-Interscience Publication.

\bibitem[BM13]{Brzezniak-2013}
Z.~a. Brze\'{z}niak and E.~Motyl.
\newblock Existence of a martingale solution of the stochastic
  {N}avier-{S}tokes equations in unbounded 2{D} and 3{D} domains.
\newblock {\em J. Differential Equations}, 254(4):1627--1685, 2013.

\bibitem[BPZ23]{Brzezniak-2023}
Z.~a. Brze\'zniak, X.~Peng, and J.~Zhai.
\newblock Well-posedness and large deviations for 2{D} stochastic
  {N}avier-{S}tokes equations with jumps.
\newblock {\em J. Eur. Math. Soc. (JEMS)}, 25(8):3093--3176, 2023.

\bibitem[CLL15]{Chen-2015}
H.~Chen, P.~Luo, and G.~Liu.
\newblock Global solution and blow-up of a semilinear heat equation with
  logarithmic nonlinearity.
\newblock {\em J. Math. Anal. Appl.}, 422(1):84--98, 2015.

\bibitem[CR04]{Rockner-2004}
S.~Cerrai and M.~R\"ockner.
\newblock Large deviations for stochastic reaction-diffusion systems with
  multiplicative noise and non-{L}ipschitz reaction term.
\newblock {\em Ann. Probab.}, 32(1B):1100--1139, 2004.

\bibitem[CT15]{Chen-2015-JD}
H.~Chen and S.~Tian.
\newblock Initial boundary value problem for a class of semilinear
  pseudo-parabolic equations with logarithmic nonlinearity.
\newblock {\em J. Differential Equations}, 258(12):4424--4442, 2015.

\bibitem[DKZ19]{Dalang-2019}
R.~C. Dalang, D.~Khoshnevisan, and T.~Zhang.
\newblock Global solutions to stochastic reaction-diffusion equations with
  super-linear drift and multiplicative noise.
\newblock {\em Ann. Probab.}, 47(1):519--559, 2019.

\bibitem[Don18]{Yuchao-2018}
Y.~Dong.
\newblock Jump stochastic differential equations with non-{L}ipschitz and
  superlinearly growing coefficients.
\newblock {\em Stochastics}, 90(5):782--806, 2018.

\bibitem[DPZ14]{Zabczyk-2014}
G.~Da~Prato and J.~Zabczyk.
\newblock {\em Stochastic equations in infinite dimensions}, volume 152 of {\em
  Encyclopedia of Mathematics and its Applications}.
\newblock Cambridge University Press, Cambridge, second edition, 2014.

\bibitem[DS89]{Stroock-1989}
J.-D. Deuschel and D.~W. Stroock.
\newblock {\em Large deviations}, volume 137 of {\em Pure and Applied
  Mathematics}.
\newblock Academic Press, Inc., Boston, MA, 1989.

\bibitem[EK05]{Edelstein-2005}
L.~Edelstein-Keshet.
\newblock {\em Mathematical models in biology}.
\newblock SIAM, 2005.

\bibitem[FG95]{Flandoli-1995}
F.~Flandoli and D.~Gatarek.
\newblock Martingale and stationary solutions for stochastic {N}avier-{S}tokes
  equations.
\newblock {\em Probab. Theory Related Fields}, 102(3):367--391, 1995.

\bibitem[FG23]{Gess-2023}
B.~Fehrman and B.~Gess.
\newblock Non-equilibrium large deviations and parabolic-hyperbolic {PDE} with
  irregular drift.
\newblock {\em Invent. Math.}, 234(2):573--636, 2023.

\bibitem[FZ05]{Shizan-2005}
S.~Fang and T.~Zhang.
\newblock A study of a class of stochastic differential equations with
  non-{L}ipschitzian coefficients.
\newblock {\em Probab. Theory Related Fields}, 132(3):356--390, 2005.

\bibitem[GK82]{Krylov-1982}
I.~Gy\"ongy and N.~V. Krylov.
\newblock On stochastics equations with respect to semimartingales. {II}.
  {I}t\^o{} formula in {B}anach spaces.
\newblock {\em Stochastics}, 6(3-4):153--173, 1981/82.

\bibitem[HLL21]{Hong-2021}
W.~Hong, S.~Li, and W.~Liu.
\newblock Freidlin-{W}entzell type large deviation principle for multiscale
  locally monotone {SPDE}s.
\newblock {\em SIAM J. Math. Anal.}, 53(6):6517--6561, 2021.

\bibitem[IW81]{Watanabe-1981}
N.~Ikeda and S.~Watanabe.
\newblock {\em Stochastic differential equations and diffusion processes},
  volume~24 of {\em North-Holland Mathematical Library}.
\newblock North-Holland Publishing Co., Amsterdam-New York; Kodansha, Ltd.,
  Tokyo, 1981.

\bibitem[JS87]{Jacod-1987}
J.~Jacod and A.~N. Shiryaev.
\newblock {\em Limit theorems for stochastic processes}, volume 288 of {\em
  Grundlehren der mathematischen Wissenschaften [Fundamental Principles of
  Mathematical Sciences]}.
\newblock Springer-Verlag, Berlin, 1987.

\bibitem[JYC16]{Yang-2016}
S.~Ji, J.~Yin, and Y.~Cao.
\newblock Instability of positive periodic solutions for semilinear
  pseudo-parabolic equations with logarithmic nonlinearity.
\newblock {\em J. Differential Equations}, 261(10):5446--5464, 2016.

\bibitem[KAV23]{Kavin-2023LDP}
R.~KAVIN.
\newblock Large deviation principle for pseudo-monotone evolutionary equation.
\newblock 2023.

\bibitem[KM24]{Kavin-2024}
R.~Kavin and A.~K. Majee.
\newblock Stochastic evolutionary {$p$}-{L}aplace equation: {L}arge deviation
  principles and transportation cost inequality.
\newblock {\em J. Math. Anal. Appl.}, 536(1):Paper No. 128163, 2024.

\bibitem[Liu10]{Liu-2010}
W.~Liu.
\newblock Large deviations for stochastic evolution equations with small
  multiplicative noise.
\newblock {\em Appl. Math. Optim.}, 61(1):27--56, 2010.

\bibitem[LLX20]{Li-2020}
S.~Li, W.~Liu, and Y.~Xie.
\newblock Small time asymptotics for {SPDE}s with locally monotone
  coefficients.
\newblock {\em Discrete Contin. Dyn. Syst. Ser. B}, 25(12):4801--4822, 2020.

\bibitem[LSZZ23]{Zhang-2023}
W.~Liu, Y.~Song, J.~Zhai, and T.~Zhang.
\newblock Large and moderate deviation principles for {M}c{K}ean-{V}lasov
  {SDE}s with jumps.
\newblock {\em Potential Anal.}, 59(3):1141--1190, 2023.

\bibitem[M{\etalchar{+}}23]{Kavin-2023nonlinear}
A.~K. Majee et~al.
\newblock Nonlinear spde driven by levy noise: Well-posedness, optimal control
  and invariant measure.
\newblock {\em arXiv preprint arXiv:2306.04303}, 2023.

\bibitem[Maj20]{Majee-2020}
A.~K. Majee.
\newblock Stochastic optimal control of a evolutionary {$p$}-{L}aplace equation
  with multiplicative {L}\'{e}vy noise.
\newblock {\em ESAIM Control Optim. Calc. Var.}, 26:Paper No. 100, 22, 2020.

\bibitem[Maj23]{Majee-2023}
A.~K. Majee.
\newblock Stochastic optimal control of a doubly nonlinear {PDE} driven by
  multiplicative {L}\'{e}vy noise.
\newblock {\em Appl. Math. Optim.}, 87(1):Paper No. 7, 39, 2023.

\bibitem[M{\'e}t88]{Metivier-1988}
M.~M{\'e}tivier.
\newblock {\em Stochastic partial differential equations in infinite
  dimensional spaces}.
\newblock Springer, 1988.

\bibitem[Mot13]{Motyl-2013}
E.~Motyl.
\newblock Stochastic {N}avier-{S}tokes equations driven by {L}\'{e}vy noise in
  unbounded 3{D} domains.
\newblock {\em Potential Anal.}, 38(3):863--912, 2013.

\bibitem[MPF91]{Mitrinovic-1991}
D.~S. Mitrinovi\'{c}, J.~E. Pe\v{c}ari\'{c}, and A.~M. Fink.
\newblock {\em Inequalities involving functions and their integrals and
  derivatives}, volume~53 of {\em Mathematics and its Applications (East
  European Series)}.
\newblock Kluwer Academic Publishers Group, Dordrecht, 1991.

\bibitem[MSZ21]{Matoussi-2021}
A.~Matoussi, W.~Sabbagh, and T.~Zhang.
\newblock Large deviation principles of obstacle problems for quasilinear
  stochastic {PDE}s.
\newblock {\em Appl. Math. Optim.}, 83(2):849--879, 2021.

\bibitem[MZ21]{Jin-2021}
J.~Ma and R.-r. Zhang.
\newblock Large deviations for 2{D} primitive equations driven by
  multiplicative {L}\'{e}vy noises.
\newblock {\em Acta Math. Appl. Sin. Engl. Ser.}, 37(4):773--799, 2021.

\bibitem[Ond04]{Ondrejat-2004}
M.~Ondrej\'{a}t.
\newblock Uniqueness for stochastic evolution equations in {B}anach spaces.
\newblock {\em Dissertationes Math. (Rozprawy Mat.)}, 426:63, 2004.

\bibitem[Par05]{Parthasarathy1967}
K.~R. Parthasarathy.
\newblock {\em Probability measures on metric spaces}, volume 352.
\newblock American Mathematical Soc., 2005.

\bibitem[PR07]{Michael-2007}
C.~Pr\'{e}v\^{o}t and M.~R\"{o}ckner.
\newblock {\em A concise course on stochastic partial differential equations},
  volume 1905 of {\em Lecture Notes in Mathematics}.
\newblock Springer, Berlin, 2007.

\bibitem[PSZ22]{Pan-2022}
T.~Pan, S.~Shang, and T.~Zhang.
\newblock Large deviations of stochastic heat equations with logarithmic
  nonlinearity.
\newblock {\em arXiv preprint arXiv:2207.02385}, 2022.

\bibitem[PZ07]{Peszat-2007}
S.~Peszat and J.~Zabczyk.
\newblock {\em Stochastic partial differential equations with {L}\'{e}vy
  noise}, volume 113 of {\em Encyclopedia of Mathematics and its Applications}.
\newblock Cambridge University Press, Cambridge, 2007.
\newblock An evolution equation approach.

\bibitem[Ros69]{Rosen-1969}
G.~Rosen.
\newblock Dilatation covariance and exact solutions in local relativistic field
  theories.
\newblock {\em Physical Review}, 183(5):1186, 1969.

\bibitem[RZ12]{Rockner-2012}
M.~R\"ockner and T.~Zhang.
\newblock Stochastic 3{D} tamed {N}avier-{S}tokes equations: existence,
  uniqueness and small time large deviation principles.
\newblock {\em J. Differential Equations}, 252(1):716--744, 2012.

\bibitem[Sit05]{Situ-2005}
R.~Situ.
\newblock {\em Theory of stochastic differential equations with jumps and
  applications}.
\newblock Mathematical and Analytical Techniques with Applications to
  Engineering. Springer, New York, 2005.
\newblock Mathematical and analytical techniques with applications to
  engineering.

\bibitem[Sko56]{Skorohod-56}
A.~V. Skorohod.
\newblock Limit theorems for stochastic processes.
\newblock {\em Teor. Veroyatnost. i Primenen.}, 1:289--319, 1956.

\bibitem[Str84]{Stroock-1984}
D.~W. Stroock.
\newblock {\em An introduction to the theory of large deviations}.
\newblock Universitext. Springer-Verlag, New York, 1984.

\bibitem[SZ22]{Shang-2022}
S.~Shang and T.~Zhang.
\newblock Stochastic heat equations with logarithmic nonlinearity.
\newblock {\em J. Differential Equations}, 313:85--121, 2022.

\bibitem[Var66]{Varadhan-1996}
S.~R.~S. Varadhan.
\newblock Asymptotic probabilities and differential equations.
\newblock {\em Comm. Pure Appl. Math.}, 19:261--286, 1966.

\bibitem[Var84]{Varadhan-1984}
S.~R.~S. Varadhan.
\newblock {\em Large deviations and applications}, volume~46 of {\em CBMS-NSF
  Regional Conference Series in Applied Mathematics}.
\newblock Society for Industrial and Applied Mathematics (SIAM), Philadelphia,
  PA, 1984.

\bibitem[vdVW23]{Wellner}
A.~W. van~der Vaart and J.~A. Wellner.
\newblock {\em Weak convergence and empirical processes---with applications to
  statistics}.
\newblock Springer Series in Statistics. Springer, Cham, second edition, [2023]
  \copyright 2023.

\bibitem[WHL{\etalchar{+}}12]{LeiWang-2012}
L.~Wang, B.~Hu, B.~Li, et~al.
\newblock Logarithmic divergent thermal conductivity in two-dimensional
  nonlinear lattices.
\newblock {\em Physical Review E}, 86(4):040101, 2012.

\bibitem[WRD12]{Wang-2012}
W.~Wang, A.~J. Roberts, and J.~Duan.
\newblock Large deviations and approximations for slow-fast stochastic
  reaction-diffusion equations.
\newblock {\em J. Differential Equations}, 253(12):3501--3522, 2012.

\bibitem[WWJ2r]{Wu-2024-1}
Z.~J. Wu~Weina and Z.~Jiahui.
\newblock Large deviations for locally monotone spdes driven by {L}\'evy noise.
\newblock {\em arXiv preprint arXiv:2401.11385v1}, 202r.

\bibitem[WZ23]{Ran-2023}
R.~Wang and B.~Zhang.
\newblock A large deviation principle for the stochastic generalized
  {G}inzburg-{L}andau equation driven by jump noise.
\newblock {\em Acta Math. Sci. Ser. B (Engl. Ed.)}, 43(2):505--530, 2023.

\bibitem[WZ24]{Wu-2024}
W.~Wu and J.~Zhai.
\newblock Large deviations for stochastic generalized porous media equations
  driven by {L}\'evy noise.
\newblock {\em SIAM J. Math. Anal.}, 56(1):1--42, 2024.

\bibitem[XZ18]{Xiong-2018}
J.~Xiong and J.~Zhai.
\newblock Large deviations for locally monotone stochastic partial differential
  equations driven by {L}\'{e}vy noise.
\newblock {\em Bernoulli}, 24(4A):2842--2874, 2018.

\bibitem[XZ19]{Fubao-2019}
F.~Xi and C.~Zhu.
\newblock Jump type stochastic differential equations with non-{L}ipschitz
  coefficients: non-confluence, {F}eller and strong {F}eller properties, and
  exponential ergodicity.
\newblock {\em J. Differential Equations}, 266(8):4668--4711, 2019.

\bibitem[YZZ15]{Yang-2015}
X.~Yang, J.~Zhai, and T.~Zhang.
\newblock Large deviations for {SPDE}s of jump type.
\newblock {\em Stoch. Dyn.}, 15(4):1550026, 30, 2015.

\bibitem[ZBL19]{Liu-2019}
J.~Zhu, Z.~a. Brze\'{z}niak, and W.~Liu.
\newblock Maximal inequalities and exponential estimates for stochastic
  convolutions driven by {L}\'{e}vy-type processes in {B}anach spaces with
  application to stochastic quasi-geostrophic equations.
\newblock {\em SIAM J. Math. Anal.}, 51(3):2121--2167, 2019.

\bibitem[Zha14]{Zhao-2014}
H.~Zhao.
\newblock Yamada–watanabe theorem for stochastic evolution equation driven by
  poisson random measure.
\newblock {\em ISRN Probab. Statist.}, pages art. 982190, 7 pp., 2014.

\bibitem[Zlo19]{Zloshchastiev-2019}
K.~G. Zloshchastiev.
\newblock Temperature-driven dynamics of quantum liquids: Logarithmic
  nonlinearity, phase structure and rising force.
\newblock {\em International Journal of Modern Physics B}, 33(17):1950184,
  2019.

\bibitem[ZZ15]{Zhai-2015}
J.~Zhai and T.~Zhang.
\newblock Large deviations for 2-{D} stochastic {N}avier-{S}tokes equations
  driven by multiplicative {L}\'{e}vy noises.
\newblock {\em Bernoulli}, 21(4):2351--2392, 2015.

\end{thebibliography}


\begin{thebibliography}{99}

\bibitem{Alfaro} M. Alfaro, R. Carles.
\newblock Superexponential growth or decay in the heat equation with a logarithmic nonlinearity.
\newblock{\em Dyn. Partial Differ. Equ.} 14 (4), 343–358 (2017).

\bibitem{Mycielski} ] I. Białynicki-Birula, J. Mycielski.
\newblock Nonlinear wave mechanics.
\newblock{\em Ann. Phys.} 100 (1–2), 62–93 (1976).

\bibitem{erika2009}
Z. Brze\'{z}niak and E. Hausenblas.
\newblock Maximal regularity for stochastic convolutions driven by L\'{e}vy processes.
\newblock{\em Probab. Theory Relat. Fields}~(2009) 145:615-637.

\bibitem{Chen}  H. Chen, P. Luo, G. Liu.
\newblock Global solution and blow-up of a semilinear heat equation with logarithmic nonlinearity.
\newblock{\em J. Math. Anal. Appl.} 422 (1), 84–98 (2015).

\bibitem{Dalang}  R.C. Dalang, D. Khoshnevisan, T. Zhang.
\newblock Global solutions to stochastic reaction-diffusion equations with superlinear drift and multiplicative noise.
\newblock{\em Ann. Probab.} 47 (1), 519–559 (2019).

\bibitem{Dong}  Y. Dong. 
\newblock Jump stochastic differential equations with non-Lipschitz and superlinearly growing coefficients.
\newblock{\em Stochastics.} 90 (5), 782–806 (2018).

\bibitem{Motyl2013}
 E. Motyl.
 \newblock Stochastic Navier-Stokes Equations Driven by L\'{e}vy Noise in Unbounded 3D Domains.
 \newblock {\em Potential Anal.} (2013) 38:863-912.

 \bibitem{Ondrejat2004}
M.  Ondrej{\'a}t.
Uniqueness for stochastic evolution equations in Banach spaces.
\newblock{\em  Diss. Math.} 426, 1-63 (2004).

\bibitem{Parthasarathy1967}
 K. R. Parthasarathy.
 \newblock Probability measures on metric spaces.
 \newblock {\em Academic Press}, New York and London, 1967.

\bibitem{peszat}
S.~Peszat and J.~Zabczyk.
\newblock {\em Stochastic partial differential equations with
{L}{\'e}vy noise}, volume 113 of {\em Encyclopedia of Mathematics and its Applications}.
\newblock Cambridge University Press, Cambridge, 2007.
\newblock An evolution equation approach.

\bibitem{Rosen} G. Rosen.
\newblock Dilatation covariance and exact solutions in local relativistic field theories.
\newblock{ \em Phys. Rev.} 183 (5), 1186–1188 (1968).

\bibitem{Shang-2022} S.Shang, T.Zhang.
\newblock Stochastic heat equations with logarithmic nonlinearity.
\newblock{\em J. Differ. Equ.} 313, 85-21 (2022).

\bibitem{Xi-2019}  F. Xi, C. Zhu.
\newblock Jump type stochastic differential equations with nonLipschitz coefficients: Non-confluence, feller and strong feller properties, and exponential ergodicity. 
\newblock{\em J. Differ. Equ.} 266, 4668–4711 (2019).
\end{thebibliography}

\end{document}